\documentclass[a4paper,10pt]{article}

\usepackage{amsmath}
\usepackage{amsfonts}
\usepackage{amssymb}
\usepackage{mathrsfs}
\usepackage{amscd}
\usepackage{pb-diagram}
\usepackage{amsthm}
\usepackage{color}
\usepackage[all]{xy}%pour les diagrams
\usepackage{graphicx}
\usepackage{url}

\author{Mathieu Molitor
\thanks{Present address: \it{Fakult\"{a}t f\"{u}r Mathematik, Ruhr-Universit\"{a}t Bochum, Germany}}\\
\it \small{Department of Mathematics, Keio University}\\
\it \small{3-14-1, Hiyoshi, Kohoku-ku, 223-8522, Yokohama, Japan}\\ 
\small{\it{e-mail:}}\,\,\url{pergame.mathieu@gmail.com}
}
\title{Exponential families, K\"{a}hler geometry and quantum mechanics}
\date{}

\begin{document}
\newtheorem{lemma}{Lemma}[section]
\newtheorem{definition}[lemma]{Definition}
\newtheorem{proposition}[lemma]{Proposition}
\newtheorem{corollary}[lemma]{Corollary}
\newtheorem{theorem}[lemma]{Theorem}
\newtheorem{remark}[lemma]{Remark}
\newtheorem{example}[lemma]{Example}
\bibliographystyle{alpha}

\maketitle 

\begin{abstract}
	Exponential families are a particular class of statistical manifolds which are particularly important 
	in statistical inference, and which appear very frequently in statistics. For example, the set of normal distributions,
	with mean $\mu$ and deviation $\sigma\,,$ form a 2-dimensional exponential family. 

	In this paper, we show that the tangent bundle of an exponential family is naturally a K\"{a}hler manifold. 
	This simple but crucial observation leads to the formalism of quantum mechanics in its geometrical 
	form, i.e. based on the K\"{a}hler structure of the complex projective space, but generalizes also to more general 
	K\"{a}hler manifolds, providing a natural geometric framework for the description of quantum systems. 
%	suited for the description of quantum systems. 
%	mechanical situations which encompasses the usual geometrical formulation. 
%	but generalizes also to other K\"{a}hler 
%	manifolds of quantum mechanical relevance. 
%	More important, but generalizes also to other 
%	K\"{a}hler manifolds as a framework for natural quantum mechanical situations. 

	Many questions related to this ``statistical K\"{a}hler geometry" are discussed, and a close connection 
	with representation theory is observed. 

	Examples of physical relevance are treated in details. 
	For example, it is shown that the spin of a particle can be entirely understood by means of 
	the usual binomial distribution. 

	This paper centers on the mathematical foundations of quantum mechanics, and on the question 
	of its potential generalization through its geometrical formulation. 
\end{abstract}

%\tableofcontents

\section{Introduction -- summary}
	
	In the 70's, it has been observed by Chernoff and 
	Marsden \cite{Chernoff-Marsden} that the Schr\"{o}dinger 
	equation $i\hbar \frac{d\psi}{dt}=H\psi$ is Hamiltonian with respect to the 
	symplectic form coming from the imaginary part of the Hermitian scalar product of the Hilbert space 
	$\mathcal{H}$ of possible quantum states. Since then, this Hamiltonian view on quantum mechanics 
	has been developed independently by several 
	authors \cite{Cirelli-Hamiltonian,Cirelli-Quantum,Heslot-Une,Heslot-Quantum,Kibble} and has led to 
	a complete geometrization of the quantum formalism, entirely based on the K\"{a}hler 
	properties of the complex projective space $\mathbb{P}(\mathcal{H})\,.$ This reformulation, which is 
	very elegant and complete, is now usually 
	referred to as the \textit{geometrical formulation}\footnote{Not to be 
	confused with the \textit{geometric quantization} of Kostant and Souriau \cite{Kostant,Souriau}. In the 
	geometrical formulation, the Hilbert space is considered as given, not as the result of a quantization scheme.} 
	of quantum mechanics \cite{Ashtekar,Cirelli-Quantum}.

	The geometrical formulation was mainly motivated by the desire to generalize quantum mechanics, 
	especially in view of quantum gravity. The basic idea is that, by geometrizing the quantum formalism, one frees it 
	from its burdensome linearity and put it on a geometrical ground akin to Einstein's theory of gravitation. Geometry is, 
	in this regard, particularly ``flexible", and seems an appropriate setting for generalizations. 
	For example, while it is not clear how to generalize Hilbert spaces, 
	generalizations of the complex projective space is straightforward: instead of $\mathbb{P}(\mathcal{H})\,,$ 
	take an arbitrary K\"{a}hler manifold. Such possibilities have been discussed in 
	\cite{Gibbons-Typical,Hughston,Kawamura,Kibble} and applications 
	towards quantum gravity have been proposed in \cite{Tze-Background,Tze-General}. 
	
	These proposals, however, are limited in their original scope by a severe limitation: 
	the need to allow for a \textit{probabilistic interpretation}. The latter, contrary to the 
	quantum state space and dynamics, is extremely difficult to generalize in a purely geometrical context and is 
	usually not even discussed. The reason is that in the geometrical formulation, all formulas related to probabilities 
	rely on an expression of the form $\cos^{2}\big(d(\,,\,)\big)\,,$ where $d(\,,\,)$ is the geodesic 
	distance on $\mathbb{P}(\mathcal{H})\,,$ expression which is clearly specific to $\mathbb{P}(\mathcal{H})$ 
	and which, consequently, cannot be generalized directly to arbitrary K\"{a}hler manifolds. 

	These difficulties --related to the probabilistic interpretation-- address the following question: 
	\textit{what is the link between the K\"{a}hler structure of 
	$\mathbb{P}(\mathbb{C}^{n})$ and probabilities?} This question, which is central in the present work, 
	was already formulated in our previous paper \cite{Molitor-quantique} where an interesting, though 
	puzzling connection with \textit{information geometry} has been observed.

	Let us recall, in this regard, that information geometry is a branch of statistics characterized by its use of differential 
	geometrical techniques \cite{Amari-Nagaoka,Murray-Rice}. Its basic objects of study are \textit{statistical manifolds}, i.e. 
	manifolds whose points can be identified with probability density functions over some fixed measured space. 
	For example, Gaussian distributions over $\mathbb{R}$ form a 2-dimensional statistical manifold parameterized by the mean 
	$\mu$ and deviation $\sigma\,.$ In general --and this is what information geometry is about-- a statistical manifold $S$ 
	possesses a rich geometry that encodes many of its statistical properties; it has a Riemannian metric 
	$h_{F}\,,$ called \textit{Fisher metric}, and a pair of dual affine connections $\nabla^{(e)}\,,\nabla^{(m)}\,,$ 
	respectively called \textit{exponential connection} and 
	\textit{mixture connection}, which can be used, for example, to give lower bounds in estimation problems (compare 
	e.g. the Cram\'{e}r-Rao inequality). Together, the triplet $(h_{F},\nabla^{(e)},\nabla^{(m)})$ forms what is 
	called a \textit{dualistic structure}, and it is probably the most important structure in information geometry. 

	Very little attention has been paid to dualistic structures outside the statistical community, but we can mention, 
	in connection with \cite{Molitor-quantique}, the work of Dombrowski\footnote{See also \cite{Nomizu-Simon}.}. 
	In a paper which already goes back to the 
	60's \cite{Dombrowski}, Dombrowski shows that if a manifold $M$ is endowed with a dualistic structure 
	($M$ needs not be a statistical manifold here), then
	its tangent bundle $TM$ becomes naturally, via a simple geometric construction, an almost Hermitian 
	manifold\footnote{In \cite{Dombrowski}, Dombrowski is not explicitly using the language of dualistic structures, and 
	his main concern is on the analytical properties of the almost complex structure that he constructs on the tangent 
	bundle $TM$ of a Riemannian manifold $(M,g)$ endowed with a connection $\nabla\,.$}.
	A direct consequence of Dombrowski's construction, which seems to have been 
	unnoticed in the existing literature, is that the tangent bundle of a statistical manifold is 
	canonically an almost Hermitian manifold. 

	This observation, although mathematically very simple, is one of the most important of \cite{Molitor-quantique}. 
	It tells us that \textit{statistics abounds with almost Hermitian manifolds}. 
	To illustrate this, let us consider what is probably the 
	most simple example that one may think of. Take a finite set 
	$\Omega:=\{x_{1},...,x_{n}\}$ and consider the space 
	$\mathcal{P}_{n}^{\times}$ of nowhere vanishing\footnote{The condition $p>0$ (instead of $p\geq 0$) is purely technical and 
	ensures that $\mathcal{P}_{n}^{\times}$ has no boundary nor corners.} 
	probabilities $p\,:\,\Omega\rightarrow \mathbb{R}\,,$ $p>0\,,$ 
	$\sum_{k=1}^{n}\,p(x_{k})=1\,.$ This is a $(n{-}1)$-dimensional statistical manifold, therefore its tangent bundle 
	$T\mathcal{P}_{n}^{\times}$ is an almost Hermitian manifold. Now the main observation in
	\cite{Molitor-quantique} may be formulated as follows:
	\textit{the canonical almost Hermitian structure of $T\mathcal{P}_{n}^{\times}$ is locally isomorphic to 
	the K\"{a}hler structure of $\mathbb{P}(\mathbb{C}^{n})\,.$}

	This result is intriguing. On one hand, it establishes a link 
	between the geometrical formulation of quantum mechanics 
	and information geometry, and suggests a possible information-theoretical origin of the quantum formalism.
	But on the other hand, the statistical relevance of Dombrowski's construction
	is not at all clear, and since $\mathcal{P}_{n}^{\times}$ is the only example in \cite{Molitor-quantique} for 
	which explicit computations are performed, their are a priori no reasons for other statistical manifolds to yield interesting 
	geometrical results of physical importance.

	As such, the results in \cite{Molitor-quantique} are potentially fruitful, 
	but they raise many questions that need further investigations.  \\
	
	In the present paper, we developed some 
	of the ideas of \cite{Molitor-quantique} and 
	present mathematical results --at the crossroad of information geometry, K\"{a}hler 
	geometry, functional analysis and, to some extent, representation theory-- 
	which reinforce the idea that the quantum formalism 
	has a statistical and information-theoretical origin. In particular, we want to describe a mechanism by which 
	K\"{a}hler geometry emerges from information geometry, and to explain, by revisiting the geometrical formulation, 
	how the latter, in its various aspects, can be understood within this larger information-theoretical setting. 

	A central role, in this development, is played by the so-called \textit{exponential families}. 
	Exponential families are a particular class of probability distributions which plays a key role 
	in various ramifications of statistics, and especially in the 
	context of statistical inference (see for example \cite{Amari-Nagaoka}). Their 
	importance stems from the fact that among all possible parameterized statistical models, 
	they are the only ones having \textit{efficient 
	estimators}, meaning roughly that it is possible, given sample data, to estimate 
	the unknown parameters of the model 
	in the ``best possible way". 
	Examples of exponential families are found among the most common probability distributions: 
	\textit{Bernoulli, beta, binomial, chi-square, Dirichlet, exponential, gamma, geometric, multinomial, 
	normal, Poisson,} to name but just a few. 

	For us, the important property of an exponential family $\mathcal{E}$ is that its canonical almost Hermitian 
	structure (on $T\mathcal{E}$) is always a K\"{a}hler structure (Corollary \ref{corollary encore que dire?}), 
	allowing us to define, via some refinements of Dombrowski's arguments, what we shall call 
	the \textit{K\"{a}hlerification} of an exponential family, denoted $\mathcal{E}^{\mathbb{C}}$ 
	(see \S\ref{section il y en a pas deux}, Definition \ref{definition kahlerification}). By construction, a 
	K\"{a}hlerification $\mathcal{E}^{\mathbb{C}}$ is 
	a K\"{a}hler manifold, and it always comes with a Riemannian submersion $\pi_{\mathcal{E}}\,:\,\mathcal{E}^{\mathbb{C}}\rightarrow
	\mathcal{E}\,.$ As an important example, the K\"{a}hlerification of 
	$\mathcal{P}_{n}^{\times}$ yields $\mathbb{P}(\mathbb{C}^{n})^{\times},$ i.e. 
	$(\mathcal{P}_{n}^{\times})^{\mathbb{C}}\cong \mathbb{P}(\mathbb{C}^{n})^{\times}\,,$
	where $\mathbb{P}(\mathbb{C}^{n})^{\times}:=\big\{[z_{1},...,z_{n}]\in \mathbb{P}(\mathbb{C}^{n})\,\vert\,
	z_{k}\neq 0\,\,\,\textup{for all}\,\,k=1,...,n\big\}$ (we use homogeneous coordinates, 
	see \S\ref{section the motivating example} and Proposition \ref{proposition le bon group indeed}). 
	Other classical K\"{a}hler manifolds can be realized as the 
	K\"{a}hlerification of appropriate exponential families, like $\mathbb{C}^{n}$ and the Poincar\'{e} upper half 
	plane $\mathbb{H}\,.$

	K\"{a}hlerifications (and their completions) generalize the usual quantum state space $\mathbb{P}(\mathbb{C}^{n})\,.$ 
	But what about the observables? In the geometrical formulation, observables are \textit{K\"{a}hler functions}, 
	i.e. functions $f\,:\,\mathbb{P}(\mathbb{C}^{n})\rightarrow \mathbb{R}$ whose associated
	Hamiltonian vector fields $X_{f}$ are Killing vector fields. 

	In \S\ref{sss kahler functions}, 
	we investigate the properties of K\"{a}hler functions in the context of K\"{a}hlerification and obtain
	a relation between the statistical structure of $\mathcal{E}$ and a class of K\"{a}hler functions on 
	$\mathcal{E}^{\mathbb{C}}\,,$ as follows. If $(\Omega,dx)$ denotes the measured space on which 
	$\mathcal{E}$ is defined and if $X\,:\,\Omega\rightarrow \mathbb{R}$ belongs to a certain class of random variables which depends on 
	the exponential structure of $\mathcal{E}$ (see \eqref{eee definition les observables...}), 
	then for any holomorphic isometry $\Phi\,:\,\mathcal{E}^{\mathbb{C}}\rightarrow \mathcal{E}^{\mathbb{C}}\,,$ 
	the function 
	\begin{eqnarray}\label{equation originelle}
		\mathcal{E}^{\mathbb{C}}\rightarrow \mathbb{R}\,,\,\,\,\,\,
		z\mapsto\int_{\Omega}\,X(x)\big[(\pi_{\mathcal{E}}\circ \Phi)(z)\big](x)dx
	\end{eqnarray} 
	is a K\"{a}hler function (Corollary \ref{ccc les fonctions de la formes...}).

	This result is a ``geometric analogue" of the usual spectral decomposition theorem for 
	Hermitian matrices. For, when $\mathcal{E}=\mathcal{P}_{n}^{\times}\,,$ then the 
	space of K\"{a}hler functions $\mathscr{K}(\mathbb{P}(\mathbb{C}^{n}))$ on $\mathbb{P}(\mathbb{C}^{n})$ (the latter 
	viewed as the natural ``completion" of $(\mathcal{P}_{n}^{\times})^{\mathbb{C}}$)\,, is isomorphic in the Lie algebra 
	sense to the space of $n\times n$ skew Hermitian matrices $\mathfrak{u}(n)\,,$ i.e., $\mathscr{K}(\mathbb{P}(\mathbb{C}^{n}))
	\cong \mathfrak{u}(n)\,,$ and the decomposition in \eqref{equation originelle} is in this case a rephrasing of 
	the diagonalisability of a Hermitian matrix (see Lemma \ref{lemma comomentum map}).
		
	In \S\ref{section quantum mechanics}, while revisiting the geometrical formulation of quantum mechanics, we
	use this correspondence with spectral theory to \textit{propose} a definition 
	for the spectrum of a K\"{a}hler function $f\,:\,\mathcal{E}^{\mathbb{C}}\rightarrow\mathbb{R}$ 
	of the form given in \eqref{equation originelle}. 
	Our definition reads as follows:
	$\textup{spec}(f):=\textup{Im}(X)\,,$ where $\textup{Im}(X)$ denotes the image of the random variable 
	$X\,:\,\Omega\rightarrow\mathbb{R}\,.$ As the decomposition in \eqref{equation originelle} is usually not unique, 
	our definition is only consistent when invariance properties are met. We also define, for a K\"{a}hler function $f$ 
	as in \eqref{equation originelle} and a point $z\in \mathcal{E}^{\mathbb{C}}\,,$ what might be interpreted, 
	in a physical jargon, as the probability that the observable $f$ yields, upon measurement, 
	the eigenvalue $\lambda$ while the system is in the state 
	$z\,:$ $P_{f,z}(\lambda):=\int_{X^{-1}(\lambda)}\,\big[(\pi_{\mathcal{E}}\circ \Phi)(z)\big](x)dx\,.$   
	
	When $\mathcal{E}=\mathcal{P}_{n}^{\times}\,,$ then $\textup{spec}(f)$ and $P_{f,z}$ are well defined 
	for all K\"{a}hler functions $f$ on $\mathbb{P}(\mathbb{C}^{n})$ (we extend our definitions via density arguments)
	and together, they yield the usual probabilistic interpretation of the geometrical formulation. 
	In particular, $P_{f,z}$ depends on the expression $\cos^{2}\big(d(\,,\,)\big)\,,$ where $d(\,,\,)$
	is the geodesic distance on $\mathbb{P}(\mathbb{C}^{n})\,.$

	When $\mathcal{E}=\mathcal{B}(n,q)$ is the space of binomial distributions with parameter $q\in\, ]0,1[$ 
	defined over $\Omega:=\{0,...,n\}\,,$ then we have the following results (see \S\ref{sss binomial}). The K\"{a}hlerification of 
	$\mathcal{B}(n,q)$ is, up to completion, 
	the 2-dimensional sphere of radius $n$ (regarded as a submanifold of $\mathbb{R}^{3}$). The space of 
	K\"{a}hler functions on the sphere is generated by 
	the functions $1,x,y,z$ and is isomorphic, in the Lie algebra sense, 
	to $\mathfrak{u}(2)\,.$ For the spectral theory, if $(u,v,w)\in \mathbb{R}^{3}$ is a vector 
	whose Euclidean norm is $n/2\,,$ 
	then the spectrum (in our sense) of the function $f(x,y,z)=ux+vy+wz$ is 
	exactly $\{-j,-j+1,...,j-1,j\}\,,$ where $j=n/2\,,$ and 
	the probability $P_{f,(x,yz)}$ is given, for $k\in \{0,...,n\}\,,$ by
	\begin{eqnarray}\label{eee le spin et oui!}
		P_{f,(x,y,z)}(-j+k)=\binom{n}{k}\,\Big(\cos^{2}\big(\theta/2\big)\Big)^{k}
		\Big(\sin^{2}\big(\theta/2\big)\Big)^{n-k}\,,
	\end{eqnarray}
	where $\binom{n}{k}=\frac{n!}{(n-k)!k!}$ and where $\theta$ is an angle 
	satisfying $\frac{ux+vy+wz}{\|(u,v,w)\|}=\cos(\theta)\,.$
	
	We recognize --and this is one of the main observations of this paper-- a formula which describes 
	the spin of a particle passing through two consecutive Stern-Gerlach devices. Recall that a Stern-Gerlach device produces 
	a magnetic field oriented in a chosen direction, and that if a beam of particles (for example silver atoms) is send through it, 
	then, due to spin effect, it will split into a finite number of deflected parts. 
	Equation \eqref{eee le spin et oui!}, in this respect, gives 
	the probability that a particle entering the second Stern-Gerlach device with maximum spin\footnote{By ``maximum spin" we mean 
	that the eigenvalue of the usual spin operator of the particle along the direction of the magnetic field of the first Stern-Gerlach 
	device, is, among the possible values 
	$-j,-j+1,...,j-1,j\,,$ exactly $j\,.$ Here $j\in \{0,1/2,1,3/2,...\}$ is the spin of the particle.\label{fff footnote}} 
	is deflected into the $(-j+k)$-th outgoing beam, $\theta$ being the angle between the two magnetic fields produced by the
	Stern-Gerlach devices. 
	
	What is remarkable with this result is that it only depends on the statistical structure of 
	the binomial distribution $\mathcal{B}(n,q)\,,$ providing support to the idea 
	that the quantum formalism owes part of its mathematical structure to statistical concepts. Also, it 
	shows that the probabilistic interpretation of the geometrical formulation, through $\textup{spec}(f)$ and $P_{f,z}\,,$ 
	can be extended to more general situations than the one originally considered with the complex projective space, 
	situations which are physically relevant.

	We have to emphasis, however, that not \textit{all} the possibilities of 
	the Stern-Gerlach experiment are exhausted with \eqref{eee le spin et oui!}, which may be interpreted, at first,
	as a limitation of the statistical approach. But actually it is not. The remaining probabilities, as it turns out, can 
	be obtained fairly easily by means of the ``universal" inclusion\footnote{By ``universal", we simply mean that 
	for any statistical manifold $S$ defined over a finite set $\{x_{1},...,x_{n}\}\,,$ there is a canonical inclusion 
	$S\subseteq\mathcal{P}_{n}^{\times}\,.$ The space $\mathcal{P}_{n}^{\times}$ thus appears as a 
	``universal container".} 
	$\mathcal{B}(n,q)\subseteq\mathcal{P}_{n+1}^{\times}\,,$ as follows. By ``K\"{a}hlerifiying" this inclusion,
	one gets an embedding $S^{2}\hookrightarrow \mathbb{P}(\mathbb{C}^{n+1})$ which makes it possible to extend every 
	K\"{a}hler function $f$ on $S^{2}$ to a unique K\"{a}hler function $\widehat{f}$ on 
	$\mathbb{P}(\mathbb{C}^{n+1})\,,$ the latter function having the advantage to carry more 
	informations than the original one. In fact, we show that the map $f\mapsto \widehat{f}$ is a 
	homomorphism of Lie algebras which is, via the appropriate identifications, an 
	irreducible unitary representation of $\mathfrak{u}(2)$ (see Proposition \ref{ppp dans un café au lait} and 
	lemmas \ref{lll existence de QQQQ} and \ref{lll tiny toons}). 
	This allows us to extract the remaining 
	probabilities (recall that the spin is usually 
	described by the unitary representations of $\mathfrak{su}(2)$). 
	Mathematically, this brings an interesting link between a purely geometrical problem -- extending K\"{a}hler functions-- 
	and representation theory. 
	
	Collecting our results, we conclude that \textit{the spin 
	of a particle can be entirely understood by means of the binomial distribution $\mathcal{B}(n,q)\,.$}

	In \S\ref{section other examples}, we briefly consider the space 
	$\mathcal{N}(\mu,1)$ of Gaussian distributions of mean $\mu$ and fixed deviation $\sigma=1$ over 
	$\Omega=\mathbb{R}\,,$ give its K\"{a}hlerification and describe its associated ``spectral theory". 
	As we observe, this exponential family is closely related to the quantum harmonic oscillator, a fact 
	which can only be fully understood by the introduction of an infinite dimensional analogue of $\mathcal{P}_{n}^{\times}\,.$ 
	On this, however, we say very little due to space limitation and refer the reader to \cite{Molitor-hydro}.\\

	To summarize, we carried out, following ideas of \cite{Molitor-quantique}, 
	an analysis of the mathematical foundations of quantum mechanics, 
	using a geometric and information-theoretical approach, which points towards the following conclusion:
	\textit{the quantum formalism is grounded on the K\"{a}hler geometry which naturally emerges from statistics.}
	Examples like the spin support this claim, and the various 
	mechanisms involved have been described; we defined the K\"{a}hlerification of an exponential family 
	and sketched the very basis of what may be considered as a ``statistical spectral" theory for K\"{a}hler functions.
	In doing so, we observed an intriguing link between the problem of extending K\"{a}hler functions and representation theory 
	which seems to connect our approach to the standard way physicists work.

	The author is fully aware that the techniques and definitions introduced in this paper 
	are still in an infant stage, and that they should probably be modified in the light of further progress. 
	Nevertheless, it is likely that the relationship between statistics and 
	K\"{a}hler geometry will grow in importance, and we hope that it may help to get a better understanding 
	of the mathematical foundations of quantum mechanics. 
	Deepening this comprehension might well led 
	to a viable generalization of quantum mechanics, or at least to a new comprehension of some of its conceptually 
	puzzling aspects, especially those related to the measurement problem.

\section{Information geometry}\label{section information geometry}
	
	In this section, we review the basic concepts of information geometry 
	needed throughout this paper. Our (very short) presentation follows the currently 
	reference book \cite{Amari-Nagaoka} whose emphasis is on the Fisher metric and 
	$\alpha$-connections of a given statistical model (see also \cite{Murray-Rice}).\\
	
	A \textit{statistical manifold} (or \textit{statistical model}), is a couple $(S,j)$ where 
	$S$ is a manifold and where $j$ is an injective map from $S$ to the space of 
	all probability density functions $p$ defined on a fixed measured space $(\Omega,dx)
	\footnote{Depending on the symbole we use for the variable living in $\Omega\,,$ for example ``$x$", ``$k$", etc., we shall 
	use the notation $``dx"$, $``dk",$ etc., for the measure on $\Omega\,.$}\,:$
	\begin{eqnarray}
		j\,:\,S\hookrightarrow\Big\{p\,:\,\Omega\rightarrow\mathbb{R}\,\big\vert\,
		p\,\,\textup{is measurable,}\,\,\,p\geq 0\,\,\,\textup{and}\,\,\,\int_{\Omega}\,
		p(x)\,dx=1\Big\}\,.
	\end{eqnarray}

	In the case of a discrete space $\Omega\,,$ it will be implicitly assumed that 
	$dx$ is the counting measure, i.e. 
	$dx(A)=\textup{card}(A)\,,$ where $\textup{card}(A)$ denotes the cardinality
	of a given subset $A\subset \Omega\,.$ In this situation, integration of 
	a function $X\,:\,\Omega\rightarrow \mathbb{R}$ with respect to the probability $p\,dx$ ($p$ being 
	a probability density function), is simply given by : 
	\begin{eqnarray}
		\int_{\Omega}\,X(x)\, p(x)\,dx=\sum_{x\in \Omega}\,X(x)\,p(x)\,.
	\end{eqnarray}

	As a matter of notation, if $(\xi\,:\,U\subseteq S \rightarrow \mathbb{R}^{n})$ is a chart 
	of a statistical manifold $S$ with local coordinates $\xi=(\xi_{1},...,\xi_{n})\,,$ 
	then we shall indistinctly write $p(x;\xi)$ or $p_{\xi}(x)$ for the probability density 
	function determined by $\xi$ and in the variable $x\in \Omega\,.$\\

	Now, given a ``reasonable" statistical manifold $S\,,$ it is possible to define a metric $h_{F}$ 
	and a family of connections $\nabla^{(\alpha)}$ on $S$ ($\alpha\in\mathbb{R}$) in the following 
	way:  for a chart $\xi=(\xi_{1},...,\xi_{n})$ 
	of $S\,,$ define  
	\begin{description}
	\item[$\bullet$] $(h_{F})_{\xi}\big(\partial_{i},\partial_{j}):=
		E_{p_{\xi}}(\partial_{i}\textup{ln}\,(p_{\xi})\cdot
		\partial_{j}\textup{ln}\,(p_{\xi})\big)\,,$
	\item[$\bullet$] $\Gamma_{ij,k}^{(\alpha)}(\xi):=E_{p_{\xi}}\Big[\Big(\partial_{i}\partial_{j}
		\textup{ln}\,(p_{\xi})+\dfrac{1-\alpha}{2}\partial_{i}\textup{ln}\,(p_{\xi})\cdot\partial_{j}
		\textup{ln}\,(p_{\xi})\Big)\,\partial_{k}\textup{ln}\,(p_{\xi})\Big]\,,$
	\end{description}
	where $E_{p_{\xi}}$ denotes the mean, or expectation, with respect to the probability 
	$p_{\xi}\,dx\,,$ and where $\partial_{i}$ is a shorthand for $\partial/\partial_{\xi_{i}}\,.$\\
	It can be shown that if the above expressions are defined and smooth for every chart of 
	$S$ (this is not always the case), then $h_{F}$ is a well defined metric on $S$ called the 
	\textit{Fisher metric}, and that the $\Gamma_{ij,k}^{(\alpha)}$'s are the Christoffel symbols 
	of a connection $\nabla^{(\alpha)}$ called the $\alpha$-\textit{connection}. 
	Among the $\alpha$-connections, the $(\pm1)$-connections are particularly important; 
	the 1-connection is usually referred to as the \textit{exponential connection}, also 
	denoted $\nabla^{(e)}\,,$ while the $(-1)$-connection is referred to as 
	the \textit{mixture connection}, denoted $\nabla^{(m)}\,.$\\
	In this paper, we will only consider statistical manifolds $S$ for which the Fisher 
	metric and $\alpha$-connections are well defined. \\
	
	One particularity of the $(\pm\alpha)$-connections is that they are \textit{dual} of each other 
	with respect to the Fisher metric $h_{F}\,,$ or equivalently, that they form a \textit{dualistic 
	structure} on $S\,.$ The general definition of a dualistic structure on an arbitrary  
	manifold $M$ is as follows: a dualistic 
	structure on $M$ is a triple $(h,\nabla,\nabla^{*})$ 
	where $h$ is a Riemannian metric on $M$ and where $\nabla$ and $\nabla^{*}$ are connections 
	satisfying 
	\begin{eqnarray}\label{equation relation journée en or}
	X\big(h(Y,Z)\big)=h_{}\big(\nabla^{}_{X}Y,Z\big)+h_{}\big(Y,\nabla^{*}_{X}Z\big)\,,
	\end{eqnarray}
	for all vector fields $X,Y,Z$ on $M\,.$ The connection $\nabla^{*}$
	is called the \textit{dual connection}, or \textit{conjugate connection}, of the connection 
	$\nabla$ (and vice versa)\footnote{Given a connection $\nabla$ on a Riemannian 
	manifold $(M,h)\,,$ there exists a unique connection $\nabla^{*}$ on $M$ such that 
	\eqref{equation relation journée en or} holds; it is thus justified to call $\nabla^{*}$
	the dual connection of $\nabla\,.$}. 
	
	An example of dualistic structure is, as we already said, given by the triple 
	$(h_{F},\nabla^{(\alpha)},\nabla^{(-\alpha)})$ that one can always consider for a fixed 
	$\alpha\in \mathbb{R}$ on a statistical manifold $S$ (provided of course that the Fisher 
	metric and $(\pm\alpha)$-connections exist).\\
	
	An important class of dualistic structures is that of \textit{dually flat structures}. 
	A dually flat structure on a manifold $M$ is a dualistic structure 
	$(h,\nabla,\nabla^{*})$ for which both connections are flat, meaning that their 
	torsions and curvature tensors vanish. As conventions are not uniform in the 
	literature, let us agree that the torsion $T$ and the curvature tensor $R$ of a connection 
	$\nabla$ on $M$ are defined as
	\begin{eqnarray}
	T(X,Y)&:=&\nabla_{X}Y-\nabla_{Y}X-[X,Y]\,,\nonumber\\
	R(X,Y)Z&:=&\nabla_{X}\nabla_{Y}Z-\nabla_{Y}\nabla_{X}Z-\nabla_{[X,Y]}Z\,,
	\end{eqnarray}
	where $X,Y,Z$ are vector fields on $M\,.$
	
	Given a dualistic structure $(h,\nabla,\nabla^{*})$ on a manifold $M\,,$ there exists a simple relation 
	between the curvature tensor $R$ of $\nabla$ and the curvature tensor $R^{*}$ 
	of $\nabla^{*}$ which is 
	given by the following formula: 
	\begin{eqnarray}
		h\big(R(X,Y)Z,W\big)=-h\big(R^{*}(X,Y)W,Z\big)\,,
	\end{eqnarray}
	where $X,Y,Z,W$ are vector fields on $S\,.$ From this relation, it is clear that if 
	$\nabla$ and $\nabla^{*}$ are both torsion-free, then
	$(h,\nabla,\nabla^{*})$ is dually flat if and only if $R$ or $R^{*}$ vanishes identically (in which 
	case both curvature tensors vanish). In particular, since $\alpha$-connections are always torsion-free, 
	$(h_{F},\nabla^{(\alpha)},\nabla^{(-\alpha)})$ 
	is dually flat if and only if $R^{(\alpha)}$ or $R^{(-\alpha)}$ 
	vanishes identically (here $R^{(\alpha)}$ 
	denotes the curvature tensor of the $\alpha$-connection).

\section{Exponential families}\label{section expo families}
%	\textit{Exponential families} are particular cases of dually 
%	flat statistical manifolds whose importance in the context of information
%	geometry is considerable; most of the present paper's results are based 
%	on their properties.  
%
%	Here is their definition:
\begin{definition}\label{definition exp}
	An exponential family $\mathcal{E}$ on a measured space $(\Omega,dx)$ is a set of probability 
	density functions $p(x;\theta)$ of the form 
	\begin{eqnarray}\label{equation definiton exp}
		p(x;\theta)=\textup{exp}\,\bigg\{C(x)+\sum_{i=1}^{n}\,\theta_{i}F_{i}(x)-\psi(\theta)\bigg\}\,,
	\end{eqnarray}
	where $C,F_{1},...,F_{n}$ are measurable functions on $\Omega\,,$ $\theta=(\theta_{1},...,\theta_{n})$ 
	is a vector varying in an open subset 
	$\Theta$ of $\mathbb{R}^{n}$ and where 
	$\psi$ is a function defined on $\Theta\,.$
\end{definition}
	In the above definition, it is understood that if $\Omega$ is discrete, then 
	$dx$ should be the counting measure. It is also assumed that the family 
	$\{1,F_{1},...,F_{n}\}$ is linearly independent, so that the map $p(x,\theta)\mapsto\theta\in\Theta$ 
	becomes a bijection, hence defining a global chart of $\mathcal{E}\,.$ 
	The parameters $\theta_{1},...,\theta_{n}$ are called the 
	\textit{natural} or \textit{canonical} \textit{parameters} of the exponential family $\mathcal{E}\,.$\\

	Besides the natural parameters $\theta_{1},...,\theta_{n}\,,$ an exponential family $\mathcal{E}$ 
	possesses another particularly important parametrization which is given by the 
	\textit{expectation} or \textit{dual parameters} $\eta_{1},...,\eta_{n}\,:$ 
	\begin{eqnarray} 
		\eta_{i}(p_{\theta}):=E_{p_{\theta}}(F_{i})=\int_{\Omega}\,F_{i}(x)\,p_{\theta}(x)\,dx\,.	
	\end{eqnarray} 
	It is not difficult, assuming $\psi$ to be smooth, to show that 
	$\eta_{i}(p_{\theta})=\partial_{\theta_{i}}\psi\,.$ The map 
	$\eta=(\eta_{1},...,\eta_{n})$ is thus a global chart of $\mathcal{E}$ provided that
	$(\partial_{\theta_{1}}\psi,...,\partial_{\theta_{n}}\psi)\,:
	\,\Theta\rightarrow \mathbb{R}^{n}$ is a diffeomorphism onto its image, 
	condition that we will always assume.\\

	The natural and expectation parameters are important in that they form 
	affine coordinate systems\footnote{Let us recall that an affine coordinate system on a 
	manifold $M$ with a flat connection $\nabla\,,$ or simply 
	a $\nabla$-affine chart, is a coordinate system in which 
	all the Christoffel symbols associated to $\nabla$ vanish.} 
	with respect to $\nabla^{(e)}$ and $\nabla^{(m)}\,:$ 
	
\begin{proposition}{\cite{Amari-Nagaoka}}\label{proposition proprietes fam exp}
	Let $\mathcal{E}$ be an exponential family such as in 
	\eqref{equation definiton exp}. Then $(\mathcal{E},h_{F},\nabla^{(e)},\nabla^{(m)})$
	is dually flat and $\theta=(\theta_{1},...,\theta_{n})$ is an affine coordinate system 
	with respect to $\nabla^{(e)}$ while $\eta=(\eta_{1},...,\eta_{n})$ 
	is an affine coordinate system with respect to $\nabla^{(m)}\,.$ 
	Moreover, the following relation holds :
	\begin{eqnarray}\label{equation etrange, et curieuse}
		h_{F}(\partial_{\theta_{i}},\partial_{\eta_{j}})=\delta_{ij}\,,
	\end{eqnarray}
	where $\delta_{ij}$ denotes the Kronecker symbol. 
\end{proposition}
	
	Let us now give some examples of exponential families, mostly taken from \cite{Amari-Nagaoka}. 
\begin{example}[Normal Distribution]\label{example normal distribution}
	Normal distributions 
	\begin{eqnarray}		
		p(x;\mu,\sigma)=\dfrac{1}{\sqrt{2\pi}\sigma}\textup{exp}
		\Big\{-\dfrac{(x-\mu)^{2}}{2\sigma^{2}}\Big\}\,\,\,\,\,\,(x\in\mathbb{R})\,,
	\end{eqnarray}
	form a 2-dimensional statistical manifold parameterized 
	by $(\mu,\sigma)\in \mathbb{R}\times\mathbb{R}_{+}^{*}$ (here $\mathbb{R}_{+}^{*}:=\{x\in \mathbb{R}\,\vert\,x>0\}$),
	subsequently denoted $\mathcal{N}(\mu,\sigma^{2})\,.$ This family is easily 
	seen to be an exponential one, for one may write
	\begin{eqnarray}\label{equation reecriture normal}
		p(x;\mu,\sigma)=\textup{exp}\Big\{\dfrac{\mu}{\sigma^{2}}x-\dfrac{1}{2\sigma^{2}}x^{2}
		-\dfrac{\mu^{2}}{2\sigma^{2}}-\textup{ln}(\sqrt{2\pi}\sigma)\Big\}\,,
	\end{eqnarray}
	and define 
	\begin{eqnarray}
		\theta_{1}=\dfrac{\mu}{\sigma^{2}},
			\,\,\,\theta_{2}=-\dfrac{1}{2\sigma^{2}}\,,\,\,\,
			C(x)=0\,,\,\,\,F_{1}(x)=x\,,\,\,\,
			F_{2}(x)=x^{2}\,,\,\,\,\psi(\theta)=
			-\dfrac{(\theta_{1})^{2}}{4\theta_{2}}+\dfrac{1}{2}\textup{ln}\,
			\Big(-\dfrac{\pi}{\theta_{2}}\Big)\,.
	\end{eqnarray}
\end{example}
%\begin{example}[Poisson Distribution] Poisson distributions 
%	\begin{eqnarray}
%		p(x;\xi)=e^{-\xi}\dfrac{\xi^{x}}{x!}\,\,\,\,\,(x\in\mathbb{N})\,,
%	\end{eqnarray}
%	form a 1-dimensional statistical manifold, subsequently denoted $\mathcal{P}\textup{oisson}\,,$ 
%	parameterized by $\xi\in\mathbb{R}_{+}^{*}\,.$ Clearly,
%	$p(x;\xi)=\textup{exp}\big\{-\textup{ln}(x!)+\textup{ln}(\xi)\,x
%	-\xi\big\}\,,$ and thus, $\mathcal{P}\textup{oisson}$ is an exponential family with 
%	\begin{eqnarray}
%		\theta=\textup{ln}(\xi)\,,\,\,\,C(x)=-\textup{ln}(x!)\,,\,\,\,
%		F(x)=x\,,\,\,\,\psi(\theta)=\textup{exp}(\theta)\,.
%	\end{eqnarray}
%\end{example}
\begin{example}[finite $\Omega$]\label{example finite omega}
	For a finite set $\Omega=\{x_{1},...,x_{n}\}\,,$ define 
	\begin{eqnarray}
		\mathcal{P}_{n}^{\times}:=\Big\{p\,:\,\Omega\rightarrow \mathbb{R}\,\big\vert\,p(x)>0\,\,
		\textup{for all}\,\,x\in \Omega\,\,\textup{and}\,\,\sum_{k=1}^{n}\,p(x_{k})=1\Big\}\,.
	\end{eqnarray}
	The space $\mathcal{P}_{n}^{\times}$ is clearly a statistical manifold of dimension $n-1\,,$ 
	and it can be turned into an exponential family by means of the following parameterization: 
	\begin{eqnarray}
		p(x;\theta)=\textup{exp}\bigg\{\sum_{i=1}^{n-1}\,\theta_{i}F_{i}(x)-\psi(\theta)\bigg\}\,,
	\end{eqnarray} 
	where $x\in \Omega\,,$ $\theta=(\theta_{1},...,\theta_{n-1})\in \mathbb{R}^{n-1}\,,$ $F_{i}(x_{j})
	=\delta_{ij}$ and where $\psi(\theta)=-\textup{ln}\big(1+\sum_{i=1}^{n-1}\,\textup{exp}(\theta_{i})
	\big)\,.$
\end{example} 
%\begin{example}\label{eee normal with fixed...}
%	If one considers normal distributions $p(x;\mu,\sigma)$ (see Example 
%	\ref{example normal distribution}) with a fixed $\mu$ or a fixed $\sigma\,,$ 
%	then one gets a submanifold of $\mathcal{N}(\mu,\sigma^{2})$ which is again an exponential 
%	family. For example, 
%	\begin{eqnarray}
%		&\mathcal{N}(\mu,1)&:=\{p(x;\mu,\sigma)\in \mathcal{N}(\mu,\sigma^{2})\,
%			\big\vert\,\sigma=1\}\\
%		{and}&\mathcal{N}(0,\sigma^{2})&:=\{p(x;\mu,\sigma)\in \mathcal{N}(\mu,\sigma^{2})\,
%		\big\vert\,\mu=0\}
%	\end{eqnarray}
%	are clearly exponential families (see \eqref{equation reecriture normal}) 
%	whose natural parameters are respectively  
%	$\theta=\mu$ and $\theta=-1/(2\sigma^{2})\,.$
%\end{example}
%	As we will see in the next section, one can use the dually flat structure of an 
%	exponential family $\mathcal{E}$ to construct a K\"{a}hler structure on its tangent bundle 
%	$T\mathcal{E}\,,$ hence allowing the introduction of K\"{a}hler geometry 
%	in the context of information geometry. 

\section{Dombrowski's construction}\label{section dom}
	In this section, we explain, following Dombrowski's paper \cite{Dombrowski}, 
	how the tangent bundle of a given dually flat manifold can be turned into a 
	K\"{a}hler manifold by a simple geometric construction. This implies in particular 
	that the tangent bundle $T\mathcal{E}$ of an exponential family 
	is naturally a K\"{a}hler manifold. 
	
	Most of the results of this section are due to Dombrowski, except for Lemma 
	\ref{lemma T=0 equi} and subsequent corollaries which are natural extensions 
	of \cite{Dombrowski}. \\

	Recall that if $M$ is a manifold endowed with an affine connection $\nabla\,,$ then
	Dombrowski splitting Theorem holds (see \cite{Dombrowski,Lang}) :
	\begin{eqnarray}
		T(TM)\cong TM\oplus TM\oplus TM\,,
	\end{eqnarray}
	this splitting being viewed as an isomorphism of vector bundles over $M\,,$ and the 
	isomorphism, say $\Phi\,,$
	being
	\begin{eqnarray}\label{equation Dombrowski}
	T_{u_{x}}TM\ni A_{u_{x}}\overset{\Phi}{\longmapsto} 
	\big(u_{x},\pi_{*_{u_{x}}}A_{u_{x}},K A_{u_{x}}\big)\,,
	\end{eqnarray}
	where $\pi\,:\,TM\rightarrow M$ is the canonical projection and where 
	$K\,:\,T(TM)\rightarrow TM$ is the canonical connector associated to the connection 
	$\nabla$ (see \cite{Lang}). 

	Having $A_{u_{x}}=\Phi^{-1}\big((u_{x},v_{x},w_{x})\big)\in T_{u_{x}}TM\,,$ we shall write, for 
	simplicity, $A_{u_{x}}=(u_{x},v_{x},w_{x})$ instead of $\Phi^{-1}\big((u_{x},v_{x},w_{x})\big)\,,$
	i.e., we will drop $\Phi\,.$ The second component 
	$v_{x}$ is usually referred to as the horizontal component of $A_{u_{x}}$ (with respect 
	to the connection $\nabla$) and $w_{x}$ the vertical component.  \\

	With the above notation, and provided that $M$ is endowed with a Riemannian metric $h\,,$ 
	it is a simple matter to define on 
	$TM$ an almost Hermitian structure. 
	Indeed, we define a metric $g\,,$ a 2-form $\omega$
	and an almost complex structure $J$ by setting
	\begin{eqnarray}\label{equation definition G, omega, etc.}
		g_{u_{x}}\big(\big(u_{x},v_{x},w_{x}\big),
			\big({u}_{x},\overline{v}_{x},
			\overline{w}_{x}\big)\big)&:=&
			h_{x}\big(v_{x},\overline{v}_{x}\big)+
			h_{x}\big(w_{x},\overline{w}_{x}\big)\,,\nonumber\\
		\omega_{u_{x}}\big(\big(u_{x},v_{x},w_{x}\big),
			\big({u}_{x},\overline{v}_{x},
			\overline{w}_{x}\big)\big)&:=&h_{x}\big(v_{x},\overline{w}_{x}\big)-
			h_{x}\big(w_{x},\overline{v}_{x}\big)\,,\nonumber\\
		J_{u_{x}}\big(\big(u_{x},v_{x},w_{x}\big)\big)&:=&
		\big(u_{x},-w_{x},v_{x}\big)\,,
	\end{eqnarray}
	where $u_{x},v_{x},w_{x},\overline{v}_{x},\overline{w}_{x}
	\in T_{x}M\,.$\\
	Clearly, $J^{2}=-\textup{Id}$  and $g(J\,.\,,J\,.\,)=g(\,.\,,\,.\,)\,,$ which means that 
	$(TM,g,J)$ is an almost Hermitian manifold, and one readily sees that 
	$g,J$ and $\omega$ are compatible, i.e., that $
	\omega=g\big(J\,.\,,\,.\,\big)\,;$ the $2$-form $\omega$ is thus the fundamental 2-form of 
	the almost Hermitian manifold $(TM,g,J)\,.$ This is Dombrowski's construction.

	Observe that the map $\pi\,:\,(TM,g)\rightarrow (M,h)$ is a Riemannian submersion. \\

	In \cite{Dombrowski}, Dombrowski shows the following:
\begin{proposition}[\cite{Dombrowski}]\label{proposition integrable egale flat}
	Let $\nabla$ be an affine connection defined on a manifold $M\,,$ and let $J$ be 
	the almost complex structure associated to $\nabla$ as in \eqref{equation definition G, omega, etc.}. 
	Then, 
	\begin{eqnarray}
		J\,\,\,\textup{is integrable}\,\,\,\,\,\Leftrightarrow\,\,\,\,\,\nabla\,\,\,\textup{is flat}\,.
	\end{eqnarray}
\end{proposition}
	The tangent bundle $TM$ of a manifold $M$ endowed with a flat connection $\nabla$ is thus 
	naturally a complex manifold. If in addition $M$ is equipped with a Riemannian metric 
	$h\,,$ then $TM$ becomes a complex Hermitian manifold for the Hermitian structure $(g,J,\omega)$ 
	considered above. 

	For the $2$-form $\omega$ defined in \eqref{equation definition G, omega, etc.}, 
	we have the following result:
\begin{lemma}\label{lemma T=0 equi}
	Let $(M,h)$ be a Riemannian manifold endowed with a flat connection $\nabla\,,$ and let 
	$\omega$ be the $2$-form defined as in \eqref{equation definition G, omega, etc.}\,. Then,
	\begin{eqnarray}
			d\omega=0\,\,\,\,\,\Leftrightarrow\,\,\,\,\,T^{*}=0\,,
	\end{eqnarray}
	where $T^{*}$ denotes the torsion of the dual connection $\nabla^{*}\,.$ 
\end{lemma}
\begin{proof}
	Let us consider a $\nabla$-affine chart 
	$(\varphi\,:\,U\subseteq M \rightarrow \mathbb{R}^{n})$ 
	with local coordinates $\varphi=(x_{1},...,x_{n})\,.$ In this chart, all Christoffel symbols 
	$\Gamma^{k}_{ij}$
	associated to $\nabla$ vanish. Let us also consider the chart 
	$(\overline{\varphi}\,:\,\overline{U}\subseteq TM\rightarrow \mathbb{R}^{n}\times
	\mathbb{R}^{n})$ with local coordinates $\overline{\varphi}=(x_{1},...,x_{n},
	{y}_{1},...,{y}_{n})$ canonically associated to $(U,\varphi)\,,$ i.e., 
	$\overline{U}:=\pi^{-1}(U)$ ($\pi\,:\,TM\rightarrow M$ 
	being the canonical projection) and where $
		\overline{\varphi}\big(\sum_{i=1}^{n}\,a_{i}\partial_{x_{i}}\vert_{x}\big)
		:=(\varphi(x),a_{1},...,a_{n})\,.$
	In this chart, it is not hard to see that 
	\begin{eqnarray}\label{equation je suis constipé!}
		\Omega(x,y)=
		\left(\begin{array}{cc}
			0&(h(x))_{ij}\\
			-(h(x))_{ij}&0
		\end{array}\right)\,,
	\end{eqnarray}
	where $(h(x))_{ij}:=h(\partial_{x_{i}},\partial_{x_{j}})$ and where
	$x=(x_{1},...,x_{n})$ and $y=({y}_{1},...,{y}_{n})\,.$

	Now, $\Omega$ is closed if and only if for all $i,j,k=1,...,n\,,$
	\begin{eqnarray}
		(d\Omega)(\partial_{x_{i}},\partial_{{x}_{j}},\partial_{x_{k}})&=&0\,,\,\,\,\,\,
		(d\Omega)(\partial_{x_{i}},\partial_{{x}_{j}},\partial_{{y}_{k}})=0\,,\label{equation pouette}\\
		(d\Omega)(\partial_{x_{i}},\partial_{{{y}}_{j}},\partial_{{y}_{k}})&=&0\,,\,\,\,\,\,
		(d\Omega)(\partial_{{y}_{i}},\partial_{{{y}}_{j}},\partial_{{y}_{k}})=0\,,
		\label{equation pouette2}
	\end{eqnarray}
	and it is easy, using \eqref{equation je suis constipé!}, to see that the only 
	possibly non-vanishing terms in the equations \eqref{equation pouette} and \eqref{equation pouette2} are 
	$(d\Omega)(\partial_{x_{i}},\partial_{{{y}}_{j}},\partial_{{y}_{k}})=\partial_{x_{i}}h_{jk}
	-\partial_{x_{j}}h_{ik}\,.$ Moreover, it is a simple calculation to show that 
	\begin{eqnarray}
		\partial_{x_{i}}h_{jk}-\partial_{x_{j}}h_{ik}
		=h\big(T^{*}(\partial_{x_{i}},\partial_{x_{j}}),\partial_{x_{k}}\big)\,.
	\end{eqnarray}
	Hence, $d\Omega=0$ if and only if  $T^{*}=0\,.$ The lemma follows. 
\end{proof}
	Recall that an almost Hermitian structure $(g,J,\omega)$ on a given manifold 
	is K\"{a}hler when the following two analytical conditions are met: 
	(1) $J$ is integrable; (2) $d\omega=0\,.$ Having this in mind, Proposition 
	\ref{proposition proprietes fam exp}, Proposition \ref{proposition integrable egale flat} and
	Lemma \ref{lemma T=0 equi} readily imply the following two corollaries:
\begin{corollary}\label{corollary c'est kahler!}
	Let $(h,\nabla,\nabla^{*})$ be a dualistic structure on a manifold $M$ and let 
	$(g,J,\omega)$ be the almost Hermitian structure on $TM$ associated to $(h,\nabla)$ via 
	Dombrowski's construction. Then, 
	\begin{eqnarray}
		(TM,g,J,\omega)\,\,\,\textup{is K\"{a}hler}\,\,\,\,\,\Leftrightarrow\,\,\,\,\,
		(M,h,\nabla,\nabla^{*})\,\,\,\textup{is dually flat.}
	\end{eqnarray}
\end{corollary}
\begin{corollary}\label{corollary encore que dire?}
	The tangent bundle $T\mathcal{E}$ of an exponential family $\mathcal{E}$ is a 
	K\"{a}hler manifold for the K\"{a}hler structure $(g,J,\omega)$ associated to 
	$(h_{F},\nabla^{(e)})$ via Dombrowski's construction.
\end{corollary}
	In the sequel, by the K\"{a}hler structure of $T\mathcal{E}\,,$ we shall implicitly refer to 
	the K\"{a}hler structure of $T\mathcal{E}$ described in Corollary \ref{corollary encore que dire?}.

\section{K\"{a}hler functions on an exponential family}\label{sss kahler functions}

\begin{definition}\label{definition 1}	
	Let $(N,g,J,\omega)$ be a K\"{a}hler manifold. We shall say that a function 
	$f\,:\,N\rightarrow \mathbb{R}$ is a K\"{a}hler function if 
	\begin{eqnarray}
		\mathscr{L}_{X_{f}}g=0\,,
	\end{eqnarray}
	where $X_{f}$ denotes the symplectic 
	gradient of $f$ with respect to the symplectic form $\omega\,,$ i.e., 
	$\omega(X_{f},\,.\,)=df(.)\,,$ and where $\mathscr{L}_{X_{f}}$ denotes the 
	Lie derivative in the direction $X_{f}\,.$
\end{definition}
	Clearly, a K\"{a}hler function $f\,:\,N\rightarrow\mathbb{R}$ preserves 
	the K\"{a}hler structure of $N$ in the sense 
	that $\mathscr{L}_{X_{f}}g=0$ and $\mathscr{L}_{X_{f}}\omega=0\,,$ hence 
	the terminology. Following \cite{Cirelli-Quantum}, we shall also denote by 
	$\mathscr{K}(N)$ the space of all K\"{a}hler functions 
	defined on a K\"{a}hler manifold $N\,.$ When $N$ has a finite number of connected components, then
	the space of K\"{a}hler functions on $N$ is a finite dimensional\footnote{The fact that $\mathscr{K}(N)$ is finite dimensional 
	comes from the following result: if $(M,h)$ is a connected Riemannian manifold, 
	then its space of Killing vector fields $\mathfrak{X}_{\textup{Kill}}(M):=
	\{X\in \mathfrak{X}(M)\,\big\vert\,\mathscr{L}_{X}h=0\}$ 
	is finite dimensional (see for example \cite{Jost}). } Lie algebra for the natural Poisson 
	bracket $\{f,g\}:=\omega(X_{f},X_{g})\,.$ \\

	In the case of a K\"{a}hler structure associated to 
	a dually flat manifold via Dombrowski's construction, we shall use the following terminology: 
\begin{definition}\label{definition 2}
	Let $(h,\nabla,\nabla^{*})$ be a dually flat structure on a given manifold 
	$M\,.$ We shall say that a function $f\,:\,M\rightarrow \mathbb{R}$ is a 
	K\"{a}hler function if $f\circ \pi\,:\,TM\rightarrow\mathbb{R}$ is a K\"{a}hler function with respect to the 
	K\"{a}hler structure of $TM$ associated to $(h,\nabla)$ via Dombrowski's construction 
	(here $\pi\,:\,TM\rightarrow M$ is the canonical projection). 
\end{definition}
	
	We shall denote by $\mathscr{K}(M)$ the space of K\"{a}hler functions 
	on a dually flat manifold $M\,.$ Clearly, $\mathscr{K}(M)\subseteq \mathscr{K}(TM)$ 
	via the map $f\mapsto f\circ \pi\,.$\\

	We now want to characterize the space of K\"{a}hler functions on a dually flat manifold. 
	To this end, recall that a vector field $X$ on a manifold $M$ is said to be 
	$\nabla$\textup{-parallel} with respect to a given connection $\nabla$ if 
	$\nabla_{Y}X=0$ for all vector fields $Y$ on $M\,.$
\begin{proposition} \label{proposition equivalence kahler parallel}
	Let $(h,\nabla,\nabla^{*})$ be a dually flat structure on a manifold $M\,,$ 
	and let $(g,J,\omega)$ be the K\"{a}hler structure on $TM$ 
	associated to $(h,\nabla)$ via Dombrowski's construction. 
	For a given function $f\,:\,M\rightarrow \mathbb{R}\,,$ we have:
	\begin{eqnarray}
		f\,\,\textup{is a K\"{a}hler function}\,\,\,\,\,\,\Leftrightarrow\,\,\,\,\,
		\textup{grad}^{h}(f)\,\,\,\,\,\textup{is}\,\,\,\nabla\textup{-parallel}\,,
	\end{eqnarray}
	where $\textup{grad}^{h}(f)$ is the Riemannian 
	gradient of $f$ with respect to $h\,,$ i.e. $h(\textup{grad}^{h}(f),\,.\,)=df(\cdot)\,.$ 
\end{proposition}
	In order to show Proposition \ref{proposition equivalence kahler parallel}, we need a lemma.
\begin{lemma}\label{lemma j'ecoute la chanson etrange}
	Under the hypothesis of Proposition \ref{proposition equivalence kahler parallel}, and, using 
	the identification given in \eqref{equation Dombrowski}, we have for $u_{x}\in T_{x}M\,,$
	\begin{eqnarray}
		(X_{f\circ\pi})_{u_{x}}=\big(u_{x},0,-\textup{grad}^{h}(f)_{x}\big)\,\,\,\,\,\,\,\,\textup{and}\,\,\,\,\,\,\,\,
		\label{equation y en pas deux}\label{equation le gradient, il est la} 
		\varphi^{X_{f\circ\pi}}_{t}(u_{x})=u_{x}-t\,\textup{grad}^{h}(f)_{x}\,,\label{j'ecoute silent hill2}
	\end{eqnarray}
	where $X_{f\circ\pi}$ is the symplectic gradient of 
	$f\circ \pi\,:\,TM\rightarrow\mathbb{R}$ with respect to $\omega$ 
	and where $\varphi^{X_{f\circ\pi}}_{t}$ denotes the flow of $X_{f\circ\pi}\,.$ 
\end{lemma}
\begin{proof}
	For $A_{u_{x}}\in T_{u_{x}}TM\,,$ we have by definition of $\omega$ 
	(see \eqref{equation definition G, omega, etc.}) and $\textup{grad}^{h}(f)\,:$
	\begin{eqnarray}
		\bullet	&&(f\circ\pi)_{*_{u_{x}}}A_{u_{x}}
				=\omega\big((X_{f\circ\pi})_{u_{x}},A_{u_{x}}\big)\nonumber\\
			&& =h\big(\pi_{*_{u_{x}}}(X_{f\circ\pi})_{u_{x}},KA_{u_{x}}\big)-
				h\big(\pi_{*_{u_{x}}}A_{u_{x}},K(X_{f\circ\pi})_{u_{x}}\big)\,,
				\label{equation que diiiiire???}\\
		\bullet	&&(f\circ \pi)_{*_{u_{x}}}A_{u_{x}}
			=f_{*_{x}}\pi_{*_{u_{x}}}A_{u_{x}}=h\big(\textup{grad}^{h}(f)_{x},
			\pi_{*_{u_{x}}}A_{u_{x}}\big)\,.\label{equation petit cafe?}
	\end{eqnarray}
	Comparing \eqref{equation que diiiiire???} and \eqref{equation petit cafe?}, we get 
	$\pi_{*_{u_{x}}}(X_{f\circ\pi})_{u_{x}}=0$ and $K(X_{f\circ\pi})_{u_{x}}=-\textup{grad}^{h}(f)_{x}\,,$
	which, in view of the identification given in \eqref{equation Dombrowski}, implies 
	the first equation in \eqref{equation le gradient, il est la}. 

	The second equation in \eqref{j'ecoute silent hill2} is an easy consequence 
	of the first. 
%	\eqref{equation le gradient, il est la} 
%	together with the fact that the right hand side of \eqref{j'ecoute silent hill2} defines 
%	a $1$-parameter group of diffeomorphims. 
	The lemma follows. 
\end{proof}
\begin{remark}\label{rrr mal au bide}
	A simple consequence of Lemma \ref{lemma j'ecoute la chanson etrange} 
	is that $\mathscr{K}(M)\,,$ viewed as a Lie subalgebra 
	of $\mathscr{K}(TM)\,,$ is commutative.
\end{remark}
\begin{proof}[Proof of Proposition \ref{proposition equivalence kahler parallel}]
	In this proof, we use the identification given in \eqref{equation Dombrowski} as well as the notation
	introduced in Lemma \ref{lemma j'ecoute la chanson etrange}. 
	
	Let $A_{u_{x}}=(u_{x},v_{v},w_{x})$ and $\overline{A}_{u_{x}}=
	({u}_{x},\overline{v}_{x},\overline{w}_{x})$ be two tangent vectors in $T_{u_{x}}TM\,.$
	Since $(\pi\circ \varphi^{X_{f\circ\pi}}_{t})(u_{x})=\pi(u_{x}-t\,\textup{grad}^{h}(f)_{x})=x$ 
	(see Lemma \ref{lemma j'ecoute la chanson etrange}),
	$\pi_{*}(\varphi^{X_{f\circ\pi}}_{t})_{*_{u_{X}}}A_{u_{x}}=\pi_{*_{u_{x}}}(u_{x},v_{x},w_{x})=v_{x}$ 
	and thus, recalling the definition of $g$ given in \eqref{equation definition G, omega, etc.}, 
	\begin{eqnarray}
		&&\Big((\varphi^{X_{f\circ\pi}}_{t})^{*}g\Big)_{u_{x}}(A_{u_{x}},\overline{A}_{u_{x}})
			=g\Big((\varphi^{X_{f\circ\pi}}_{t})_{*_{u_{x}}}A_{u_{x}},
			(\varphi^{X_{f\circ\pi}}_{t})_{*_{u_{x}}}\overline{A}_{u_{x}}\Big)\nonumber\\
		&=&h(v_{x},\overline{v}_{x})+h\Big(K(\varphi^{X_{f\circ\pi}}_{t})_{*_{u_{x}}}A_{u_{x}},
			K(\varphi^{X_{f\circ\pi}}_{t})_{*_{u_{x}}}\overline{A}_{u_{x}}\Big)\,.
			\label{equation chanson j'aime moins}
	\end{eqnarray}
	We have to compute $K(\varphi^{X_{f\circ\pi}}_{t})_{*_{u_{x}}}{A}_{u_{x}}\,.$ For this, 
	observe that
	\begin{eqnarray}
		&& K(\varphi^{X_{f\circ\pi}}_{t})_{*_{u_{x}}}{A}_{u_{x}}=
			K(\varphi^{X_{f\circ\pi}}_{t})_{*_{u_{x}}}(u_{x},v_{x},w_{x})\nonumber\\
		&&=K(\varphi^{X_{f\circ\pi}}_{t})_{*_{u_{x}}}(u_{x},v_{x},0)	
			+K(\varphi^{X_{f\circ\pi}}_{t})_{*_{u_{x}}}(u_{x},0,w_{x})\,.
			\label{equation encore chansonssss}
	\end{eqnarray}
	The second term in \eqref{equation encore chansonssss} is easily computed:
	\begin{eqnarray}
		K(\varphi^{X_{f\circ\pi}}_{t})_{*_{u_{x}}}(u_{x},0,w_{x})=K\dfrac{d}{ds}\bigg\vert_{0}\,
			\varphi^{X_{f\circ\pi}}_{t}(u_{x}+sw_{x})
			=K\dfrac{d}{ds}\bigg\vert_{0}\,(u_{x}+sw_{x}-t\,\textup{grad}^{h}(f)_{x})=w_{x}\,.
			\label{trop belle cette chanson!!!}
	\end{eqnarray}
	For the first term in \eqref{equation encore chansonssss}, we introduce a curve $V(s)$ in $TM$ such that
	$dV(s)/ds\vert_{0}=(u_{x},v_{x},0)$ and such that $V(s)$ is horizontal for all $s\,.$ Observe that 
	$d(\pi\circ V(s))/ds\vert_{0}=\pi_{*}(u_{x},v_{x},w_{x})=v_{x}\,.$ Using this curve, we see that 
	\begin{eqnarray}
		&&K(\varphi^{X_{f\circ\pi}}_{t})_{*_{u_{x}}}(u_{x},v_{x},0)=K\dfrac{d}{ds}\bigg\vert_{0}\,
			\varphi^{X_{f\circ\pi}}_{t}\big(V(s)\big)=
			K\dfrac{d}{ds}\bigg\vert_{0}\,\big(V(s)-t\,\textup{grad}^{h}(f)_{\pi(V(s))}\big)\nonumber\\
			&=&\nabla_{v_{x}}\big(V(s)-t\,\textup{grad}^{h}(f)_{\pi(V(s))}\big)
				=\nabla_{v_{x}}V(s)-t\nabla_{v_{x}}\textup{grad}^{h}(f)_{\pi(V(s))}\nonumber\\
		&=&	-t\nabla_{v_{x}}\textup{grad}^{h}(f)_{\pi(V(s))}\,.\label{vraiment trop bellllllle}
	\end{eqnarray}
	Now, \eqref{equation chanson j'aime moins}, \eqref{equation encore chansonssss}, 
	\eqref{trop belle cette chanson!!!} and \eqref{vraiment trop bellllllle} yield
	\begin{eqnarray}
		\Big((\varphi^{X_{f\circ\pi}}_{t})^{*}g\Big)_{u_{x}}(A_{u_{x}},\overline{A}_{u_{x}})&=&
			g(A_{u_{x}},\overline{A}_{u_{x}})-t\Big(h\big(\nabla_{v_{x}}\textup{grad}^{h}(f),\overline{w}_{x}\big)+
			h\big(\nabla_{\overline{v}_{x}}\textup{grad}^{h}(f),{w}_{x}\big)\Big)\nonumber\\
		&& +\,\,t^{2}h\big(\nabla_{v_{x}}\textup{grad}^{h}(f),\nabla_{\overline{v}_{x}}\textup{grad}^{h}(f)\big)
	\end{eqnarray}
	from which we clearly see that $\varphi^{X_{f\circ\pi}}_{t}$ is an isometry for all $t$ if and only if 
	$\textup{grad}^{h}(f)$ is $\nabla$-parallel. The proposition follows. 
\end{proof}	

	Let us now specialize to the case of an exponential family. So, let $\mathcal{E}$ 
	be an exponential family defined on a measured space $(\Omega,dx)$ 
	with elements of the form $p(x;\theta)=
	\textup{exp}\big\{C(x)+\sum_{i=1}^{n}\,\theta_{i}F_{i}(x)-\psi(\theta)\big\}$ as in 
	\eqref{equation definiton exp}, and let us consider the following space of functions:
	\begin{eqnarray}\label{eee definition les observables...}
		\mathcal{A}_{\mathcal{E}}:=
			\textup{Vect}_{\mathbb{R}}\big\{1,F_{1},...,F_{n}\big\}\,,
	\end{eqnarray}
	i.e., $\mathcal{A}_{\mathcal{E}}$ is the real vector space generated by the constant function 1 
	and the functions $F_{1},...,F_{n}\,:\,\Omega\rightarrow\mathbb{R}\,.$
\begin{proposition}[]\label{proposition c'est chouette nom}
	For a function $f\,:\,\mathcal{E}\rightarrow\mathbb{R}\,,$ we have:
	\begin{eqnarray}
		f\,\,\textup{is a K\"{a}hler function}\,\,\,\,\,\Leftrightarrow\,\,\,\,\,
		\Big(\exists \, X\in \mathcal{A}_{\mathcal{E}}\,:\,f(p)=\int_{\Omega}\,X(x)\,p(x)\,dx\,\,\,
		\forall p\in \mathcal{E}\Big).
	\end{eqnarray}
\end{proposition}
	In order to show Proposition \ref{proposition c'est chouette nom}, we need the following lemma.
\begin{lemma}\label{bientot le week end!!!}
	The expectation parameters $\eta_{i}\,:\,\mathcal{E}\rightarrow \mathbb{R}\,,p\mapsto
	\int_{\Omega}\,F_{i}(x)p(x)dx$ satisfy the following relation:
	\begin{eqnarray}
		\textup{grad}^{h_{F}}(\eta_{i})=\partial_{\theta_{i}}\,.
	\end{eqnarray}
	In particular, expectation parameters are K\"{a}hler functions. 
\end{lemma}
\begin{proof}
	Using the duality between the natural and expectation parameters (see 
	\eqref{equation etrange, et curieuse}), one easily sees that 
	$\partial_{\eta_{i}}=\sum_{k=1}^{n}\,(h_{F})^{ik}\partial_{\theta_{k}}\,,$ where the $(h_{F})^{ik}$'s
	are the coefficients of the inverse of the matrix 
	$(h_{F})_{ij}:=h_{F}(\partial_{\theta_{i}},\partial_{\theta_{j}})\,,$ and also that 
	$h_{F}(\partial_{\eta_{i}},\partial_{\eta_{j}})=(h_{F})^{ij}\,.$ It follows that  
	\begin{eqnarray}
		\textup{grad}^{h_{F}}(\eta_{i})&=&\sum_{a,b=1}^{n}\,(h_{F})_{ab}
			\dfrac{\partial \eta_{i}}{\partial_{\eta_{a}}}\partial_{\eta_{b}}=
			\sum_{a,b=1}^{n}\,(h_{F})_{ab}\delta_{ia}\partial_{\eta_{b}}
			=\sum_{b=1}^{n}\,(h_{F})_{ib}\partial_{\eta_{b}}\nonumber\\
		&=&\sum_{b,k=1}^{n}\,(h_{F})_{ib}(h_{F})^{bk}\partial_{\theta_{k}}
			=\sum_{k=1}^{n}\,\delta_{ik}\partial_{\theta_{k}}=\partial_{\theta_{i}}\,,
	\end{eqnarray}
	which is the desired relation. 
\end{proof}
\begin{proof}[Proof of Proposition \ref{proposition c'est chouette nom}]
	Let $f\,:\,\mathcal{E}\rightarrow \mathbb{R}$ be a function. If $f$ is a K\"{a}hler function, then according 
	to Proposition \ref{proposition equivalence kahler parallel}, $\textup{grad}^{h_{F}}(f)$ 
	is $\nabla^{(e)}$-parallel, which means that when expressed in the 
	$\nabla^{(e)}$-affine chart $\theta=(\theta_{1},...,\theta_{n})\,,$ $\textup{grad}^{h_{F}}(f)$ 
	is a constant vector field. We can thus write 
	$\textup{grad}^{h_{F}}(f)=\sum_{i=1}^{n}a_{i}\partial_{\theta_{i}}\,,$ where 
	$a_{1},...,a_{n}$ are some real constants. But then, according to 
	Lemma \ref{bientot le week end!!!} and the definition of $\eta_{i}\,,$ 
	\begin{eqnarray}
		&&\textup{grad}^{h_{F}}(f)=\sum_{i=1}^{n}a_{i}\partial_{\theta_{i}}=
			\sum_{i=1}^{n}a_{i}\textup{grad}^{h_{F}}(\eta_{i})
			=\textup{grad}^{h_{F}}\Big(\sum_{i=1}^{n}\,a_{i}\eta_{i}\Big)\nonumber\\
		&&=\textup{grad}^{h_{F}}\Big(\sum_{i=1}^{n}\,a_{i}\int_{\Omega}\,F_{i}(x)p(x)dx\Big)
			=\textup{grad}^{h_{F}}\Big(\int_{\Omega}\,\sum_{i=1}^{n}\,a_{i}F_{i}(x)p(x)dx\Big)\,.
	\end{eqnarray}
	Hence, and up to an additive constant, $f(p)=\int_{\Omega}\,\sum_{i=1}^{n}\,a_{i}F_{i}(x)p(x)dx$
	which shows one direction of the proposition. 
	The other direction being trivial, the proposition follows. 
\end{proof}
	A direct consequence of Proposition \ref{proposition c'est chouette nom} is the following result:
%\begin{corollary}\label{rrr pouette poutte}
%	Let $f\,:\,T\mathcal{E}\rightarrow\mathbb{R}$ be a smooth function and let 
%	$\Phi\,:\,T\mathcal{E}\rightarrow T\mathcal{E}$ be a holomorphic isometry. If $f$ factorizes 
%	through the map $\pi\circ\Phi\,:\,T\mathcal{E}\rightarrow \mathcal{E}\,,$ i.e., 
%	if $f$ can be written $f=\overline{f}\circ\pi \circ\Phi$ for some smooth function 
%	$\overline{f}\,:\,\mathcal{E}\rightarrow\mathbb{R}\,,$
%	then
%	\begin{eqnarray}
%		f\,\,\textup{is K\"{a}hler}\,\,\,\,\,\Leftrightarrow\,\,\,\,\,
%		\Big(\exists \, X\in \mathcal{A}_{\mathcal{E}}\,:\,\overline{f}(p)=
%		\int_{\Omega}\,X(x)\,p(x)\,dx\Big)\,.
%	\end{eqnarray} 
%\end{corollary}
%	In particular,  
\begin{corollary}\label{ccc les fonctions de la formes...}
	Functions of the form
	\begin{eqnarray}\label{remark la forme au'il faut}
		T\mathcal{E}\rightarrow\mathbb{R}\,,\,\,\,\,
		z\mapsto \int_{\Omega}\,X(x)\big[(\pi\circ \Phi)(z)\big](x)dx\,,	
	\end{eqnarray} 
	where $X\in \mathcal{A}_{\mathcal{E}}$ and where $\Phi\,:\,T\mathcal{E}\rightarrow T\mathcal{E}$ 
	is a holomorphic isometry, are K\"{a}hler functions on $T\mathcal{E}\,.$
\end{corollary}

%	Since we assume the expectation parameters $\eta_{i}$ to form a global chart 
%	$\eta=(\eta_{1},...,\eta_{n})$ of $\mathcal{E}\,,$ it is not hard to see that 
%	if a K\"{a}hler function $f\,:\,\mathcal{E}\rightarrow\mathbb{R}$ 
%	can be written $f(p)=\int_{\Omega}\,X(x)p(x)dx$ for some $X\in \mathcal{A}\,,$ 
%	then $X$ is unique. We can thus, adopting a statistical terminology, define 
%	the set of \textit{observed values} of 
%	a K\"{a}hler function $f\,:\,\mathcal{E}\rightarrow\mathbb{R}$ 
%	as the subset of $\mathbb{R}$ given by the image $\textup{Im}(X)$ of $X\,.$ 
%
%	More generally, we shall adopt the following 
%	definition:
%\begin{definition}\label{definiton spectre}
%	Let $f\,:\,T\mathcal{E}\rightarrow\mathbb{R}$ be a K\"{a}hler function which can be written as 
%	in \eqref{remark la forme au'il faut}. If among all such possible decompositions of $f\,,$
%	the sets $\textup{Im}(X)$ are all equal, 
%	then we shall say that 
%	$\textup{Im}(X)$ is the set of \textit{observed values} of $f\,,$ and we shall 
%	write  
%	\begin{eqnarray}
%		\textup{obv}(f):=\textup{Im}(X)\,.
%	\end{eqnarray}
%\end{definition}
%	As we will see in \S\ref{section quantum mechanics}, by an appropriate choice of
%	exponential family, it is possible to realize in a natural way the spectrum 
%	of any given Hermitian matrix as a set of observable values; this will lead us to 
%	the standard Hilbert space formalism of quantum mechanics. 

\section{K\"{a}hlerification of an exponential family}\label{section il y en a pas deux}

	Let $\mathcal{E}$ be an exponential family and let $(g,J,\omega)$ be the K\"{a}hler 
	structure of $T\mathcal{E}\,.$ 
	We define a subgroup $\Gamma(\mathcal{E})$ 
	of the group of all diffeomorphisms $\textup{Diff}(T\mathcal{E})$  of 
	$T\mathcal{E}$ by letting 
	\begin{eqnarray}\label{eee definition gamma}
		\Gamma(\mathcal{E}):=\big\{\phi\in\textup{Diff}(T\mathcal{E})\,\big\vert\,
		\phi^{*}g=g\,,\,\,\phi_{*}\,J=J\,\phi_{*}\,\,\textup{and}\,\,f\circ \phi
		=f\,\,\textup{for all}\,\,f\in \mathscr{K}(T\mathcal{E})\big\}.
		\,\,\,\,\,\,\,\,\,\,\,\,\,\textbf{}
	\end{eqnarray}
%	The group $\Gamma(\mathcal{E})$ is the largest subgroup of $\textup{Diff}(T\mathcal{E})$ 
%	which leaves invariant the K\"{a}hler structure of $T\mathcal{E}$ as well as the space of 
%	K\"{a}hler functions. 

\begin{definition}\label{definition kahlerification}
	Let $\mathcal{E}$ be an exponential family having a discrete $\Gamma(\mathcal{E})$ and whose 
	natural action\footnote{The natural action of $\Gamma(\mathcal{E})$ on 
	$T\mathcal{E}$ is simply given by 
	$\gamma\cdot u_{x}:=\gamma(u_{x})\,,$ where $\gamma\in \Gamma(\mathcal{E})$ and 
	$u_{x}\in T\mathcal{E}\,.$}on $T\mathcal{E}$ is free and proper. The quotient space  
	$T\mathcal{E}/\Gamma(\mathcal{E})$ is thus naturally a K\"{a}hler manifold for which 
	the quotient map $T\mathcal{E}\rightarrow T\mathcal{E}/\Gamma(\mathcal{E})$ becomes 
	a holomorphic Riemannian submersion. We shall call 
	this quotient the K\"{a}hlerification of $\mathcal{E}\,,$ and use the following notation: 
	\begin{eqnarray}
		\mathcal{E}^{\mathbb{C}}:=T\mathcal{E}/\Gamma(\mathcal{E})\,.
	\end{eqnarray}
\end{definition}
	At this point, it is worth mentioning that 
	in all the examples considered in this paper, the group $\Gamma(\mathcal{E})$ is
	discrete and that its natural action on $T\mathcal{E}$ is free and proper. 
	In the sequel, we will always assume that the exponential families under consideration 
	fulfill the conditions of Definition \ref{definition kahlerification}.\\ 
	
	Let us investigate the geometrical structure of $\mathcal{E}^{\mathbb{C}}\,.$
\begin{lemma}\label{lemme ca commute}
	For every $\gamma\in \Gamma(\mathcal{E})\,,$ we have
	\begin{eqnarray}
		\pi\circ \gamma=\pi\,,
	\end{eqnarray}
	where $\pi\,:\,T\mathcal{E}\rightarrow \mathcal{E}$ is the canonical projection. 
\end{lemma}
\begin{proof}
	Since expectation parameters $\eta_{i}\,:\,\mathcal{E}\rightarrow \mathbb{R}$ are 
	K\"{a}hler functions (see Lemma \ref{bientot le week end!!!}), and since  
	$\eta=(\eta_{1},...,\eta_{n})$ is a chart of $\mathcal{E}\,,$ we have by definition 
	of $\Gamma(\mathcal{E})\,,$
	\begin{eqnarray}
		&&\eta_{i}\circ \pi\circ\gamma=\eta_{i}\circ\pi\,\,\,\,\,\textup{for all}\,\,i\nonumber\\
		&\Rightarrow& \eta\big(\pi(\gamma(x))\big)=\eta\big(\pi(x)\big)\,\,\,\,\,\textup{for all}\,\,
			x\in T\mathcal{E}\,\nonumber\\
		&\Rightarrow& \pi\circ \gamma=\pi\,.
	\end{eqnarray}
	This is the desired relation. 
\end{proof}
	Lemma \ref{lemme ca commute} readily implies that the projection 
	$\pi\,:\,T\mathcal{E}\rightarrow\mathcal{E}$ factorizes through $\mathcal{E}^{\mathbb{C}}\,,$ 
	yielding a submersion $\mathcal{E}^{\mathbb{C}}\rightarrow \mathcal{E}$ that 
	we shall denote by $\pi_{\mathcal{E}}\,,$ or simply $\pi\,.$ A K\"{a}hlerification has thus the structure 
	of a fiber bundle induced by the submersion 
	\begin{eqnarray}
		\pi_{\mathcal{E}}\,:\,\mathcal{E}^{\mathbb{C}}\rightarrow\mathcal{E}\,,
	\end{eqnarray}
	whose fiber over $p\in \mathcal{E}$ is diffeomorphic to 
	$T_{p}\mathcal{E}/\Gamma(\mathcal{E})\,.$\\

	Clearly, $\pi_{\mathcal{E}}\,:\,\mathcal{E}^{\mathbb{C}}\rightarrow \mathcal{E}$ is a Riemannian submersion. Also, 
	$\mathscr{K}(T\mathcal{E})\cong \mathscr{K}(\mathcal{E}^{\mathbb{C}})$ (Lie algebra isomorphism), and   
	consequently there are analogues of Proposition \ref{proposition c'est chouette nom} and Corollary \ref{ccc les fonctions de la formes...} 
	for the space of K\"{a}hler functions 
	on $\mathcal{E}^{\mathbb{C}}\,.$

\section{K\"{a}hlerification of $\mathcal{P}_{n}^{\times}$ and complex projective spaces}
	\label{section the motivating example}\label{section khalerif jskjfdf}

%	As we already discussed on several occasions, a quantum system can be described geometrically 
%	by means of the K\"{a}hler structure of the complex projective space $\mathbb{P}(\mathcal{H})$
%	associated to the Hilbert space $\mathcal{H}$ of quantum states. This is the so-called 
%	\textit{geometrical formulation} of quantum mechanics (see \cite{Ashtekar} and below).
%	In this section, we review some aspects of this formulation in the particular
%	case $\mathcal{H}=\mathbb{C}^{n}\,.$ Our presentation is more statistically oriented than 
%	in \cite{Ashtekar} and heavily relies on a projection from the complex projective space 
%	to a statistical model. We will also use symmetries the appropriate
%	geometric mechanical tools, such as developed in \cite{Marsden-Ratiu}. \\

%	In this section, we show that the K\"{a}herification of $\mathcal{P}_{n}^{\times}$ 
%	is an open dense subset of the complex projective space $\mathbb{P}(\mathbb{C}^{n})$ 
%	of complex lines in $\mathbb{C}^{n}\,.$ Our derivation is based on a geometric result given in 
%	\cite{Molitor-quantique} (see Proposition \ref{proposition la clé du bordel} below).\\

	Recall from Example \ref{example finite omega} that $\mathcal{P}_{n}^{\times}$ is the space of 
	non-vanishing probability density functions $p$ defined on a finite set $\Omega=\{x_{1},...,x_{n}\}\,,$ 
	i.e., 
	\begin{eqnarray}
		\mathcal{P}_{n}^{\times}:=\Big\{p\,:\,\Omega\rightarrow \mathbb{R}\,\,\Big\vert\,\,
		p(x_{i})>0\,\,\,\textup{for all}\,\,\,x_{i}\in \Omega\,\,\,\textup{and}\,\,\,
		\sum_{i=1}^{n}p(x_{i})=1\Big\}\,.
	\end{eqnarray}
	
	This space is clearly a connected manifold of dimension $n-1\,,$ and as we 
	already saw, $\mathcal{P}_{n}^{\times}$ is an exponential family.
	
	In the context of information geometry, it is 
	customary to describe the tangent bundle of $\mathcal{P}_{n}^{\times}$ using the 
	\textit{exponential representation} :
	\begin{eqnarray}
		T_{p}\mathcal{P}_{n}^{\times}\cong \{u=(u_{1},...,u_{n})\in \mathbb{R}^{n}\,
		\vert\,u_{1}\,p_{1}+\ldots + u_{n}\,p_{n}=0\}\,,
	\end{eqnarray}
	where $p\in \mathcal{P}_{n}^{\times}\,,$ and where by definition, $p_{i}:=p(x_{i})$ 
	for all $x_{i}\in \Omega\,.$\\
	If $u\in \mathbb{R}^{n}$ is a vector satisfying  $u_{1}\,p_{1}+\ldots+u_{n}\,p_{n}=0$ 
	for a given probability density function $p\,:\,\Omega\rightarrow\mathbb{R}\,,$ 
	then we shall denote by $[u]_{p}$ the unique tangent vector 
	of $\mathcal{P}_{n}^{\times}$ at the point $p$ determined by the exponential representation.
	One easily sees that if $p(t)$ is a smooth curve in $\mathcal{P}_{n}^{\times}\,,$ 
	then
	\begin{eqnarray}\label{equation derivée plus identification expo}
	\frac{d}{dt}\bigg\vert_{0}p(t)=[u]_{p(0)}\,\,\,\,\,\Leftrightarrow\,\,\,\,\,
	\dfrac{d}{dt}\bigg\vert_{0}\,p_{i}(t)=p_{i}(0)\,u_{i}
	\,\,\,\text{for all}\,\,i=1,\cdots,n\,,
	\end{eqnarray}
	where $p_{i}(t):=\big(p(t)\big)(x_{i})\,.$\\
	Equation \eqref{equation derivée plus identification expo} is actually one way 
	to define the exponential representation. 

	In term of the exponential representation, the Fisher metric $h_{F}$ 
	has the following expression :
	\begin{eqnarray}\label{eee epression explicite fisher}
		(h_{F})_{p}([u]_{p},[v]_{p})=\sum_{i=1}^{n}p_{i}u_{i}v_{i}\,,
	\end{eqnarray}
	while the covariant derivative $D^{(e)}[V]_{p(t)}/dt$ of a vector field $[V]_{p(t)}$ along a curve 
	$p\,:I\subseteq \mathbb{R}\rightarrow \mathcal{P}_{n}^{\times}$ with respect to 
	the exponential connection $\nabla^{(e)}$ is given by 
	\begin{eqnarray}\label{eee bochum bientot}
		\dfrac{D^{(e)}}{dt}\big[V(t)\big]_{p(t)}=
		\big[\,\dot{V}(t)-E_{p(t)}\big(\dot{V}(t)\big)\,\big]_{p(t)}\,,
	\end{eqnarray}
	where $E_{p(t)}\big(\dot{V}(t)\big):=p_{1}(t)\,\dot{V}_{1}(t)+\cdots+ p_{n}(t)\,\dot{V}(t)$ 
	is the mean of the vector $\dot{V}(t)=\big(\dot{V}_{1}(t),...,\dot{V}_{n}(t)\big)$ 
	with respect to the probability 
	$p(t)$ and where $E_{p(t)}\big(\dot{V}(t)\big)$ in \eqref{eee bochum bientot} as to be understood as the vector 
	$\big(E_{p(t)}\big(\dot{V}(t)\big),...,E_{p(t)}\big(\dot{V}(t)\big)\big)\in \mathbb{R}^{n}\,.$ \\
	
	For later purposes, let us also give the following result which gives an 
	explicit description of the inverse of the map 
	$\Phi_{}\,:\,T(T\mathcal{P}_{n}^{\times})\rightarrow  
	T\mathcal{P}_{n}^{\times}\oplus T\mathcal{P}_{n}^{\times}
	\oplus T\mathcal{P}_{n}^{\times}$ introduced in \eqref{equation Dombrowski}.
	The proof may be found in \cite{Molitor-quantique}.
\begin{lemma}\label{lemme formule inverse identification}
	For $[u]_{p},[v]_{p},[w]_{p}
	\in T_{p}\mathcal{P}_{n}^{\times}\,,$ we have : 
	\begin{eqnarray}
		\Phi_{}^{-1}\big([u]_{p},[v]_{p},[w]_{p}\big) = 
		\dfrac{d}{dt}\bigg\vert_{0}\,
		\Big[u+tw-E_{p(t)}(u+tw)\Big]_{p(t)}\,,
	\end{eqnarray}
	where $p(t)$ is a smooth curve in $\mathcal{P}_{n}^{\times}$ satisfying $p(0)=p$ and 
	$dp(t)/dt\big\vert_{0}=[v]_{p}\,.$
\end{lemma}
	
	We now want to relate the natural K\"{a}hler 
	structure of $T\mathcal{P}_{n}^{\times}$ to the K\"{a}hler structure of the complex projective 
	space $\mathbb{P}(\mathbb{C}^{n})\,.$ To this end, and for the reader's convenience, 
	let us digress a little on $\mathbb{P}(\mathbb{C}^{n})\,.$

	Recall that the complex projective space $\mathbb{P}(\mathbb{C}^{n})$ is the quotient 
	$(\mathbb{C}^{n}-\{0\})/{\sim}\,,$ where the equivalence relation $``\sim"$ 
	is defined by 
	\begin{eqnarray}
		(z_{1},...,z_{n})\sim (w_{1},...,w_{n})\,\,\,\,\,\Leftrightarrow\,\,\,\,\,
		\exists \lambda\in \mathbb{C}-\{0\}: (z_{1},...,z_{n})=\lambda (w_{1},...,w_{n})\,.
	\end{eqnarray}
	For $z=(z_{1},...,z_{n})\in \mathbb{C}^{n}-\{0\}\,,$ we shall denote by $[z]=[z_{1},...,z_{n}]$
	the corresponding element of $\mathbb{P}(\mathbb{C}^{n})\,.$ One may 
	identify $[z]$ with the complex line $\mathbb{C}\cdot z\,.$

	The manifold structure of $\mathbb{P}(\mathbb{C}^{n})$ may be defined as follows. 
	For a vector $u=(u_{1},...,u_{n})\in \mathbb{C}^{n}$ such that 
	$|u|^{2}=\langle u,u\rangle=\overline{u}_{1}u_{1}+\cdots +\overline{u}_{n}u_{n}=1$ 
	(our convention for the Hermitian product $\langle\,,\,\rangle$ on $\mathbb{C}^{n}$
	is that $\langle\,,\,\rangle$ is linear in the second argument)\,,
	we define a chart $(U_{u},\phi_{u})$ of $\mathbb{P}(\mathbb{C}^{n})$ by letting
	\begin{eqnarray}\label{equation definition carte projective}
		\left \lbrace
			\begin{array}{cc}
				U_{u}:=\big\{[z]\in \mathbb{P}(\mathbb{C}^{n})\,\big\vert\, 
				[u]\cap [z]=\{0\}\big\}\,,\\
				\phi_{u}\,:\,U_{u}\rightarrow [u]^{\perp}\subseteq \mathbb{C}^{n}\,,
				[z]\mapsto\dfrac{1}{\langle u,z\rangle}\cdot z-u\,.
			\end{array}
		\right.
	\end{eqnarray}
	If $u$ varies among all the unit vectors in $\mathbb{C}^{n}\,,$ then the corresponding 
	charts $(U_{u},\phi_{u})$ form a holomorphic atlas for $\mathbb{P}(\mathbb{C}^{n})\,;$ the projective 
	space is thus a manifold of real dimension $2(n-1)\,,$ and using the above charts 
	we have the identification
	\begin{eqnarray}
		T_{[u]}\mathbb{P}(\mathbb{C}^{n})\cong [u]^{\perp}=\{w\in\mathbb{C}^{n}\,\big\vert\,
		\langle u,w\rangle=0\}\,.
	\end{eqnarray}
	
	The Fubini-Study metric $g_{FS}$ and the Fubini-Study symplectic form $\omega_{FS}$ 
	are now defined at the point $[u]$ in $\mathbb{P}(\mathbb{C}^{n})$ via the formulas : 
	\begin{eqnarray}\label{equation definition de fubini-study}
		\Big((\phi_{u}^{-1})^{*}g_{FS}\Big)_{0}(\xi_{1},\xi_{2}):=\textup{Re}\,
		\langle\xi_{1},\xi_{2}\rangle\,\,,\,\,\,\,\,\,
		\Big((\phi_{u}^{-1})^{*}\omega_{FS}\Big)_{0}(\xi_{1},\xi_{2}):=\textup{Im}\,
		\langle\xi_{1},\xi_{2}\rangle\,,
	\end{eqnarray}
	where $\xi_{1},\xi_{2}\in [u]^{\perp}\cong T_{[u]}\mathbb{P}(\mathbb{C}^{n})$ and where 
	``\textup{Re}" and ``\textup{Im}" stand for the real and imaginary parts.\\
	One may show that $g_{FS}$ and $\omega_{FS}$ are globally well defined
	on $\mathbb{P}(\mathbb{C}^{n})\,,$ and that $(g_{FS},J_{FS},\omega_{FS})$ is a K\"{a}hler 
	structure on $\mathbb{P}(\mathbb{C}^{n})\,,$ where $J_{FS}$ denotes the natural complex structure 
	of $\mathbb{P}(\mathbb{C}^{n})\,.$\\

	Now, consider the map
	\begin{eqnarray}\label{eee definition tau}
			\tau\,:\,T\mathcal{P}_{n}^{\times}\rightarrow \mathbb{P}(\mathbb{C}^{n})^{\times}\,,
			\,\,\,[u]_{p}\mapsto \big[\,\sqrt{p_{1}}\,e^{iu_{1}/2},...,
			\sqrt{p_{n}}\,e^{iu_{n}/2}\,\big]\,,
	\end{eqnarray}
	where $\mathbb{P}(\mathbb{C}^{n})^{\times}$ is the open subset of $\mathbb{P}(\mathbb{C}^{n})$ 
	defined by 
	\begin{eqnarray}
			\mathbb{P}(\mathbb{C}^{n})^{\times}:=
			\big\{\,[z_{1},...,z_{n}]\in \mathbb{P}(\mathbb{C}^{n})\,\big\vert\,z_{i}\neq 0\,\,\,
			\textup{for all}\,\,\,i=1,...,n\,\big\}\,.
	\end{eqnarray}
	In \cite{Molitor-quantique}, the following is shown:
\begin{proposition}[\cite{Molitor-quantique}]\label{proposition la clé du bordel}
	\textbf{}
	\begin{description}
		\item[$(i)$] The map $\tau\,:\,T\mathcal{P}_{n}^{\times}\rightarrow 
			\mathbb{P}(\mathbb{C}^{n})^{\times}$ is a universal covering map whose 
			deck transformation 
			group is isomorphic to $\mathbb{Z}^{n-1}\,,$
		\item[$(ii)$] the following relations hold :
			\begin{eqnarray}\label{equation origine de tout}
				\tau^{*}g_{FS}=\dfrac{1}{4}\,g\,,\,\,\,\,\,\tau^{*}\omega_{FS}=
				\dfrac{1}{4}\,\omega\,,\,\,\,\,\,
				\tau_{*}J=J_{FS}\,\tau_{*}\,,
			\end{eqnarray}
			where $(g,J,\omega)$ is the K\"{a}hler structure of $T\mathcal{P}_{n}^{\times}$ 
			associated to $(h_{F},\nabla^{(e)})$ 
			via Dombrowski's construction.
%		\item[$(iii)$] every deck transformation of $T\mathcal{P}_{n}^{\times}$ 
%			is a holomorphic isometry. 
	\end{description}
\end{proposition}
\begin{remark}
	In \cite{Molitor-quantique}, we were defining the Fisher metric $h_{F}$ as being 
	the one considered in this paper,
	but multiplied by a factor $1/4\,.$ Because of that, the first two formulas
	in \eqref{equation origine de tout} differ from the corresponding formulas  in 
	\cite{Molitor-quantique}
	by a factor $1/4\,.$
\end{remark}
\begin{remark}
%	In \cite{Molitor-quantique}, it is not explicitly stated that the deck transformations of 
%	$T\mathcal{P}_{n}^{\times}$ are holomorphic isometries, but this result is very easy to obtain 
%	using the material developed in \cite{Molitor-quantique}.
	Observe that every deck transformation of $T\mathcal{P}_{n}^{\times}$ has to be a holomorphic 
	isometry.
\end{remark}
	We can now state the main result of this section. 
\begin{proposition}\label{proposition le bon group indeed}
	The group $\Gamma(\mathcal{P}_{n}^{\times})$ coincides with the deck transformation group 
	of the universal covering map 
	$\tau\,:\,T\mathcal{P}_{n}^{\times}\rightarrow \mathbb{P}(\mathbb{C}^{n})^{\times}\,.$ In particular, 
	if we multiply both the Fubini-Study metric $g_{FS}$ and the Fubini-Study symplectic form 
	$\omega_{FS}$ by a factor 4, then 
	we get a natural identification of K\"{a}hler manifolds: 
	\begin{eqnarray}
		(\mathcal{P}_{n}^{\times})^{\mathbb{C}}\cong \mathbb{P}(\mathbb{C}^{n})^{\times}\,.
	\end{eqnarray}
	Moreover, in term of the above identification, the canonical projection 
	$\pi_{\mathcal{P}_{n}^{\times}}\,:\,(\mathcal{P}_{n}^{\times})^{\mathbb{C}}\rightarrow\mathcal{P}_{n}^{\times}$ becomes
\begin{eqnarray}\label{equation defintion de pipipi}
		\pi_{\mathcal{P}_{n}^{\times}}\,:\,\mathbb{P}(\mathbb{C}^{n})^{\times}\rightarrow \mathcal{P}_{n}^{\times}\,,\,\,\,\,
		\pi_{\mathcal{P}_{n}^{\times}}([z])(x_{k}):=\dfrac{z_{k}\overline{z_{k}}}{\langle z,z\rangle}\,.
\end{eqnarray}
\end{proposition}
	We will show Proposition \ref{proposition le bon group indeed} with a series of lemmas.

\begin{lemma}\label{lemma comomentum map}
	Let $\mathscr{K}\big(\mathbb{P}(\mathbb{C}^{n})\big)$ be the space of K\"{a}hler functions on 
	$\mathbb{P}(\mathbb{C}^{n})$ and let $\mathfrak{u}(n)$ be the space of complex $n\times n$ skew Hermitian 
	matrices. If $\mathscr{K}\big(\mathbb{P}(\mathbb{C}^{n})\big)$ is endowed with its natural 
	Poisson bracket $\{f,g\}:=\omega_{FS}(X_{f},X_{g})\,,$ then the map 
	$\mathfrak{u}(n)\rightarrow \mathscr{K}\big(\mathbb{P}(\mathbb{C}^{n})\big)\,,
		\,\,A\mapsto \xi^{A}\,,$ where 
	\begin{eqnarray}\label{eee cafe encore}
		\xi^{A}([z]):=\dfrac{i}{2}\dfrac{\langle z,A\cdot z \rangle}{\langle z,z\rangle}\,,\,\,\,\,\,\,(z\in \mathbb{C}^{n}-\{0\})
	\end{eqnarray}
	is a Lie algebra isomorphism. 
\end{lemma}
\begin{proof}
	See for example \cite{Cirelli-Quantum}. 
\end{proof}	
\begin{lemma}\label{lemma comment le montrer?}
	Let $[z],[w]$ be two points in $\mathbb{P}(\mathbb{C}^{n})\,.$ We have:
	\begin{eqnarray}
		[z]=[w]\,\,\,\,\,\Leftrightarrow\,\,\,\,\,f([z])=f([w])\,\,\,\textup{for all}\,\,\,f\in
		\mathscr{K}\big(\mathbb{P}(\mathbb{C}^{n})\big)\,.
	\end{eqnarray}
\end{lemma}
\begin{proof}
	Let $z,w\in \mathbb{C}^{n}-\{0\}$ be two vectors. Using 
	Proposition \ref{lemma comomentum map} and especially \eqref{eee cafe encore}, it is easy 
	to see that if $f([z])=f([w])$ for all K\"{a}hler functions $f$ on 
	$\mathbb{P}(\mathbb{C}^{n})\,,$ then   
	\begin{eqnarray}\label{eee manger?! or not?}
		2\,\textup{Re}\Big(\sum_{a<b}\,A_{ab}(\overline{z}_{a}z_{b}-\overline{w}_{a}w_{b})\Big)
		+\sum_{a}A_{aa}(|z_{a}|^{2}-|w_{a}|^{2})=0\,,	
	\end{eqnarray}
	where $A_{ab}$ ($a,b=1,...,n\,,$ $a<b$)  are arbitrary complex numbers, 
	and where $A_{aa}$ ($a=1,...n$)  are arbitrary real numbers
	(one may think of $A_{ab}$ as the coefficients of a Hermitian matrix). 
	From \eqref{eee manger?! or not?}, we deduce that 
	\begin{eqnarray}
		\overline{z}_{a}z_{b}=\overline{w}_{a}w_{b}
	\end{eqnarray}
	for all $a,b=1,...,n\,.$ By introducing polar decompositions
	and with some algebraic manipulations, it is then easy to see that $z$ and $w$ 
	are collinear. The lemma follows.  
\end{proof}
\begin{lemma}\label{lemma la decomposition qu'il faut entre blabla...}
	A function $f\,:\,T\mathcal{P}_{n}^{\times}\rightarrow \mathbb{R}$ is a K\"{a}hler function if and only 
	if there exists a K\"{a}hler function $\overline{f}\,:\,
	\mathbb{P}(\mathbb{C}^{n})\rightarrow\mathbb{R}$ such that 
	\begin{eqnarray}
		f=\overline{f}\circ \tau\,.
	\end{eqnarray}
\end{lemma}
\begin{proof}
	Let $f\,:\,T\mathcal{P}_{n}^{\times}\rightarrow\mathbb{R}$ be a K\"{a}hler function.
	Since $\tau\,:\,T\mathcal{P}_{n}^{\times}\rightarrow \mathbb{P}(\mathbb{C}^{n})^{\times}$ 
	is a covering map, for every $z\in T\mathcal{P}_{n}^{\times}\,,$ there exists an open and connected 
	set $U_{z}\subseteq T\mathcal{P}_{n}^{\times}$ containing $z$ and 
	such that the restriction of $\tau$ to $U_{z}$ becomes a diffeomorphism between $U_{z}$ and 
	$\tau(U_{z})\,.$ 
	Let us denote this restriction by $\tau\vert_{U_{z}}\,.$ According to $(ii)$ in Proposition 
	\ref{proposition la clé du bordel}, $\tau\vert_{U_{z}}$ is a holomorphic isometry; this implies that 
	the map $f\circ (\tau\vert_{U_{z}})^{-1}\,:\,\tau(U_{z})\rightarrow \mathbb{R}$
	is a K\"{a}hler function, which means in particular that $X_{f\circ (\tau\vert_{U_{z}})^{-1}}$ is a 
	Killing vector field on $\tau(U_{z})\,.$ But now, since $\mathbb{P}(\mathbb{C}^{n})$ 
	is a connected, simply connected and complete (in the Riemannian sense) K\"{a}hler manifold, 
	their exists a Killing vector field 
	on $\mathbb{P}(\mathbb{C}^{n})$ extending 
	$X_{f\circ (\tau\vert_{U_{z}})^{-1}}$ (see \cite{Nomizu}),
	and this Killing vector field is bound\footnote{It is well known that 
	every Killing vector field on $\mathbb{P}(\mathbb{C}^{n})$ can be realized as the Hamiltonian vector field 
	of an appropriate K\"{a}hler function.} to be a Hamiltonian vector field 
	$X_{f_{U_{z}}}$ for some K\"{a}hler function 
	$f_{U_{z}}\,:\,\mathbb{P}(\mathbb{C}^{n})\rightarrow\mathbb{R}\,,$ the latter being defined 
	only up to an additive constant. By choosing this constant appropriately, we thus get a
	K\"{a}hler function $f_{U_{z}}$ such that 
	\begin{eqnarray}
		f\vert_{U_{z}}=(f_{U_{z}}\circ \tau)\vert_{U_{z}}\,.
	\end{eqnarray}
	The above formula shows that the statement in the lemma is locally true. Let us now show that 
	it is also globally true. So let $U,V\subseteq T\mathcal{P}_{n}^{\times}$ be two connected 
	open sets whose intersection is not empty, and such that there exist two K\"{a}hler functions 
	$f_{U},f_{V}\,:\,\mathbb{P}(\mathbb{C}^{n})\rightarrow\mathbb{R}$ verifying 
	$f\vert_{U}=(f_{U}\circ \tau)\vert_{U}$ and $f\vert_{V}=(f_{V}\circ \tau)\vert_{V}\,.$
	Since $f_{U}\circ \tau$ and $f_{U}\circ \tau$ coincide on the intersection of $U$ and $V\,,$
	there exists a connected open subset of $\mathbb{P}(\mathbb{C}^{n})$ on 
	which $f_{U}$ and $f_{V}$ coincide.
	But now, since K\"{a}hler functions on 
	$\mathbb{P}(\mathbb{C}^{n})$ are of the form $f^{A}$ 
	(see Lemma \ref{lemma comomentum map}), it is clear from \eqref{eee cafe encore} 
	that if two K\"{a}hler functions 
	on the complex projective space coincide on an open subset, then they are equal. This 
	implies $f_{U}=f_{V}$ from which the lemma follows.
\end{proof}
	Since every deck transformation of $T\mathcal{P}_{n}^{\times}$ 
	is a holomorphic isometry, it follows from 
	Lemma \ref{lemma la decomposition qu'il faut entre blabla...} that the deck transformation group
	of the universal covering map 
	$\tau\,:\,T\mathcal{P}_{n}^{\times}\rightarrow \mathbb{P}(\mathbb{C}^{n})^{\times}$ 
	is a subgroup of $\Gamma(\mathcal{P}_{n}^{\times})\,.$ The converse is given by the 
	following lemma. 
\begin{lemma}\label{lemma necessairement!!!}
	An element $\gamma\in \Gamma(\mathcal{P}_{n}^{\times})$ is necessarily a deck transformation. 
\end{lemma}
\begin{proof}
	Let $\gamma$ be an element of $\Gamma(\mathcal{P}_{n}^{\times})\,.$ By definition of 
	$\Gamma(\mathcal{P}_{n}^{\times})\,,$ and taking into account Lemma \ref{lemma comment le montrer?} 
	and Lemma \ref{lemma la decomposition qu'il faut entre blabla...}, we see that  
	\begin{eqnarray}
		 \Big(f\circ\tau\circ\gamma=f\circ \tau\,\,\,\,\,\,\,\forall\,\,f\in 
			\mathscr{K}\big(\mathbb{P}(\mathbb{C}^{n})\big)\Big)\,\,\,\,\,
		\Rightarrow\,\,\,\,\, \tau\circ \gamma =\tau\,,
	\end{eqnarray}
	i.e., $\gamma$ is a deck transformation. The lemma follows. 
\end{proof}
	Equation \eqref{equation defintion de pipipi} being straightforward, 
	Proposition \ref{proposition le bon group indeed} is now a direct consequence of Lemma 
	\ref{lemma la decomposition qu'il faut entre blabla...} and Lemma \ref{lemma necessairement!!!}.\\

	Let us end this section with a few important remarks on the map 
	$\pi_{\mathcal{P}_{n}^{\times}}\,:\,\mathbb{P}(\mathbb{C}^{n})^{\times}
	\rightarrow \mathcal{P}_{n}^{\times}\,.$ 
	Clearly, $\pi_{\mathcal{P}_{n}^{\times}}$ extends uniquely as a continuous map 
	$\pi_{\mathcal{P}_{n}}\,:\,\mathbb{P}(\mathbb{C}^{n})\rightarrow \mathcal{P}_{n}\,,$
	where\footnote{We endow $\mathcal{P}_{n}$ with the topology induced by $\mathbb{R}^{n}$ via 
	the injection $\mathcal{P}_{n}\hookrightarrow \mathbb{R}^{n}\,,\,\,p\mapsto \big(p(x_{1}),...,p(x_{n})\big)\,.$}
	\begin{eqnarray}
		\mathcal{P}_{n}:=\Big\{p\,:\,\Omega\rightarrow \mathbb{R}\,\,\Big\vert\,\,
		p(x_{i})\geq 0\,\,\,\textup{for all}\,\,\,x_{i}\in \Omega\,\,\,\textup{and}\,\,\,
		\sum_{i=1}^{n}p(x_{i})=1\Big\}\,. 
	\end{eqnarray}
	Notice that $\mathcal{P}_{n}^{\times}\subseteq \mathcal{P}_{n}\,,$ and that 
	these two spaces are distinguished only by the conditions $p>0$ for $\mathcal{P}_{n}^{\times}$ 
	and $p\geq 0$ for $\mathcal{P}_{n}\,.$ Notice also that $\mathcal{P}_{n}$
	is not a smooth manifold, for it has a boundary and corners.

	The map $\pi_{\mathcal{P}_{n}}\,:\,\mathbb{P}(\mathbb{C}^{n})\rightarrow \mathcal{P}_{n}$ has the property 
	that it makes the following diagram commutative
	\begin{eqnarray}
		\xymatrix{
   			\mathbb{P}(\mathbb{C}^{n})^{\times} \ar[r]^{\pi_{\mathcal{P}_{n}^{\times}}} \ar[d]_i & 
				\mathcal{P}_{n}^{\times} \ar[d]^j \\
   			\mathbb{P}(\mathbb{C}^{n}) \ar[r]_{{\pi_{\mathcal{P}_{n}}}} & \mathcal{P}_{n}
  				}	
	\end{eqnarray}
	($i,j$ are inclusions), which allows to carry over many structural properties of the 
	space of K\"{a}hler functions on 
	$\mathbb{P}(\mathbb{C}^{n})^{\times}$ to the space of K\"{a}hler functions on 
	$\mathbb{P}(\mathbb{C}^{n})\,.$ Indeed, let $\mathscr{K}(\mathcal{P}_{n})$ denotes the following 
	space of functions
	\begin{eqnarray}
		\Big\{f\,:\,\mathcal{P}_{n}\rightarrow\mathbb{R}\,\big\vert\,f(p)=
		\sum_{k=1}^{n}X_{k}p(x_{k})\,,\,X=(X_{1},...,X_{n})\in\mathbb{R}^{n}\Big\}\,.
	\end{eqnarray}
	By using Corollary \ref{ccc les fonctions de la formes...}, 
	Lemma \ref{lemma la decomposition qu'il faut entre blabla...}, some obvious continuity 
	arguments and the fact that
	$\mathcal{A}_{\mathcal{P}_{n}^{\times}}=\{X\,:\,\Omega\rightarrow 
	\mathbb{R}\}\cong\mathbb{R}^{n}$ (see the general definition of $\mathcal{A}_{\mathcal{E}}$ given 
	in \eqref{eee definition les observables...} and Example \ref{example finite omega}),
	one easily shows the following: 
\begin{proposition}\label{lll nouvel methodologie!!!}
	We have:
	\begin{description}
		\item[$(i)$] $\mathscr{K}\big(\mathbb{P}(\mathbb{C}^{n})^{\times}\big)\cong
			\mathscr{K}\big(\mathbb{P}(\mathbb{C}^{n})^{}\big)\,,$\,\,
			(Lie algebra isomorphism),
		\item[$(ii)$] $\mathscr{K}(\mathcal{P}_{n}^{\times})\cong 
			\mathscr{K}(\mathcal{P}_{n})\,,$
		\item[$(iii)$] functions on $\mathbb{P}(\mathbb{C}^{n})$ of the form 
			$f\circ \pi_{\mathcal{P}_{n}}\circ \Phi\,,$ where 
			$\Phi\,:\,\mathbb{P}(\mathbb{C}^{n})\rightarrow\mathbb{P}(\mathbb{C}^{n})$ is a holomorphic isometry 
			and where $f\in \mathscr{K}(\mathcal{P}_{n})\,,$ are K\"{a}hler functions.  
	\end{description}
\end{proposition}
	We see from Proposition \ref{lll nouvel methodologie!!!} that the map 
	$\pi_{\mathcal{P}_{n}}\,:\,\mathbb{P}(\mathbb{C}^{n})\rightarrow \mathcal{P}_{n}$
	behaves like the canonical projection of a K\"{a}hlerification. It is thus a natural ``completion" of the 
	map $\pi_{\mathcal{P}_{n}^{\times}}\,:\,\mathbb{P}(\mathbb{C})^{\times}\rightarrow \mathcal{P}_{n}^{\times}\,,$ 
	and \textit{formally} we have $(\mathcal{P}_{n})^{\mathbb{C}}=\mathbb{P}(\mathbb{C}^{n})\,.$\\

\section{K\"{a}hlerification and the geometrical formulation of quantum mechanics}\label{section quantum mechanics}

	The goal of this section is to rederive the geometrical formulation of quantum mechanics in finite dimension 
	(based on the K\"{a}hler properties of $\mathbb{P}(\mathbb{C}^{n})$ as in \cite{Ashtekar}), using a statistically oriented 
	approach through the equation $(\mathcal{P}_{n}^{\times})^{\mathbb{C}}=\mathbb{P}(\mathbb{C}^{n})^{\times}$ 
	and its formal version $(\mathcal{P}_{n})^{\mathbb{C}}=\mathbb{P}(\mathbb{C}^{n})\,.$

	This exercise is necessary, for we want to express all the relevant quantities of the geometrical formulation 
	in terms of statistical concepts, aiming to generalize them to situations where $\mathcal{P}_{n}^{\times}$ is replaced by 
	a more general exponential family $\mathcal{E}\,.$

%	In this section, we rederive the whole geometrical formulation of quantum mechanics 
%	(such as formulated for example in \cite{Ashtekar}), aiming, by means of the equation 
%	$(\mathcal{P}_{n}^{\times})^{\mathbb{C}}=\mathbb{P}(\mathbb{C}^{n})^{\times}\,,$ 
%	to introduce statistics into its very core. This exercise will allows us to generalize the geometrical formulation 
%	to situations where $\mathcal{P}_{n}^{\times}$ is replaced by an exponential family $\mathcal{E}\,.$ 

	\text{}\\
	Let us start with the following ``statistical" characterization of K\"{a}hler functions on the complex 
	projective space (see also Corollary \ref{ccc les fonctions de la formes...}).
\begin{proposition}[]
	\label{ppp decomposition spectrale geometrique statistique}
	Let $f\,:\,\mathbb{P}(\mathbb{C}^{n})\rightarrow\mathbb{R}$ be a smooth function. Then, $f$ is 
	a K\"{a}hler function if and only if there exist a random variable $X\,:\,\Omega=\{x_{1},...,x_{n}\}\rightarrow \mathbb{R}$ 
	and an unitary matrix $U\in U(n)$ such that 
	\begin{eqnarray}\label{eee pouete il fait trop chaud j'en ai marre}
		f([z])=\int_{\Omega}\,X(x)\,\big[(\pi_{\mathcal{P}_{n}}\circ \Phi_{U})([z])\big](x)dx\,,
	\end{eqnarray}
	where $\pi_{\mathcal{P}_{n}}\,:\,\mathbb{P}(\mathbb{C}^{n})\rightarrow\mathcal{P}_{n}$ is the map 
	considered at the end of \S\ref{section khalerif jskjfdf}, and where $\Phi_{U}$ is the 
	holomorphic isometry of $\mathbb{P}(\mathbb{C}^{n})$ defined by $\Phi_{U}([z])=\big[U\cdot z\big]\,.$
\end{proposition}
\begin{proof}

	Simply use Lemma \ref{lemma comomentum map} together with the usual spectral decomposition 
	theorem. 
%	Le $f\,:\,\mathbb{P}(\mathbb{C}^{n})\rightarrow\mathbb{R}$ be a K\"{a}hler function. 
%	According to Proposition \ref{lemma comomentum map}, there exists $A\in \mathfrak{u}(n)$ such that 
%	$f=\textbf{\textup{J}}^{A}\,.$ Thus, and in view of \eqref{equation que faire...}, 
%	we have for all $[z]\in \mathbb{P}(\mathbb{C}^{n})\,,$
%	\begin{eqnarray}\label{eee pouettttteeeefff}
%		f([z])=\dfrac{i}{2}\,
%		\dfrac{\langle z,A\cdot z\rangle}{\langle z\,,z\rangle}\,.
%	\end{eqnarray}
%	Now, by application of the spectral theorem, 
%	there exist an unitary matrix $U\in U(n)$ and a random variable 
%	$X\,:\,\Omega\rightarrow\mathbb{R}$ 
%	such that $(i/2)A=U^{*}\textup{diag}(X)U\,,$ where $\textup{diag}(X)$ is the diagonal matrix 
%	whose k$th$ diagonal entry is $X(x_{k})\,,$ and where $U^{*}$ denotes the adjoint of $U\,.$
%	Using this decomposition, \eqref{eee pouettttteeeefff} may be rewritten, after some manipulations, 
%	as 
%	\begin{eqnarray}
%		f([z])=\sum_{k=1}^{n}\,X(x_{k})\dfrac{(\overline{Uz})_{k}(Uz)_{k}}{\langle Uz,Uz\rangle}\,.
%	\end{eqnarray}
%	This last expression is exactly  
%	\eqref{eee pouete il fait trop chaud j'en ai marre} (see the definition of $\pi$ and $\Phi_{U}$ 
%	given in \eqref{equation defintion de pipipi} and \eqref{eee definition de phiiii} 
%	respectively). The proposition follows. 
\end{proof}
	Proposition \ref{ppp decomposition spectrale geometrique statistique} implies that 
	every K\"{a}hler function on $\mathbb{P}(\mathbb{C}^{n})$ can be realized as an 
	expectation of the form $f([z])=E_{(\pi\circ \Phi_{U})([z])}(X)\,,$ where 
	$X\,:\,\Omega\rightarrow\mathbb{R}$ is a random variable. The image $\textup{Im}(X)$ of $X$
	has thus an important statistical meaning that we would like to relate with the usual 
	\textit{spectrum} of $f$ as defined in the geometrical formulation of quantum mechanics. 

	To this end, we could use Lemma \ref{lemma comomentum map}, 
	Proposition \ref{ppp decomposition spectrale geometrique statistique}, 
	and relate $\textup{Im}(X)$ with the usual spectrum of an appropriate Hermitian matrix. 
	We prefer, however, to use a generalization of a statistical result that we now present. 
	
	Recall that if $X\,:\,\Omega\rightarrow\mathbb{R}$ is a random variable, and if 
	$p\in\mathcal{P}_{n}$ is a given probability, then the 
	variance of $X$ with respect to $p$ is given by $V_{p}(X)=E_{p}\big((X-E_{p}(X))^{2}\big)\,.$
\begin{proposition}[Cram\'{e}r-Rao equality]\label{ppp cramer rao}
	Let $f([z])=\int_{\Omega}\,X(x)[(\pi_{\mathcal{P}_{n}}\circ\Phi_{U})([z])](x)dx$ be a K\"{a}hler function on 
	$\mathbb{P}(\mathbb{C}^{n})\,.$ For all $[z]\in \mathbb{P}(\mathbb{C}^{n})\,,$ we have:
	\begin{eqnarray}\label{eee le bar le burau bis!}
		V_{(\pi_{\mathcal{P}_{n}}\circ\Phi_{U})([z])}(X)=\dfrac{1}{4}\|\,\textup{grad}(f)_{[z]}\|^{2}\,,
	\end{eqnarray}
	where $\textup{grad}(f)$ denotes the Riemannian gradient of $f$ with respect to 
	the Fubini-Study metric $g_{FS}\,.$ 
\end{proposition}  
\begin{proof}
	If $f\,:\,\mathbb{P}(\mathbb{C}^{n})\rightarrow\mathbb{R}$ is a (non-necessarily K\"{a}hler) function, 
	and if $\phi\,:\,\mathbb{P}(\mathbb{C}^{n})\rightarrow\mathbb{P}(\mathbb{C}^{n})$ is an isometry, 
	then, for $[z]\in\mathbb{P}(\mathbb{C}^{n})\,,$ we have:
	\begin{eqnarray}
		\textup{grad}(f\circ \phi)_{[z]}=(\phi^{-1})_{*}(\textup{grad}(f))_{\phi([z])}\,.
	\end{eqnarray} 
	From this formula, and the fact that $\Phi_{U}$ is an isometry, 
	we see that it is sufficient to show the proposition for $U=I_{n}\,.$ So let 
	$X\,:\,\Omega\rightarrow\mathbb{R}$ be a random variable and assume that 
	$f([z])=\sum_{k=1}^{n}X(x_{k})\pi_{\mathcal{P}_{n}}([z])(x_{k})\,.$ 

	Taking into account 
	Proposition \ref{proposition la clé du bordel}, we see that the function 
	$\overline{f}:=f\circ\tau\,:\,T\mathcal{P}_{n}^{\times}\rightarrow\mathbb{R}$ satisfies 
	\begin{eqnarray}\label{eee il me gonfle ce type}
		\tau_{*_{[u]_{p}}}\textup{grad}(\overline{f})_{[u_{p}]}=\dfrac{1}{4}\textup{grad}(f)_{\tau([u]_{p})}
		\,\,\,\,\textup{and}\,\,\,\,
		\|\,\textup{grad}(\overline{f})_{[u]_{p}}\|^{2}=\dfrac{1}{4}\|\,\textup{grad}(f)_{\tau([u]_{p})}\|^{2}\,, 
	\end{eqnarray}
	where $\textup{grad}(\overline{f})$ denotes the Riemannian gradient of $\overline{f}$ with respect to the 
	Riemannian metric $g$ (see the definition of $g$ in Proposition \ref{proposition la clé du bordel}). 
	Moreover, for $A_{[u]_{p}}=([u]_{p},[v]_{p},[w]_{p})=d/dt\vert_{0}[u+tw-E_{p(t)}(u+tw)]_{p(t)}$ 
	as in Lemma \ref{lemme formule inverse identification}, and taking into account the relation 
	$(\pi_{\mathcal{P}_{n}}\circ\tau)([u]_{p})=p\,,$ we see that 
	\begin{eqnarray}
		&&g_{[u]_{p}}\big(\textup{grad}(\overline{f})_{[u]_{p}},A_{[u]_{p}}\big)
			=\overline{f}_{*_{[u]_{p}}}A_{[u]_{p}}
			=\dfrac{d}{dt}\bigg\vert_{0}\, \overline{f}\big(\big[u+tw-E_{p(t)}(u+tw)\big]_{p(t)}\big)
			\nonumber\\
		&& =\dfrac{d}{dt}\bigg\vert_{0}\,\sum_{k=1}^{n}\,p(t)(x_{k})X(x_{k})=
			\sum_{k=1}^{n}\,p(x_{k})v_{k}X(x_{k})
			=(h_{F})_{p}\big([v]_{p},[X-E_{p}(X)]_{p}\big)+(h_{F})_{p}\big([w]_{p},0\big)\nonumber\\
		&& = g_{[u]_{p}}\Big(\big([u]_{p},[X-E_{p}(X)]_{p},[0]_{p}\big),
			\big([u]_{p},[v]_{p},[w]_{p}\big)\Big)\,,
	\end{eqnarray}
	and thus, 
	\begin{eqnarray}
		\textup{grad}(\overline{f})_{[u]_{p}}=\big([u]_{p},[X-E_{p}(X)]_{p},[0]_{p}\big)\,.
	\end{eqnarray}
	From this equation, and taking into account \eqref{eee epression explicite fisher}, we get
	\begin{eqnarray}
		&&\|\textup{grad}(\overline{f})_{[u]_{p}}\|^{2}=E_{p}\big((X-E_{p}(X))^{2}\big)=V_{p}(X)\,.
	\end{eqnarray}
	The proposition is now a consequence of this last equation together 
	with \eqref{eee il me gonfle ce type} and the fact that 
	$\tau\big(T\mathcal{P}_{n}^{\times}\big)$ is dense in $\mathbb{P}(\mathbb{C}^{n})\,.$
\end{proof}
\begin{remark} 
	Proposition \ref{ppp cramer rao} is a direct generalization of a formula 
	which is well known in the context of information geometry, 
	namely\footnote{The absence of the factor 
	$1/4$ in \eqref{eee nar le bureau} compared to \eqref{eee le bar le burau bis!} is due to the normalizing factor 
	of the Fubini-Study metric $g_{FS}$ used throughout this section (see also Proposition 
	\ref{proposition la clé du bordel}).}
		\begin{eqnarray}\label{eee nar le bureau}
		\|\,\textup{grad}\big(E_{p}(X)\big)\|^{2}=V_{p}(X)\,,
	\end{eqnarray}
	where $E_{p}(X)$ denotes the function $\mathcal{P}_{n}^{\times}\rightarrow\mathbb{R}\,,
	p\mapsto E_{p}(X)$ ($X\,:\,\Omega\rightarrow\mathbb{R}$ being 
	a given random variable), and where the norm and the Riemannian gradient are taken 
	with respect to the Fisher metric $h_{F}\,.$
	The above formula is sometimes called Cram\'{e}r-Rao equality for, it allows 
	to recover the usual Cram\'{e}r-Rao inequality, the latter being, roughly, an inequality 
	which gives a ``lower bound" for the variance-covariance matrix of an unbiased 
	estimator on a given statistical model $S$ (see \cite{Amari-Nagaoka} for details). 
\end{remark}
\begin{remark}
	More generally, if $\mathcal{E}$ is an exponential family whose elements are of the form 
	$p(x;\theta)=\textup{exp}\big\{C(x)+\sum_{i=1}^{n}\,\theta_{i}F_{i}(x)-\psi(\theta)\big\}$ on a fixed measured space 
	$(\Omega,dx)$ (see Definition \ref{definition exp}), 
	and if $X\,:\,\Omega\rightarrow \mathbb{R}$ is a linear combinaition of $1,F_{1},...,F_{n}\,,$ then 
	one can easily show the following identity 
	\begin{eqnarray}
		\|\,\textup{grad}\big(E_{p}(X)\big)\|^{2}=V_{p}(X)\,,
	\end{eqnarray}
	where $E_{p}(X)$ is viewed as the function $\mathcal{E}\rightarrow \mathbb{R}\,,\,\,p\mapsto E_{p}(X)$ and where 
	the norm and the gradient are taken with respect to the Fisher metric $h_{F}\,.$
\end{remark}
	By inspection of the Cram\'{e}r-Rao equality (as formulated in Proposition \ref{ppp cramer rao}), 
	one deduces easily the following corollary:
\begin{corollary}
	Let $f([z])=\int_{\Omega}\,X(x)[(\pi_{\mathcal{P}_{n}}\circ\Phi_{U})([z])](x)dx$ be a K\"{a}hler function on 
	$\mathbb{P}(\mathbb{C}^{n})\,.$ Then a real number $\lambda$ belongs to 
	$\textup{Im}(X)$ if and only if $\lambda$ is a critical value of $f\,,$ i.e. if and only if 
	\begin{eqnarray}
		\exists\, [z]\in \mathbb{P}(\mathbb{C}^{n})\,\,
		\textup{such that}\,\,f_{*_{[z]}}=0\,\,\textup{and}\,\,f([z])=\lambda\,.
	\end{eqnarray}
\end{corollary}
	The above corollary implies that the set $\textup{Im}(X)$ doesn't depend on the particular 
	decomposition of $f$ given in Proposition \ref{ppp decomposition spectrale geometrique statistique}. 
	We can thus give the following definition: 
\begin{definition}
	The \textit{spectrum} of a K\"{a}hler function 
	$f([z])=\int_{\Omega}\,X(x)[(\pi_{\mathcal{P}_{n}}\circ\Phi_{U})([z])](x)dx$
	on $\mathbb{P}(\mathbb{C}^{n})$ is the subset of $\mathbb{R}$ given by 
	\begin{eqnarray}
		\textup{spec}(f):=\textup{Im}(X)\,.
	\end{eqnarray}
	This is the set of all critical values of $f\,.$
\end{definition}
	Following \cite{Ashtekar}, we shall call elements of $\textup{spec}(f)$ 
	\textit{eigenvalues} and the corresponding critical points \textit{eigenpoints}. These are the geometrical 
	analogues, in the geometrical formulation, of the usual eigenvalues and eigenvectors of Hermitian matrices used in the 
	standard formulation of quantum mechanics. \\

	Given a K\"{a}hler function $f([z])=\int_{\Omega}\,X(x)[(\pi_{\mathcal{P}_{n}}\circ\Phi_{U})([z])](x)dx$
	and a point $[z]\in \mathbb{P}(\mathbb{C}^{n})\,,$
	there is an obvious associated probability $P_{f,[z]}$ on $\textup{spec}(f)\,:$
	\begin{eqnarray}\label{eee definiton de la measure associee}
		P_{f,[z]}(A):=\int_{X^{-1}(A)}\,\big[(\pi_{\mathcal{P}_{n}}\circ \Phi_{U})([z])\big](x)dx\,,
	\end{eqnarray}
	where $A\subseteq\textup{spec}(f)$ is a given subset. This is the pushforward of 
	$\big[(\pi_{\mathcal{P}_{n}}\circ \Phi_{U})([z])\big](x)dx$ via the random variable $X\,:\,\Omega\rightarrow \mathbb{R}\,.$

	As for $\textup{spec}(f)\,,$ we would like to show that this probability doesn't depend 
	on the particular decomposition of $f$ given in 
	\eqref{ppp decomposition spectrale geometrique statistique}. To this end, 
	we introduce, for a given real number $\lambda\,,$ the following space:
	\begin{eqnarray}
		M_{f,\lambda}:=\big\{[z]\in \mathbb{P}(\mathbb{C}^{n})\,\big\vert\,f_{*_{[z]}}=0\,\,\,
	\textup{and}\,\,\,f([z])=\lambda\big\}\,.
	\end{eqnarray}
	Observe that if $\lambda\notin \textup{spec}(f)\,,$ then $M_{f,\lambda}=\emptyset\,.$
\begin{lemma}\label{lll description eigenmanifold}
	Let $\lambda\in \textup{spec}(f)$ be an eigenvalue of a given K\"{a}hler function 
	$f([z])=\int_{\Omega}\,X(x)[(\pi_{\mathcal{P}_{n}}\circ\Phi_{U})([z])](x)dx\,,$ and let us fix, for 
	notational convenience, some indices 
	$k_{1},...,k_{m(\lambda)}\in\{1,...,n\}$ so that we can write 
	$X^{-1}(\lambda)=\{x_{k_{1}},...,x_{k_{m(\lambda)}}\}\,.$
	Let us also denote by $\{e_{1},...,e_{n}\}$ the canonical basis for $\mathbb{C}^{n}\,.$
	Then,
	\begin{eqnarray}
		M_{f,\lambda}=\Big\{U^{*}\cdot \big[c_{1}\cdot e_{k_{1}}+...
		+c_{m(\lambda)}\cdot e_{k_{m(\lambda)}}\big]\in \mathbb{P}(\mathbb{C}^{n})\,
		\big|\,c_{1},...,c_{{m(\lambda)}}\in \mathbb{C}
		\Big\}\,.
	\end{eqnarray}
	In particular, $M_{f,\lambda}\cong \mathbb{P}\big(\mathbb{C}^{m(\lambda)}\big)\,.$
\end{lemma}
\begin{proof}
	Let $[z]$ be an element of $\mathbb{P}(\mathbb{C}^{n})\,.$ We have:
	\begin{eqnarray}
		&&[z]\in M_{f,\lambda}\,\,\,\Leftrightarrow\,\,\,\Big(f_{*_{[z]}}=0\,\,\,\textup{and}\,\,\,
			f([z])=\lambda\Big)\nonumber\\
		&\Leftrightarrow& V_{(\pi_{\mathcal{P}_{n}}\circ\Phi_{U})([z])}(X)=0\,\,\,\textup{and}\,\,\,
			E_{(\pi_{\mathcal{P}_{n}}\circ\Phi_{U})([z])}(X)=\lambda\nonumber\\
		&\Leftrightarrow& \sum_{k=1}^{n}\,(X(x_{k})-\lambda)^{2}(\pi_{\mathcal{P}_{n}}\circ\Phi_{U})([z])(x_{k})=0
			\,\,\,\textup{and}\,\,\,E_{(\pi_{\mathcal{P}_{n}}\circ\Phi_{U})([z])}(X)=\lambda\nonumber\\
		&\Leftrightarrow& 
			\left\lbrace
				\begin{array}{cc}
					&(X(x_{k})-\lambda)^{2}(\pi_{\mathcal{P}_{n}}\circ\Phi_{U})([z])(x_{k})=
						0\,\,\,\textup{for all}\,\,\,k\in\{1,...,n\}\\
					&\textup{and}\,\,\,E_{(\pi_{\mathcal{P}_{n}}\circ\Phi_{U})([z])}(X)=\lambda
				\end{array}
			\right.
			\nonumber\\
		&\Leftrightarrow&
			\left\lbrace
				\begin{array}{cc}
					&(\pi_{\mathcal{P}_{n}}\circ\Phi_{U})([z])(x_{k})=
						0\,\,\,\textup{for all}\,\,\,k\in\{1,...,n\}-
						\{k_{1},...,k_{m(\lambda)}\}\\
					&\textup{and}\,\,\,E_{(\pi_{\mathcal{P}_{n}}\circ\Phi_{U})([z])}(X)=\lambda
				\end{array}
			\right.
			\nonumber\\
		&\Leftrightarrow&(\pi_{\mathcal{P}_{n}}\circ\Phi_{U})([z])(x_{k})=
						0\,\,\,\textup{for all}\,\,\,k\in\{1,...,n\}-
						\{k_{1},...,k_{m(\lambda)}\}\,.
	\end{eqnarray}
	Now observe that for $k\in\{1,...,n\}\,,$
	\begin{eqnarray}
		(\pi_{\mathcal{P}_{n}}\circ\Phi_{U})([z])(x_{k})=
		\dfrac{|\langle U\cdot z,e_{k}\rangle|^{2}}{\langle z,z\rangle}=
		\dfrac{|\langle z,U^{*}e_{k}\rangle|^{2}}{\langle z,z\rangle}\,,
	\end{eqnarray}
	and thus, 
	\begin{eqnarray}
		[z]\in M_{f,\lambda}\,\,\,\Leftrightarrow\,\,\,
			\langle z,U^{*}e_{k}\rangle=0\,\,\,\textup{for all}\,\,\,k\in\{1,...,n\}-
			\{k_{1},...,k_{m(\lambda)}\}
	\end{eqnarray}
	from which the lemma follows. 
\end{proof}
	From Lemma \ref{lll description eigenmanifold}, we see that the cardinal 
	of $X^{-1}(\lambda)$ doesn't depend on the 
	decomposition of $f\,;$ we shall call this number the \textit{multiplicity} of the 
	eigenvalue $\lambda$ and denote it by $m(\lambda)\,.$ As we already saw,  
	$M_{f,\lambda}\cong \mathbb{P}\big(\mathbb{C}^{m(\lambda)}\big)\,.$ In \cite{Ashtekar}, $M_{f,\lambda}$ is called 
	the \textit{eigenmanifold} of $f$ associated to $\lambda\,.$\\
%\begin{remark}
%	In view of Proposition \ref{ppp cramer rao}, the set $M_{f,\lambda}$ may also be defined 
%	as
%	\begin{eqnarray}
%		M_{f,\lambda}=\big\{[z]\in \mathbb{P}(\mathbb{C}^{n})\,\big\vert\, 
%		V_{(\pi\circ\Phi_{U})([z])}(X)=0\,\,\,\textup{and}\,\,\,
%		E_{(\pi\circ\Phi_{U})([z])}(X)=\lambda\big\}\,.
%	\end{eqnarray}
%\end{remark}

	Now recall that the geodesic distance $d(\,,\,)$ on 
	$\mathbb{P}(\mathbb{C}^{n})$ induced by the Fubini-Study metric $g_{FS}$ is given, 
	for $[z],[w]\in \mathbb{P}(\mathbb{C}^{n})\,,$ by:
	\begin{eqnarray}
		d\big([z],[w]\big)=
		\textup{cos}^{-1}\bigg(\dfrac{|\langle z,w \rangle|}{\|z\|\cdot\|w\|}\bigg)\,.
	\end{eqnarray}
	The above formula together with the Cauchy-Schwarz inequality readily implies 
	the following lemma:
\begin{lemma}\label{lll la forme de la projection sur}
	Let $\lambda\in \textup{spec}(f)$ be an eigenvalue and let $[z]$ be a point in 
	$\mathbb{P}(\mathbb{C}^{n})\,.$ Then their exists a unique point $p_{M_{f,\lambda}}([z])$ 
	in $M_{f,\lambda}$ verifying
	\begin{eqnarray}
		d\Big(p_{M_{f,\lambda}}([z]),[z]\Big)<d\big([w],[z]\big)
	\end{eqnarray}
	for all $[w]\in M_{f,\lambda}$ such that $[w]\neq p_{M_{f,\lambda}}([z])\,.$ 
	Moreover, if $X^{-1}(\lambda)=\{x_{k_{1}},...,x_{k_{m(\lambda)}}\}\,,$ then 
	\begin{eqnarray}
		p_{M_{f,\lambda}}([z])=\bigg[\sum_{l=1}^{m(\lambda)}\,\langle U^{*}e_{k_{l}},z \rangle
		\cdot U^{*}e_{k_{l}}\bigg]\,.
	\end{eqnarray}
\end{lemma}
	Clearly, if $[z]\in \mathbb{P}(\mathbb{C}^{n})\,,$ then 
	$d\big([z],p_{M_{f,\lambda}}([z])\big)$ is the geodesic distance between 
	$[z]$ and the subset $M_{f,\lambda}\,;$ we shall write
	$d([z],M_{f,\lambda})=d\big([z],p_{M_{f,\lambda}}([z])\big)\,.$ 
\begin{proposition}\label{ppp formule entre proba et cos}
	Let $f([z])=\int_{\Omega}\,X(x)[(\pi_{\mathcal{P}_{n}}\circ\Phi_{U})([z])](x)dx$ be a K\"{a}hler function,
	and let $\lambda\in \textup{spec}(f)$ be an eigenvalue. For 
	$[z]\in\mathbb{P}(\mathbb{C}^{n})\,,$ we have:
	\begin{eqnarray}\label{eee la mesure independante}
		\int_{X^{-1}(\lambda)}\,\big[(\pi_{\mathcal{P}_{n}}\circ \Phi_{U})([z])\big](x)dx=
		\cos^{2}\Big(d\big([z],M_{f,\lambda}\big)\Big)\,.
	\end{eqnarray}
\end{proposition}
\begin{proof}
	Let $z\in \mathbb{C}^{n}$ be a normalized vector, and assume that 
	$X^{-1}(\lambda)=\{x_{k_{1}},...,x_{k_{m(\lambda)}}\}\,.$ 
	If $z_{\lambda}:=
	\sum_{l=1}^{m(\lambda)}\,\langle U^{*}e_{k_{l}},z \rangle\cdot U^{*}e_{k_{l}}\,,$ then 
	clearly $\langle z_{\lambda},z\rangle=\langle z_{\lambda},z_{\lambda}\rangle\,,$ 
	and according to Lemma \ref{lll la forme de la projection sur}, $[z_{\lambda}]=p_{M_{f,\lambda}}([z])$
	and $d([z],[z_{\lambda}])=d([z],M_{f,\lambda})\,.$ Hence, 
	\begin{eqnarray}
		&&\int_{X^{-1}(\lambda)}\,\big[(\pi_{\mathcal{P}_{n}}\circ \Phi_{U})([z])\big](x)dx =
			\sum_{l=1}^{m(\lambda)}\,(\pi_{\mathcal{P}_{n}}\circ \Phi_{U})([z])(x_{k_{l}})
			=\sum_{l=1}^{m(\lambda)}\,\big\vert\langle U\cdot z,e_{k_{l}}\rangle \big \vert^{2}\nonumber\\
			&&=\sum_{l=1}^{m(\lambda)}\,\big\vert\langle z,U^{*}\cdot e_{k_{l}}\rangle \big\vert^{2}
			=\big\vert \langle z_{\lambda},z \rangle  \big\vert
			=\bigg(\dfrac{\big\vert \langle z_{\lambda},z \rangle  \big\vert}
			{\|z_{\lambda}\|\cdot \|z\|}\bigg)^{2}
			=\bigg(\cos\Big(d\big([z],p_{M_{f,\lambda}}([z])\big)\Big)\bigg)^{2}
	\end{eqnarray}
	which is exactly \eqref{eee la mesure independante}. The proposition follows. 
\end{proof}
	A direct consequence of Proposition \ref{ppp formule entre proba et cos} is that 
	the measure $P_{f,[z]}$ on $\textup{spec}(f)$ 
	defined in \eqref{eee definiton de la measure associee} doesn't 
	depend on a particular decomposition of $f$ such as in 
	\eqref{ppp decomposition spectrale geometrique statistique}.\\

	With the above proposition, we have completed our ``statistical" study of the geometrical formulation of quantum mechanics 
	in finite dimension. 

	The important points are : the configuration space $\mathbb{P}(\mathbb{C}^{n})$ is the (formal) 
	K\"{a}hlerification of $\mathcal{P}_{n}$\,; observables are K\"{a}hler functions 
	$f\,:\,\mathbb{P}(\mathbb{C}^{n})\rightarrow \mathbb{R}$ that can be decomposed 
	as $f([z])=\int_{\Omega}\,X(x)[(\pi_{\mathcal{P}_{n}}\circ\Phi_{U})([z])](x)dx\,,$ 
	where $X\in \mathcal{A}_{\mathcal{P}_{n}^{\times}}$ and where $\Phi_{U}$ is a holomorphic 
	isometry of $\mathbb{P}(\mathbb{C})\,;$ the spectrum of a K\"{a}hler function is $\textup{Im}(X)$ and its associated 
	probability is $P_{f,[z]}(\lambda)=\int_{X^{-1}(\lambda)}\,X(x)[(\pi_{\mathcal{P}_{n}}\circ\Phi_{U})([z])](x)dx\,.$

	As we see, these quantities depends only on the exponential structure of 
	$\mathcal{P}_{n}^{\times}$\footnote{Although being physically clear, this statement has still, on mathematical grounds, 
	to be clarified since one has to make precise the passage from $\mathbb{P}(\mathbb{C}^{n})^{\times}$ to 
	its natural ``completion" $\mathbb{P}(\mathbb{C}^{n})\,.$ In this paper, we shall not treat this technical question, preferring 
	to focus on the general procedure and the physical applications.}\,. 
	Thus, we can try to generalize them to a given exponential family, as follows. \\
	
	Let $\mathcal{E}$ be an exponential family defined on a measured space $(\Omega,dx)$ with elements of the form $p(x;\theta)=
	\textup{exp}\big\{C(x)+\sum_{i=1}^{n}\,\theta_{i}F_{i}(x)-\psi(\theta)\big\}\,,$ 
	$\mathcal{A}_{\mathcal{E}}:=\textup{Vect}_{\mathbb{R}}\{1,F_{1},...,F_{n}\}$ and
	$\pi_{\mathcal{E}}\,:\,\mathcal{E}^{\mathbb{C}}\rightarrow \mathcal{E}$ its associated K\"{a}hlerification. 

	Regarding $\mathcal{E}$ as the underlying statistical model of a ``generalized quantum system", we are led to the following definitions :
%	set the following 
%	list whose items (arguably) generalize the main aspects of the geometrical formulation of quantum mechanics in finite dimension :
	\begin{description}
		\item[$\bullet$] \textbf{Configuration space} : $\mathcal{E}^{\mathbb{C}}\,,$ viewed as a K\"{a}hler manifold,
		\item[$\bullet$] \textbf{Observables} : this is the set of functions 
			$f\,:\,\mathcal{E}^{\mathbb{C}}\rightarrow \mathbb{R}$ of the form 
			\begin{eqnarray}\label{eee j'ai ma montre enfin!!}
				\mathcal{E}^{\mathbb{C}}\rightarrow \mathbb{R}\,,\,\,\,
				z\mapsto \int_{\Omega}\,X(x)\big[(\pi_{\mathcal{E}}\circ \Phi)(z)\big](x)dx\,,
			\end{eqnarray}
			where $X\in \mathcal{A}_{\mathcal{E}}$ and where 
			$\Phi$ is a holomorphic isometry of $\mathcal{E}^{\mathbb{C}}\,.$ Such a function is necessarily 
			a K\"{a}hler function according to Proposition \ref{ccc les fonctions de la formes...},
		\item[$\bullet$] \textbf{Dynamics} : it is given by the flow of the Hamiltonian vector field $X_{H}$ associated to a 
			given observable $H\,:\,\mathcal{E}^{\mathbb{C}}\rightarrow \mathbb{R}$ with respect to the natural symplectic form 
			of the K\"{a}hler manifold $\mathcal{E}^{\mathbb{C}}\,,$
		\item[$\bullet$] \textbf{Spectrum of an observable} : the spectrum of an observable $f$ as in 
			\eqref{eee j'ai ma montre enfin!!} is
			given by the image of the random variable $X\,,$
			\begin{eqnarray}\label{eee perceuse}
				\textup{spec}(f):=\textup{Im}(X)\,,
			\end{eqnarray}
		\item[$\bullet$] \textbf{Probabilities associated to an observable} : the probability that an 
			observable $f$ as in \eqref{eee j'ai ma montre enfin!!} yields upon 
			measurement an eigenvalue belonging to a subset $A\subseteq \textup{spec}(f)$ while the system is in the state 
			$z\in \mathcal{E}^{\mathbb{C}}$ is :
			\begin{eqnarray}\label{eee les proba et la choucroute}
				P_{f,z}(A):=\int_{X^{-1}(A)}\,\big[(\pi_{\mathcal{E}}\circ \phi)(z)\big](x)dx\,.
			\end{eqnarray}
	\end{description}
	\begin{remark}
		Usually, the decomposition of a K\"{a}hler function $f$ as in \eqref{eee j'ai ma montre enfin!!} is not unique, and thus 
		$\textup{spec}(f)$ and $P_{f,z}$ are only well defined when invariance properties are met. 
	\end{remark}
%	\begin{remark}
%		Observe that we are not discussing the problem of the \textit{state vector reduction}. 
%	\end{remark}
	Of course, and from a physical point of view, the above definitions cannot be taken too literally. For 
	example when $\mathcal{E}=\mathcal{P}_{n}^{\times}\,,$ 
	then it is not $(\mathcal{P}_{n}^{\times})^{\mathbb{C}}\cong 
	\mathbb{P}(\mathbb{C}^{n})^{\times}\cong (\mathbb{C}^{*})^{n-1}$  
	which is interesting, rather its ``completion" $\mathbb{P}(\mathbb{C}^{n})\,,$ the formal K\"{a}hlerification of $\mathcal{P}_{n}\,.$ 
	Another similar example is the space of binomial distributions $\mathcal{B}(n,q)$ considered in \S\ref{sss binomial}. As we will 
	see, its K\"{a}hlerification is an open dense subset of the two dimensional sphere $S^{2}\,,$ and, as for $\mathcal{P}_{n}^{\times}\,,$ it 
	needs to be completed in order to recover a satisfactory description of the spin. 

	Despite these technical difficulties and ambiguities, we shall use the above definitions as a basis for our physical investigations, 
	and adapt them in an obvious way when a natural ``completion" exists. As we will see, this already leads 
	to interesting physical results.

\section{Binomial distribution and the spin of a particle}
	\label{sss binomial}
	Let $\mathcal{B}(n,q)$ be the space of binomial distributions defined over $\Omega:=\{0,...,n\}\,.$ By definition, 
	an element $p\in \mathcal{B}(n,q)$ is characterized by a real parameter $q\in\,\, ]0,1[$ verifying, for
	$k\in \Omega\,,$ $p(k)=\binom{n}{k}q^{k}(1-q)^{n-k}\,,$ where $\binom{n}{k}=\frac{n!}{(n-k)!k!}\,.$ The 
	set of binomial distributions forms a $1$-dimensional statistical manifold parameterized by $q$ and is easily seen 
	to be an exponential family, for one may write 
	\begin{eqnarray}\label{eee diner Peter bientot!!!}
		p(k)=\binom{n}{k}q^{k}(1-q)^{n-k}=
		\exp\bigg\{\ln{\binom{n}{k}}+k\theta-n\ln{\big(1+\exp{\theta}\big)}\bigg\}\,,
	\end{eqnarray} 
	where $\theta:=\ln(\frac{q}{1-q})\,.$ In particular, setting $C(k):=\ln{\binom{n}{k}}\,,$ $F(k):=k$ and $\psi(\theta):=
	n\ln{(1+\exp{\theta})}\,,$ one has $p(k)=\exp\{C(k)+\theta\cdot F(k)-\psi{(\theta)}\}\,.$\\

	Let $S^{2}:=\{(x,y,z)\in \mathbb{R}^{3}\,\vert\,x^{2}+y^{2}+z^{2}=1\}$ be the unit sphere endowed 
	with its natural K\"{a}hler structure $(g_{S^{2}},J_{S^{2}},\omega_{S^{2}})$ and let us write  
	$(S^{2})^{\times}:=S^{2}-\{(1,0,0),(-1,0,0)\}\,.$

\begin{proposition}\label{ppp kahlerification binomiale}
	If $S^{2}$ is endowed with the K\"{a}hler structure $(n\cdot g_{S^{2}},J_{S^{2}},\,n\cdot 
	\omega_{S^{2}})$ (i.e. its natural K\"{a}hler structure is multiplied
	by $n$)\,, then
	\begin{eqnarray}\label{eee kahlerification de bino}
		\mathcal{B}(n,q)^{\mathbb{C}}\cong (S^{2})^{\times}\,,
	\end{eqnarray}
	and in term of this identification, the map $\pi_{\mathcal{B}(n,q)}\,:\, \mathcal{B}(n,q)^{\mathbb{C}}\rightarrow \mathcal{B}(n,q)$ 
	becomes
	\begin{eqnarray}\label{eee pi pour bino}
		\pi_{\mathcal{B}(n,q)}\,:\,(S^{2})^{\times}\rightarrow\mathcal{B}(n,q)\,,\,\,\,\,\,
		\pi_{\mathcal{B}(n,q)}(x,y,z)(k)=\dfrac{1}{2^{n}}\binom{n}{k}(1+x)^{k}(1-x)^{n-k}\,.
	\end{eqnarray}
\end{proposition}
	Proposition \ref{ppp kahlerification binomiale} follows from direct computations. Indeed, 
	in term of the natural parameter $\theta\in \mathbb{R}\,,$ the Fisher metric $h_{F}$ on $\mathcal{B}(n,q)$ is 
	(see \S\ref{section information geometry})
	\begin{eqnarray}
		h_{F}(\theta)=\frac{n\exp{\theta}}{(1+\exp{\theta})^{2}}
	\end{eqnarray}
	from which one easily sees that the canonical structure $(g,J,\omega)$ of $T\mathcal{B}(n,q)$ is, using the identification
	$T\mathcal{B}(n,q)\cong \mathbb{R}^{2}\,,\dot{\theta}\partial_{\theta}\mapsto (\theta,\dot{\theta})$ as 
	well as \eqref{equation definition G, omega, etc.},
	\begin{eqnarray}
		g(\theta,\dot{\theta})=\dfrac{n\exp{\theta}}{(1+\exp{\theta})^{2}}\cdot
		\begin{pmatrix}
				1&0\\
                 	        0&1
		\end{pmatrix}
	\,,\,\,\,\,\,
		J=\begin{pmatrix}
				0&-1\\
                 	        1&0
		\end{pmatrix}
	\,,\,\,\,\,\,
	\omega(\theta,\dot{\theta})=\dfrac{n\exp{\theta}}{(1+\exp{\theta})^{2}}\cdot
		\begin{pmatrix}
				0&1\\
                 	        -1&0
		\end{pmatrix}\,.
	\end{eqnarray}
	A basis for $\mathscr{K}\big(T\mathcal{B}(n,q)\big)$ is easily seen to be 
	\begin{eqnarray}
		1,\,\,\,\,\tanh(\theta/2)\,,\,\,\,\,\dfrac{\cos(\dot{\theta}/2)}{\cosh{(\theta/2)}}\,,\,\,\,\,
		\dfrac{\sin(\dot{\theta}/2)}{\cosh{(\theta/2)}}\,,
	\end{eqnarray}
	where $\tanh{(x)}=\frac{\exp{(x)}-\exp{(-x)}}{\exp{(x)}+\exp{(-x)}}$ and $\cosh{(x)}=\frac{\exp{(x)}+\exp{(-x)}}{2}\,.$
	
	As a Lie algebra, the space $\mathscr{K}\big(T\mathcal{B}(n,q)\big)\,,$ endowed with the natural Poisson bracket associated 
	to $\omega\,,$ is isomorphic to the Lie algebra $\mathfrak{u}(2)$ of the group of unitary matrices $U(2)$ via the isomorphism 
	\begin{eqnarray}
		&&1\mapsto \dfrac{1}{2n}
			\begin{pmatrix}
					i&0\\
                	 	        0&i
			\end{pmatrix}\,,\,\,\,
		\tanh(\theta/2)\mapsto\dfrac{1}{2n}
			\begin{pmatrix}
					i&0\\
                	 	        0&-i
			\end{pmatrix}\,,\,\,\,\nonumber\\
		&&\dfrac{\cos(\dot{\theta}/2)}{\cosh{(\theta/2)}}\mapsto\dfrac{1}{2n}
			\begin{pmatrix}
					0&-1\\
                	 	        1&0
			\end{pmatrix}\,,\,\,\,
		\dfrac{\sin(\dot{\theta}/2)}{\cosh{(\theta/2)}}\mapsto\dfrac{1}{2n}
			\begin{pmatrix}
					0&i\\
                	 	        i&0
			\end{pmatrix}\,.
	\end{eqnarray}

	Clearly (see \eqref{eee definition gamma}), 
	the group $\Gamma\big(\mathcal{B}(n,q)\big)$ is isomorphic to $\mathbb{Z}\,,$ its natural action on 
	$T\mathcal{B}(n,q)$ being $k\cdot (\theta,\dot{\theta})=(\theta,\dot{\theta}+4k\pi)\,,$ which is obviously free 
	and proper, and the quotient $T\mathcal{B}(n,q)/\Gamma\big(\mathcal{B}(n,q)\big)$ is diffeomorphic 
	to $(S^{2})^{\times}$ via the 
	map
	\begin{eqnarray}\label{eee defini je suis starbucks}
		[(\theta,\dot{\theta})]\mapsto \Big(\tanh(\theta/2)\,,\dfrac{\cos(\dot{\theta}/2)}{\cosh{(\theta/2)}}\,,
		\dfrac{\sin(\dot{\theta}/2)}{\cosh{(\theta/2)}}\Big)\,,
	\end{eqnarray}
	where $[(\theta,\dot{\theta})]:=\mathbb{Z}\cdot (\theta,\dot{\theta})=
	\big\{(\theta,\dot{\theta}+4k\pi)\in \mathbb{R}^{2}\,\vert\,k\in \mathbb{Z}\big\}\,.$

	A direct calculation shows that if the canonical K\"{a}hler structure of $(S^{2})^{\times}$ is weighted by $n\,,$ then 
	\eqref{eee defini je suis starbucks} defines a map which is an isomorphism of K\"{a}hler manifolds, 
	whence Proposition \ref{ppp kahlerification binomiale}.\\

	%As for $(\mathcal{P}_{n}^{\times})^{\mathbb{C}}\rightarrow \mathcal{P}_{n}^{\times}\,,$ 
	The canonical projection $\pi_{\mathcal{B}(n,q)}\,:\,(S^{2})^{\times}\cong\mathcal{B}(n,q)^{\mathbb{C}}\rightarrow \mathcal{B}(n,q)$ 
	can be naturally extended to the whole sphere $S^{2}$ 
	provided we adjoint two elements to $\mathcal{B}(n,q)\,,$ namely the Dirac measures $\delta_{0}$ and $\delta_{n}$ 
	defined, for $k\in \{0,...,n\}\,,$ by 
	\begin{eqnarray}
		\delta_{0}(k)=
			\left\lbrace 
			\begin{array}{cc}
				1\,\,\,\textup{if}\,\,\, k=0\\
				0\,\,\,\textup{if}\,\,\, k\neq 0
			\end{array}
			\right.
		\,,\,\,\,\,\,\,\,\,\delta_{n}(k)=
			\left\lbrace 
			\begin{array}{cc}
				1\,\,\,\textup{if}\,\,\, k=n\\
				0\,\,\,\textup{if}\,\,\, k\neq n
			\end{array}
			\right.
			\,.
	\end{eqnarray}
	Let us denote $\overline{\mathcal{B}}(n,q):=\mathcal{B}(n,q)\cup \{\delta_{0},\delta_{n}\}$ (disjoint union). 
	Clearly, the map $\pi_{\mathcal{B}(n,q)}\,:\,(S^{2})^{\times}\rightarrow \mathcal{B}(n,q)$ extends uniquely as a continuous map 
	$\pi_{\overline{\mathcal{B}}(n,q)}\,:\,S^{2}\rightarrow \overline{\mathcal{B}}(n,q)\,,$ with 
	$\pi_{\overline{\mathcal{B}}(n,q)}(-1,0,0):=\delta_{0}\,,\,\,\,\pi_{\overline{\mathcal{B}}(n,q)}(1,0,0):=\delta_{n}\,,$ 
	making the following diagram commutative :
	\begin{eqnarray}
		\xymatrix{
   			(S^{2})^{\times} \ar[r]^{\pi_{{\mathcal{B}}(n,q)}} \ar[d]_i & 
				\mathcal{B}(n,q) \ar[d]^j \\
   			S^{2}\ar[r]_{{\pi_{\overline{\mathcal{B}}(n,q)}}} & \overline{\mathcal{B}}(n,q)
  				}	
	\end{eqnarray}
	($i,j$ being inclusions). Observe that $\overline{\mathcal{B}}(n,q)$ is naturally a topological space since it is 
	included in $\mathcal{P}_{n+1}\,.$
		
	In this situation, one may show an analogue of Proposition \ref{lll nouvel methodologie!!!}, 
	and so, in the sequel we shall mainly focus on $\pi_{\overline{\mathcal{B}}(n,q)}$ instead of $\pi_{{\mathcal{B}}(n,q)}\,,$ 
	and heuristically we shall write 
	$\overline{\mathcal{B}}(n,q)^{\mathbb{C}}\cong S^{2}\,.$ The space $S^{2}$ will be thought of as the 
	\textit{formal} K\"{a}hlerification of $\overline{\mathcal{B}}(n,q)^{\mathbb{C}}\,.$\\

	The space of K\"{a}hler functions on $S^{2}$ is easily seen to be $\mathscr{K}(S^{2})=\textup{Vect}_{\mathbb{R}}\{1,x,y,z\}\,.$
	From the K\"{a}hlerification point of view however, it is more natural to give a description of 
	$\mathscr{K}(S^{2})$ similar to that of Proposition \ref{ppp decomposition spectrale geometrique statistique}. To this end, 
	observe that the group of holomorphic isometries of $S^{2}$ is $SO(3)$ and that the space of 
	functions $\mathcal{A}_{\mathcal{B}(n,q)}$ associated to $\mathcal{B}(n,q)$ 
	(see \eqref{eee definition les observables...} for the general definition of $\mathcal{A}_{\mathcal{E}}$) 
	is generated by the constant function 1 and the identity function $k\,:\,\Omega\rightarrow \Omega\,,$ 
	i.e. $X\in \mathcal{A}_{\mathcal{B}(n,q)}$ if and only if $X(k)=\alpha+\beta\cdot k$ for some $\alpha, \beta\in \mathbb{R}\,.$

\begin{proposition}\label{ppp decomposition Kah functions on S2}
	Let $f\,:\,S^{2}\rightarrow \mathbb{R}$ be a smooth function. Then $f$ is a K\"{a}hler function if and only if there exist 
	$X\in \mathcal{A}_{\mathcal{B}(n,q)}$ and $\phi\in SO(3)$ such that $f$ can be written
	\begin{eqnarray}\label{eee decompo de f}
		f(x,y,z)=\int_{\Omega}\,X(k)\big[(\pi_{\overline{\mathcal{B}}(n,q)}\circ \phi)(x,y,z)\big](k)dk\,, \,\,\,\,\,\,\,
		\big((x,y,z)\in S^{2}\big)
	\end{eqnarray}
	where $\pi_{\overline{\mathcal{B}}(n,q)}\,:\,S^{2}\rightarrow \overline{\mathcal{B}}(n,q)$ is the canonical 
	projection coming from the K\"{a}hlerification
	of $\mathcal{B}(n,q)\,.$ 
\end{proposition}
\begin{proof}
	Let $X(k)=\alpha+\beta\cdot k \in \mathcal{A}_{{\mathcal{B}}(n,q)}$ be arbitrary and 
	let $\phi\in SO(3)$ be an isometry such that, using a matrix representation, 
	\begin{eqnarray}\label{eee me sens un peu lourd}
		\phi=\begin{pmatrix}
					a&b&c\\
               	 	        *&*&*\\
					*&*&*
			\end{pmatrix}
	\end{eqnarray}
	(in particular, the real numbers $a,b,c$ satisfy $\|(a,b,c)\|=1\,,$ i.e. $a^{2}+b^{2}+c^{2}=1)\,.$

	A simple calculation shows that for $(x,y,z)\in S^{2}\,,$
	\begin{eqnarray}
		\int_{\Omega}\,X(k)\big[(\pi_{\overline{\mathcal{B}}(n,q)}\circ 
		\phi)(x,y,z)\big](k)dk=\alpha+\frac{\beta n}{2}+\frac{\beta n}{2}(ax+by+cz)
	\end{eqnarray}
	which is a K\"{a}hler function on $S^{2}$ since it is a linear combination of $1,x,y,z\,.$ 

	Reciprocally, if $u_{0},u,v,w\in \mathbb{R}$ with $(u,v,w)\neq 0\,,$ then the equation 
	$\int_{\Omega}\,X(k)\big[(\pi_{\overline{\mathcal{B}}(n,q)}\circ \phi)(x,y,z)\big](k)dk=u_{0}+ux+vy+wz\,,$ with unknowns 
	$\alpha,\beta,a,b,c\,,$ admits as a solution
	\begin{eqnarray}\label{eee bar a la lune}
		\alpha=u_{0}\pm\|(u,v,w)\|\,,\,\,\,\,\beta=\mp \dfrac{2}{n}\|(u,v,w)\|\,,\,\,\,\,
		(a,b,c)=\mp \dfrac{1}{\|(u,v,w)\|}(u,v,w)\,,
	\end{eqnarray}
	where ``$\pm=-$" if $\beta>0$ and ``$\pm=+$" if $\beta<0\,,$ and where $\|\,.\,\|$ is the Euclidean norm. 
	If $(u,v,w)=0\,,$ then a solution is given 
	by $\alpha=u_{0}$ and $\beta=0$ ($\phi$ being arbitrary). The proposition follows.
\end{proof}
	Following our discussion at the end of \S\ref{section quantum mechanics}, we want to define the spectrum $\textup{spec}(f)$ 
	of a K\"{a}hler function $f(x,y,z)=\int_{\Omega}\,X(k)\big[(\pi_{\overline{\mathcal{B}}(n,q)}\circ \phi)(x,y,z)\big](k)dk$ on $S^{2}$ as 
	$\textup{Im}(X)\,,$ and its associated probability on $\textup{spec}(f)$ as 
	$P_{f,(x,y,z)}(A)=\int_{X^{-1}(A)}\,\big[(\pi_{\overline{\mathcal{B}}(n,q)}\circ 
	\phi)(x,y,z)\big](k)dk\,.$
%	\footnote{In this formula, it is understood that 
%	if $\phi(x,y,z)=(a,0,0)$ with $a\in \{-1,1\}\,,$ 
%	then $\int_{X^{-1}(A)}\,\big[(\pi_{\overline{\mathcal{B}}(n,q)}\circ \phi)(x,y,z)\big](k)dk$ equals $1$ if $a\in X^{-1}(A)$ and $0$ if 
%	$a\notin X^{-1}(A)\,.$} (here $A\subseteq \textup{spec}(f)$). 
	For this to be consistent, we need to check that these formulas are independent of the decomposition of $f$ given in Proposition 
	\ref{ppp decomposition Kah functions on S2}.
%	Let us now look at the ``spectral" properties of 
%	K\"{a}hler functions on the sphere $S^{2}\,,$ the latter space being 
%	viewed as the (formal) K\"{a}hlerification of $\overline{\mathcal{B}}(n,q)\,.$ 

\begin{proposition}\label{ppp je suis avec gerry schwarz!!}
	Let $f(x,y,z)=u_{0}+ux+vy+wz$ be a K\"{a}hler function on $S^{2}\,.$ Then the 
	spectrum $\textup{spec}(f)$ and the probability $P_{f,(x,y,z)}$ are well defined, and we have 
	$\textup{spec}(f)=\{\lambda_{0},...,\lambda_{n}\}\,,$ where 
	\begin{eqnarray}\label{eee encore bar a la lune}
		\lambda_{k}=u_{0}+\dfrac{2}{n}\,\|(u,v,w)\|\cdot \Big(-\dfrac{n}{2}+k\Big)
	\end{eqnarray}
	($k=0,...,n\,;$ $\|\,.\,\|$ Euclidean norm), and if $(u,v,w)\neq 0\,,$ then
	\begin{eqnarray}
		P_{f,(x,y,z)}(\lambda_{k})=
				\dfrac{1}{2^{n}}\displaystyle\binom{n}{k}\Big(1+\dfrac{ux+vy+wz}{\|(u,v,w)\|}\Big)^{k}
				\Big(1-\dfrac{ux+vy+wz}{\|(u,v,w)\|}\Big)^{n-k}\,.
	\end{eqnarray}
%	\begin{eqnarray}\label{eee encore bar a la lune}
%		&\bullet& \textup{spec}(f)=
%			\Big\{u_{0}+\dfrac{2}{n}\,\|(u,v,w)\|\cdot (-\dfrac{n}{2}+k)\in \mathbb{R}\,\Big\vert\,k=0,...,n\Big\}\,,
%			\nonumber\\
%		&\bullet& P_{f,(x,y,z)}\Big(u_{0}+\dfrac{2}{n}\,\|(u,v,w)\|\cdot (-\dfrac{n}{2}+k)\Big)=\nonumber\\
%		&&\,\,\,\,
%			\left\lbrace
%			\begin{array}{cc}
%				\dfrac{1}{2^{n}}\displaystyle\binom{n}{k}\Big(1+\dfrac{ux+vy+wz}{\|(u,v,w)\|}\Big)^{k}
%				\Big(1-\dfrac{ux+vy+wz}{\|(u,v,w)\|}\Big)^{n-k} & \textup{if}\,\,\,\|(u,v,w)\|\neq 0\,,
%				\label{eee clip etrange}\\
%				1 & \textup{if}\,\,\,\|(u,v,w)\|= 0\,,
%			\end{array}
%			\right.
%	\end{eqnarray}
\end{proposition}
\begin{corollary}
	Let $(u,v,w)\in \mathbb{R}^{3}$ be a vector whose Euclidean norm is $j:=n/2\,,$ and let 
	$f\,:\,S^{2}\rightarrow \mathbb{R}$ be the K\"{a}hler function defined by $f(x,y,z):=ux+vy+wz\,.$ Then,
	\begin{eqnarray}\label{eee j'ai trop tres fain}
		&\bullet& \textup{spec}(f)=
			\big\{-j,-j+1,...,j-1,j\big\}\,,
			\nonumber\\
		&\bullet& P_{f,(x,y,z)}\big(-j+k\big)=\label{eee probab dans stern gerlach}
			\binom{n}{k}\Big(\cos^{2}(\theta/2)\Big)^{k}
			\Big(\sin^{2}(\theta/2)\Big)^{n-k}\,,
	\end{eqnarray}
	where $\theta$ is an angle satisfying $\frac{ux+vy+wz}{\|(u,v,w)\|}=\cos(\theta)\,.$
\end{corollary}
\begin{remark}
	As mentioned in the introduction, \eqref{eee j'ai trop tres fain} gives 
	the probability that a spin-$j$ particle
	entering a second Stern-Gerlach device with maximum spin state (see Footnote \ref{fff footnote})
	is deflected into the $(-j+k)$-th outgoing beam, where $\theta$ is the angle between the two magnetic fields produced by the
	two Stern-Gerlach devices (see for example \cite{Marchildon}). 
	We will see subsequently how to obtain the probabilities corresponding to an incoming particle 
	when the eigenvalue of its spin operator along the magnetic field of the first Stern-Gerlach 
	device is arbitrary. 
%	Recall that in the latter,
%	a beam of particles (for example silver atoms) is send through a 
%	magnetic field with an angle $\theta$ which splits, due to spin effects, 
%	into a finite number of deflected parts whose amount of deflection depends 
%	on $\theta\,.$ Equation \eqref{eee probab dans stern gerlach}, in this respect, gives 
%	the probability that an incoming particle is
%	deflected into the $(-j+k)$-th outgoing beam 
	 
\end{remark}
\begin{proof}[Proof of Proposition \ref{ppp je suis avec gerry schwarz!!}]
	Let $f(x,y,z)=\int_{\Omega}\,X(k)\big[(\pi_{\overline{\mathcal{B}}(n,q)}\circ \phi)(x,y,z)\big](k)dk$ be a 
	K\"{a}hler function on $S^{2}$ with $X(k)=\alpha+\beta\cdot k$ ($\alpha,\beta\in \mathbb{R}$) and $\phi$ 
	having a matrix representation as in \eqref{eee me sens un peu lourd}. 
	We have to show that $\textup{spec}(f):=\textup{Im}(\alpha+\beta\cdot k)=\{\alpha,\alpha+\beta,...,
	\alpha+\beta\cdot n\}$ is independent of the decomposition of $f\,.$ For this, we need to check that if $f$ can be written 
	$f(x,y,z)=\sum_{k=0}^{n}\,(\overline{\alpha}+\overline{\beta}\cdot k)\big[({\pi_{\overline{\mathcal{B}}(n,q)}}
	\circ \overline{\phi})(x,y,z)\big](k)$ 
	with different $\overline{\alpha},
	\overline{\beta}\in \mathbb{R}$ and a different $\overline{\phi}\in SO(3)$ (with different 
	$\overline{a},\overline{b},\overline{c}\in \mathbb{R}$), 
	then $\textup{Im}(\alpha+\beta\cdot k)=\textup{Im}(\overline{\alpha}+\overline{\beta}\cdot k)\,.$ 
	To this end, observe that if $\sum_{k=0}^{n}\,(\alpha+\beta\cdot k)\big[(\pi_{\overline{\mathcal{B}}(n,q)}\circ \phi)(x,y,z)\big](k)=
	\sum_{k=0}^{n}\,(\overline{\alpha}+\overline{\beta}\cdot k)\big[(\pi_{\overline{\mathcal{B}}(n,q)}\circ \overline{\phi})(x,y,z)\big](k)$ 
	for all $(x,y,z)\in S^{2}\,,$ then 
	\begin{eqnarray}\label{eee double clef}
		\alpha+\dfrac{\beta n}{2} = \overline{\alpha}+
		\dfrac{\overline{\beta} n}{2}\,\,\,\,\,\textup{and}\,\,\,\,\,\dfrac{\beta n}{2}\cdot (a,b,c)
		=\dfrac{\overline{\beta} n}{2}\cdot (\overline{a},\overline{b},\overline{c})\,.
	\end{eqnarray}
	Taking into account the fact that $\|(a,b,c)\|=\|(\overline{a},\overline{b},\overline{c})\|=1\,,$ one immediately sees that 
	$\vert \beta\vert=\vert \overline{\beta}\vert\,,$ and we are led to the following three possibilities:
	\begin{eqnarray}
		&&\beta=\overline{\beta}=0,\,\alpha=\overline{\alpha}\,\,\,\,\,\textup{or}\,\,\,\,\,
		\beta \neq 0,\,\beta=\overline{\beta},\,\alpha=\overline{\alpha},\,(a,b,c)=
		\overline{a},\overline{b},\overline{c})\\
		&\textup{or}&\beta\neq 0,\,\beta=-\overline{\beta},\,\alpha=\overline{\alpha}-\beta\cdot n,\,(a,b,c)=
			-(\overline{a},\overline{b},\overline{c})\,.
	\end{eqnarray} 
	The only ambiguity is for the last case for which we have :
	\begin{eqnarray}\label{eee les doubles quand?}
		\textup{Im}(\overline{\alpha}+\overline{\beta}\cdot k)&=&\textup{Im}(\alpha+\beta\cdot n-\beta\cdot  k)=
		\{\alpha+\beta\cdot  n-\beta\cdot k\,\vert\,k=0,...,n\}\nonumber\\
		&=& \{\alpha+\beta\cdot  (n-k)\,\vert\,k=0,...,n\}=\{\alpha+\beta\cdot k\,\vert\,k=0,...,n\}\nonumber\\
		&=& \textup{Im}(\alpha+\beta\cdot k)\,.
	\end{eqnarray}
	Hence $\textup{Im}(\alpha+\beta\cdot k)=\textup{Im}(\overline{\alpha}+\overline{\beta}\cdot k)\,.$ It follows that 
	$\textup{spec}(f)$ is well defined. 

	In terms of $u_{0},u,v,w,$ and assuming $\beta>0$ for simplicity (the case 
	$\beta\leq 0$ leads to the same result),
	we have, using \eqref{eee bar a la lune}, 
	\begin{eqnarray}
		\textup{spec}(f)&=&\textup{Im}({\alpha}+{\beta}\cdot k)=\{\alpha+k\cdot \beta\,\vert\,k=0,...,n\}\nonumber\\
		&=&\Big\{u_{0}-\|(u,v,w)\|+k\cdot \dfrac{2}{n}\,\|(u,v,w)\|\,\Big\vert\,k=0,...,n\Big\}\nonumber\\
		&=&\Big\{u_{0}+\dfrac{2}{n}\,\|(u,v,w)\|\cdot (-\dfrac{n}{2}+k)\,\Big\vert\,k=0,...,n\Big\}\,.
	\end{eqnarray}
	This is exactly the first item in \eqref{eee encore bar a la lune}.

%	Let us now consider, for $(x,y,z)\in (S^{2})^{\times}\,,$ the probability 
%	$P_{f,(x,y,z)}\,.$ For this, let us write for $k\in \{0,...,n\}\,,$ $X(k):=\alpha+\beta\cdot k$ and let us 
%	assume for simplicity that $\beta>0\,;$ in particular $X^{-1}(\alpha+k\cdot \beta)=k\,.$ 
%	By definition, we have :
%	\begin{eqnarray}
%		P_{f,(x,y,z)}(\alpha+k\cdot \beta)&=&
%		\int_{X^{-1}(\alpha+k\cdot \beta)}\,\big[(\pi\circ \phi)(x,y,z)\big](k)\,dk
%		=\big[(\pi\circ \phi)(x,y,z)\big](k)\nonumber\\
%			&=&\dfrac{1}{2^{n}}\binom{n}{k}\Big(1+ax+by+cz\Big)^{k}
%			\Big(1-(ax+by+cz)\Big)^{n-k}\,. \label{eee a la la...}
%	\end{eqnarray}
%	Using the three items between \eqref{eee double clef} and 
%	\eqref{eee les doubles quand?}, it is not difficult to see that $P_{f,(x,y,z)}$ is a well defined 
%	probability on $\textup{spec}(f)$ which doesn't depend on a particular 
%	decomposition of $f(x,y,z)=\sum_{k=0}^{n}\,(\alpha+\beta\cdot k)\big[(\pi\circ \phi)(x,y,z)\big](k)\,,$ and that 
%	\eqref{eee clip etrange} is a direct consequence of \eqref{eee bar a la lune} together with \eqref{eee a la la...} . 
	Now, by similar arguments, one shows that $P_{f,(x,y,z)}$ is indeed well defined and yields a probability 
	on $\textup{spec}(f)\,.$ The proposition follows. 
\end{proof}

	As we already mentioned, not all the possibilities in the Stern-Gerlach experiment are exhausted with 
	\eqref{eee probab dans stern gerlach}. But of course, we would like to recover all these probabilities 
	following our ``statistical approach". This may be done as follows. 

%	The space $\mathcal{B}(n,q)$
%	is naturally embedded in $\mathcal{P}_{n+1}^{\times}\,,$ and as we will see, by ``K\"{a}hlerifying" this inclusion 
%	one arrives at an embedding $S^{2}\hookrightarrow \mathbb{P}(\mathbb{C}^{n+1})$ which allows 
%	to extend every K\"{a}hler function $f$ on $S^{2}$ to an unique K\"{a}hler function 
%	$\widehat{f}$ on $\mathbb{P}(\mathbb{C}^{n+1})\,,$
%	the latter function having the advantage to carry more information than the original one. 
%	Actually, and as we will see shortly, the resulting map 
%	$\mathscr{K}{(S^{2})}\rightarrow \mathscr{K}\big(\mathbb{P}(\mathbb{C}^{n+1})\big)$ is an irreducible 
%	unitary representation which is exactly the representation physicists use to describe the spin. \\

	Let $j\,:\,\mathcal{B}(n,q)\hookrightarrow \mathcal{P}_{n+1}^{\times}$ be the canonical inclusion.
	The composition of its derivative $j_{*}\,:\,T\mathcal{B}(n,q)\hookrightarrow T\mathcal{P}_{n+1}^{\times}$ 
	with the map $\tau\,:\,T\mathcal{P}_{n+1}^{\times}\rightarrow \mathbb{P}(\mathbb{C}^{n+1})^{\times}$ 
	considered in \S\ref{section the motivating example} yields 
	a map $T\mathcal{B}(n,q)\rightarrow \mathbb{P}(\mathbb{C}^{n+1})$ that we would like to describe. 
	To this end, recall that the elements of $\mathcal{B}(n,q)$ can be parameterized by the natural parameter 
	$\theta\in\mathbb{R}$ as 
	$p(k;\theta):=\textup{exp}\,\big\{\ln\binom{n}{k}+k\theta-n\ln\big(1+\text{exp}(\theta)\big)\big\}$
	(see \eqref{eee diner Peter bientot!!!}), and that $T\mathcal{B}(n,q)$ is identified with $\mathbb{R}^{2}$ via the 
	map $\dot{\theta}\partial_{\theta}\mapsto (\theta,\dot{\theta})\,.$
\begin{lemma}
	In therm of the natural parameter $\theta\,,$ the map $\tau\circ j_{*}\,:\,
	T\mathcal{B}(n,q)\rightarrow \mathbb{P}(\mathbb{C}^{n+1})$ 
	reads :
	\begin{eqnarray}
		(\tau\circ j_{*})(\theta,\dot{\theta})&=&\Big[\,p(0;\theta)^{1/2},\,
		p(1;\theta)^{1/2}\,e^{{i}\dot{\theta}/2},...,\, p(n;\theta)^{1/2}\,e^{{i}\dot{\theta}\cdot n/2}\,\Big]\,.
	\end{eqnarray}
	\end{lemma}
\begin{proof}
%	The map $\tau$ was defined in \eqref{eee definition tau} by means of the exponential representation $T_{p}\mathcal{P}_{n+1}^{\times}
%	\cong \{u\in \mathbb{R}^{n+1}\,\vert\,u_{1}p_{1}+...+u_{n+1}p_{n+1}=0\}$ (see \S\ref{section the motivating example}). So, we first need
%	to describe $j_{*}$ using the exponential representation. 
	Take a smooth curve $\theta(t)$ in $\mathbb{R}$ and set
	$\theta:=\theta(0)$ and $\dot{\theta}:=\frac{d}{dt}\big\vert_{0}\theta(t)\,.$ For $k\in \{0,...,n\}\,,$ we have : 
	\begin{eqnarray}
		\dfrac{d}{dt}\bigg\vert_{0}\,p\big(k;\theta(t)\big)&=&\dfrac{d}{dt}\bigg\vert_{0}\,\textup{exp}\bigg\{
		\ln\binom{n}{k}+k\theta(t)-n\ln\big(1+\textup{exp}(\theta(t)\big)\bigg\}\nonumber\\
		&=&\bigg[\dfrac{d}{dt}\bigg\vert_{0}\,\bigg(\ln\binom{n}{k}+k\theta(t)-n\ln\big(1+\textup{exp}(\theta(t))\big)
		\bigg)\bigg]\cdot p(k,\theta)
		\nonumber\\
		&=& \dot{\theta}\,\bigg(k-\dfrac{\textup{exp}(\theta)}{1+\textup{exp}(\theta)}\bigg)\cdot p(k,\theta)\,,
	\end{eqnarray}
	from which we see that $j_{*}(\theta,\dot{\theta})$ corresponds, 
	in the exponential representation of $T\mathcal{P}_{n+1}^{\times}$ (see \S\ref{section the motivating example}), to the vector 
	$[u(\theta)]_{j\big(p(.,\theta)\big)}\,,$ where $u(\theta)\in \mathbb{R}^{n+1}$ is defined, for $k\in \{0,...,n\}\,,$ by
	\begin{eqnarray} \label{equation trop mange}
		u(\theta)_{k}=
		\dot{\theta}\Big(k-\dfrac{\textup{exp}(\theta)}{1+\textup{exp}(\theta)}\Big)\,.
	\end{eqnarray}
	The lemma is now a simple consequence of \eqref{equation trop mange} together with the definition of $\tau$ 
	(see \eqref{eee definition tau}) and the homogeneity of the homogeneous coordinates of the complex 
	projective space $\mathbb{P}(\mathbb{C}^{n+1})\,.$
\end{proof}

	Recall that the group $\Gamma\big(\mathcal{B}(n,q)\big)$ is isomorphic $\mathbb{Z}$ and that its action on $T\mathcal{B}(n,q)$ is given 
	by $k\cdot (\theta,\dot{\theta})=(\theta,\dot{\theta}+4k\pi)\,.$ Clearly, $\tau\circ j_{*}$ is $\mathbb{Z}$-invariant,
	and since $T\mathcal{B}(n,q)/\Gamma\big(\mathcal{B}(n,q)\big)\cong (S^{2})^{\times}$ (see \eqref{eee defini je suis starbucks}), 
	we get a map $(S^{2})^{\times}\rightarrow \mathbb{P}(\mathbb{C}^{n+1})$ which can be conveniently described by the 
	following parametrization of the sphere,
	\begin{eqnarray}\label{eee mal a la tete}
		x=\cos(\alpha),\,\,\,\,\,y=\sin(\alpha)\cos(\beta),\,\,\,\,\,z=\sin(\alpha)\sin(\beta)\,,
	\end{eqnarray}
	where $\alpha\in [0,\pi]\,,$\,\,$\beta\in [0,2\pi]\,.$ With these parameters, the map 
	$(S^{2})^{\times}\rightarrow \mathbb{P}(\mathbb{C}^{n+1})$ reads
	\begin{eqnarray}
		(\alpha,\beta)\mapsto \big[\Psi(\alpha,\beta)\big]\,,
	\end{eqnarray}
	where $\Psi(\alpha,\beta)$ is the vector in $\mathbb{C}^{n+1}$ whose $k$th component is ($k=0,...,n$) :
	\begin{eqnarray}\label{eee fatigué}
		\Psi(\alpha,\beta)_{k}:=\binom{n}{k}^{1/2}
		\big(\cos(\alpha/2)\big)^{k}
		\big(\sin(\alpha/2)\big)^{n-k}\cdot e^{i\beta k}\,.
	\end{eqnarray}
%	\begin{eqnarray}
%		(\alpha,\beta)\mapsto \bigg[...,\,\binom{n}{k}^{1/2}\big(\cos(\alpha/2)\big)^{k}
%		\big(\sin(\alpha/2)\big)^{n-k}\cdot e^{i\beta k}\,,...\bigg]\,,\,\,\,\,\,\,\,\,(k=0,...,n)
%	\end{eqnarray}
	Observe that the map $\Psi$ is defined on the whole sphere, i.e. also for $\alpha=0$ and $\alpha=\pi$ (we agree that ``$0^{0}=1$"), 
	and that $\langle \Psi,\,\Psi\rangle=1\,,$ i.e. $\Psi$ is normalized 
	(here $\Psi:=\Psi(\alpha,\beta)$).

%	this yields a smooth map $S^{2}\hookrightarrow \mathbb{P}(\mathbb{C}^{n+1})$ which is an embedding. 
%	Now, since it is 
%	easier to work on $\mathbb{C}^{n+1}\,,$ we shall actually work with the following lift 
%	$\Psi\,:\,S^{2}\rightarrow \mathbb{C}^{n+1}-\{0\}$ whose $k$th coordinate ($k=0,...,n$) is :
%	\begin{eqnarray}\label{eee fatigué}
%		\Psi(\alpha,\beta)_{k}:=\binom{n}{k}^{1/2}
%		\big(\cos(\alpha/2)\big)^{k}
%		\big(\sin(\alpha/2)\big)^{n-k}\cdot e^{i\beta k}\,.
%	\end{eqnarray}
%	Observe that $\langle \Psi,\,\Psi\rangle=1\,,$ i.e. $\Psi$ is normalized 
%	(here $\Psi:=\Psi(\alpha,\beta)$).

	By construction, if $pr\,:\,\mathbb{C}^{n+1}{-}\{0\}\rightarrow \mathbb{P}(\mathbb{C}^{n+1})$ is the canonical 
	projection, then we have the following commutative diagram, 
	\begin{eqnarray}
		\xymatrix{
    			S^{2} \ar[r]^{\Psi} \ar[d]_{\pi_{\overline{B}(n,q)}} & \mathbb{C}^{n+1} \ar[d]^{r} \ar[r]^{pr} 
				&  \mathbb{P}(\mathbb{C}^{n+1})\ar[d]^{{\pi_{\mathcal{P}_{n+1}}}} \\
    			\overline{\mathcal{B}}(n,q) \ar[r]_{\overline{j}} & \mathcal{P}_{n+1}\ar[r]_{Id} &\mathcal{P}_{n+1}
  				}	
	\end{eqnarray}
	where $\overline{j}\,:\,\overline{\mathcal{B}}(n,q) \hookrightarrow  \mathcal{P}_{n+1}$ is the canonical injection, and 
	where $r\,:\,\mathbb{C}^{n+1}-\{0\}\rightarrow \mathcal{P}_{n+1}$ is defined by 
	$r(z_{1},...,z_{n+1})(k):=\frac{\vert z_{k}\vert^{2}}{\|z\|^{2}}\,,$ where $z=(z_{1},...,z_{n+1})\,.$\\
	
	Regarding $S^{2}$ as an embedded submanifold of $\mathbb{P}(\mathbb{C}^{n+1})$ via the map $pr\circ\Psi\,,$ 
	we have the following proposition.
\begin{proposition}\label{ppp dans un café au lait}
	Every K\"{a}hler function $f$ on $S^{2}$ extends uniquely as a K\"{a}hler function 
	$\widehat{f}$ on $\mathbb{P}(\mathbb{C}^{n+1})\,,$ and the resulting linear map 
	$\mathscr{K}(S^{2})
	\rightarrow \mathscr{K}\big(\mathbb{P}(\mathbb{C}^{n+1})\big)\,,\,\,f\mapsto \widehat{f}$ satisfies
	\begin{eqnarray}
		\widehat{\{f,g\}}=\dfrac{1}{4}\{\widehat{f},\widehat{g}\}
	\end{eqnarray}
	for all $f,g\in \mathscr{K}(S^{2})\,.$
\end{proposition}
	We will show Proposition \ref{ppp dans un café au lait} with a series of Lemmas.\\

	Let $f\,:\,S^{2}\rightarrow \mathbb{R}$ be a K\"{a}hler function on the sphere, that is, a linear combinaison of 
	$1,x,y,z\,:$ $f(x,y,z)=u_{0}+ux+vy+wz\,.$ 
	According to the characterization of K\"{a}hler functions on $\mathbb{P}(\mathbb{C}^{n+1})$ given in Lemma \ref{lemma comomentum map}, 
	the function $f$ possesses a K\"{a}hler extension on $\mathbb{P}(\mathbb{C}^{n+1})$ if and only if their exists 
	a Hermitian matrix $\mathbf{Q}(f)\in \textup{Herm}(\mathbb{C}^{n+1})$ such that, in terms of the parameters 
	$\alpha\in [0,\pi]$ and $\beta\in [0,2\pi]$ introduced in \eqref{eee mal a la tete},
	\begin{eqnarray}\label{eee je dois voir tilmann, il a pas l'air content...}
		u_{0}+u\cos(\alpha)+v\sin(\alpha)\cos(\beta)+w\sin(\alpha)\sin(\beta)
		=\big\langle \Psi(\alpha,\beta),\,\,\mathbf{Q}(f)\cdot \Psi(\alpha,\beta)\big\rangle
	\end{eqnarray}
	for all $\alpha\in [0,\pi]$ and all $\beta\in [0,2\pi]\,.$ 

\begin{lemma}\label{lll existence de QQQQ}
	Their exists a unique Hermitian matrix $\mathbf{Q}(f)$ such that \eqref{eee je dois voir tilmann, il a pas l'air content...} holds
	for all $\alpha\in [0,\pi]$ and all $\beta\in [0,2\pi]\,.$ It is explicitly given by 
	\begin{eqnarray}\label{eee incroyabale pas de label!!!}
		\mathbf{Q}(f)_{kk}=u_{0}-u\cdot\dfrac{2}{n}\Big(\dfrac{n}{2}-k\Big)\,,\,\,\,\mathbf{Q}(f)_{l,l+1}=
		\dfrac{1}{n}\sqrt{(n-l)(1+l)}\cdot (v-iw)\,,
	\,\,\,\mathbf{Q}(f)_{ab}=0\,,
	\end{eqnarray}
	where $k=0,...,n\,,$ $l=0,...,n-1\,,$ $a=0,...,n-2$ and where $b$ is such that $a+2\leq b\leq n\,.$
\end{lemma}
\begin{proof}
	For $a,b,k\in \{0,1,...,n\}$ and $\alpha\in [0,\pi]\,,$ set 
	\begin{description}
		\item[$\bullet$] $C_{ab}(\alpha):= \binom{n}{a}^{1/2}\binom{n}{b}^{1/2}\Big(\cos(\alpha/2)\Big)^{a+b}
			\Big(\sin(\alpha/2)\Big)^{2n-(a+b)}\cdot \mathbf{Q}(f)_{ab}\,,$
		\item[$\bullet$] $A_{k}(\alpha)
			:=\sum_{j=0}^{n-k}\,C_{j,k+j}(\alpha)\,.$
	\end{description}
	Observe that $\overline{C}_{ab}=C_{ba}$ (complex conjugate).
	By definition of $\Psi$ (see \eqref{eee fatigué}), we have
	\begin{eqnarray}
		&&\big\langle \Psi(\alpha,\beta),\,\,\mathbf{Q}(f)\cdot \Psi(\alpha,\beta)\big\rangle\nonumber\\
		&=&	\sum _{a,b=0}^{n}\,\binom{n}{a}^{1/2}\binom{n}{b}^{1/2}\Big(\cos(\alpha/2)\Big)^{a+b}\Big(\sin(\alpha/2)\Big)^{2n-(a+b)}
			e^{i\beta (b-a)}\mathbf{Q}(f)_{ab}\nonumber\\
		&=&\sum_{a,b=0}^{n}\,C_{ab}(\alpha)\,e^{i\beta(b-a)}
			=\sum_{k=0}^{n}\bigg(\sum_{a=0}^{n-k}\,C_{a,k+a}(\alpha)\bigg)\,e^{ik\beta}
			+\sum_{k=0}^{n}\bigg(\sum_{a=0}^{n-k}\,\overline{C}_{a,k+a}(\alpha)\bigg)\,e^{-ik\beta}\nonumber\\
		&=& \sum_{a=0}^{n}\,C_{aa}(\alpha)+\sum_{k=1}^{n}\bigg(A_{k}(\alpha)\,e^{ik\beta}+\overline{A}_{k}(\alpha)\,e^{-ik\beta}\bigg)
			\nonumber\\
		&=& A_{0}(\alpha)+2\,\sum_{k=1}^{n}\,\textup{Rel}\big(A_{k}(\alpha)\big)\cdot \cos(k\beta)-
			2\,\sum_{k=1}^{n}\,\textup{Im}\big(A_{k}(\alpha)\big)\cdot \sin(k\beta)\,,
	\end{eqnarray}
	and thus \eqref{eee je dois voir tilmann, il a pas l'air content...} may be rewritten :
	\begin{eqnarray}\label{eee manger ou pas?}
		&&u_{0}+u\cos(\alpha)+v\sin(\alpha)\cos(\beta)+w\sin(\alpha)\sin(\beta)\nonumber\\
		&=&A_{0}(\alpha)+2\,\sum_{k=1}^{n}\,\textup{Rel}\big(A_{k}(\alpha)\big)\cdot \cos(k\beta)-
		2\,\sum_{k=1}^{n}\,\textup{Im}\big(A_{k}(\alpha)\big)\cdot \sin(k\beta)\,.
	\end{eqnarray}
	Since the functions $\cos(k\beta)$ and $\sin(k'\beta)$ ($k=0,...,n$ and $k'=1,...,n$) are linearly independent, we 
	obtain:
	\begin{eqnarray}\label{eee boire un coup ce soir}
		&& A_{0}(\alpha)=u_{0}+u\cos(\alpha)\,,\,\,\,\,\,2\,\textup{Rel}(A_{1}(\alpha))=v\sin(\alpha)\,,\,\,\,\,\,
			2\,\textup{Rel}(A_{k}(\alpha))=0\,\,\,\textup{for all}\,\,\,k\geq 2\,,\nonumber\\	
		&\textup{and}& -2\,\textup{Im}(A_{1}(\alpha))=w\sin(\alpha)\,,\,\,\,\,\,
			-2\,\textup{Im}(A_{k}(\alpha))=0\,\,\,\textup{for all}\,\,\,k\geq 2\,.
	\end{eqnarray}
	From this set of equations, one sees that $\textup{Rel}(A_{k}(\alpha))=\textup{Im}(A_{k}(\alpha))=0$ for all $k\geq 2\,,$ i.e.
	$A_{k}(\alpha)=0$ for all $k\geq{2}\,.$ In view of the definition of $A_{k}(\alpha)\,,$ we thus have 
	\begin{eqnarray}
		\sum_{a=0}^{n-k}\,\binom{n}{a}^{1/2}\binom{n}{a+k}^{1/2}\Big(\cos(\alpha/2)\Big)^{2a+k}
		\Big(\sin(\alpha/2)\Big)^{2n-(2a+k)}\mathbf{Q}(f)_{a,k+a}=0\,.
	\end{eqnarray}
	It is not difficult to show that the functions $\cos(\alpha/2)^{k}\sin(\alpha/2)^{N-k}$ ($k=0,...,N\,,$ $N\in \mathbb{N}$) are linearly 
	independent, and thus for all $k\geq 2$ and for all $a$ such that $0\leq a\leq n-k\,,$ 
	\begin{eqnarray}
	\mathbf{Q}(f)_{a,k+a}=\mathbf{Q}(f)_{k+a,a}=0\,.
	\end{eqnarray}
	Hence, except for the three ``central diagonals", all entries of 
	$\mathbf{Q}(f)$ vanish ; this corresponds to the third equation in \eqref{eee incroyabale pas de label!!!}.

	Now there are still three equations in \eqref{eee boire un coup ce soir} we haven't used, namely $A_{0}(\alpha)=u_{0}+u\cos(\alpha)\,,$
	$2\,\textup{Rel}(A_{1}(\alpha))=v\sin(\alpha)$ and $-2\,\textup{Im}(A_{1}(\alpha))=w\sin(\alpha)\,.$ Using the definitions of 
	$A_{0}(\alpha)$ and $A_{1}(\alpha)\,,$  these equations reads
	\begin{eqnarray}\label{eee trop lourd ces saucisses!!}
			&\bullet& \displaystyle\sum_{a=0}^{n}\,\binom{n}{a}\Big(\cos(\alpha/2)\Big)^{2a}\Big(\sin(\alpha/2)\Big)^{2n-2a}
				\mathbf{Q}(f)_{aa}=u_{0}+u\cos(\alpha)\,,\\
			&\bullet& 2\,\displaystyle\sum_{a=0}^{n-1}\,\binom{n}{a}^{1/2}\binom{n}{a+1}^{1/2}\Big(\cos(\alpha/2)\Big)^{2a+1}
				\Big(\sin(\alpha/2)\Big)^{2n-(2a+1)}\,\mathbf{Q}(f)_{a,a+1}=(v-iw)\cdot\sin(\alpha)\,.\,\,\,\,\,\,\,\,
				\,\,\,\,\textbf{}\label{eee fait chaud ici, c'est bon}
	\end{eqnarray}
	Using the identity $\sin(\alpha)=2\sin(\alpha/2)\cos(\alpha/2)$ as well as 
	\begin{eqnarray}
		\binom{n}{a}^{1/2}\binom{n}{a{+}1}^{1/2}=\binom{n{-}1}{a}\dfrac{n}{\sqrt{(n{-}a)(a{+}1)}}\,;\,\,\,\,
		1=\sum_{a=0}^{n-1}\,\binom{n{-}1}{a}\Big(\cos^{2}(\alpha/2)\Big)^{a}\Big(\sin^{2}(\alpha/2)\Big)^{(n-1)-a}\,,
	\end{eqnarray}
	one rewrites \eqref{eee fait chaud ici, c'est bon} as
	\begin{eqnarray}\label{eee je me sens plein!!!!}
		&&\sum_{a=0}^{n-1}\,\binom{n-1}{a}\bigg[\dfrac{n}{\sqrt{(n-a)(a+1)}}\mathbf{Q}(f)_{a,a+1}-(v-iw)\bigg]\Big(\cos(\alpha/2)\Big)^{2a}
				\Big(\sin(\alpha/2)\Big)^{2n-2a-2}=0\,\,\,\,\,\textbf{}
%			\nonumber\\
%		&=&\sum_{a=0}^{n-1}\,\binom{n-1}{a}\bigg[(v-iw)\bigg]\Big(\cos(\alpha/2)\Big)^{2a}
%				\Big(\sin(\alpha/2)\Big)^{2n-2a-2}\,.
	\end{eqnarray}
	from which it follows that $\mathbf{Q}(f)_{a,a+1}=1/n\cdot\sqrt{(n-a)(a+1)}\,(v-iw)$ for all $a$ such that $0\leq a\leq n-1\,.$

	Finally, from \eqref{eee trop lourd ces saucisses!!} together with the identity 
	\begin{eqnarray}
		\sum_{a=0}^{n}\,\Big(\dfrac{n}{2}-a\Big)\binom{n}{a}
		\Big(\cos(\alpha/2)\Big)^{2a}\Big(\sin(\alpha/2)\Big)^{2n-2a}=-\dfrac{n}{2}\cdot \cos(\alpha)\,,
	\end{eqnarray}
	one easily obtains the first equation in \eqref{eee incroyabale pas de label!!!}. The lemma follows. 
\end{proof}
\begin{lemma}\label{lll tiny toons}
	The map $\mathbf{Q}\,:\,\mathscr{K}(S^{2})\rightarrow \textup{Herm}(\mathbb{C}^{n+1})$ satisfies 
	\begin{eqnarray}
		\mathbf{Q}(\{f,g\})=-\dfrac{i}{2}\big[\mathbf{Q}(f),\mathbf{Q}(g)\big]
	\end{eqnarray}
	for all $f,g\in \mathscr{K}(S^{2})\,.$
\end{lemma}
\begin{proof}
	By direct computations using \eqref{eee incroyabale pas de label!!!}.
\end{proof}
\begin{lemma}
	For $f,g\in \mathscr{K}(S^{2})\,,$ $\widehat{\{f,g\}}=\dfrac{1}{4}\{\widehat{f},\widehat{g}\}\,.$
\end{lemma}
\begin{proof}
	The lemma is a consequence of Lemma \ref{lll tiny toons} together with Lemma \ref{lemma comomentum map}. Indeed, 
	using the Lie algebra isomorphism $\mathfrak{u}(n{+}1)\rightarrow \mathscr{K}(\mathbb{P}(\mathbb{C}^{n+1}))\,,A\mapsto \xi^{A}$ 
	given in Lemma \ref{lemma comomentum map} and the fact that $\widehat{f}=\xi^{-2i\mathbf{Q}(f)}\,,$ we see that
	\begin{eqnarray}
		&&\widehat{\{f,g\}}=\xi^{-2i\mathbf{Q}(\{f,g\})}=\xi^{i^{2}[\mathbf{Q}(f),\mathbf{Q}(g)]}=
		\xi^{[i\mathbf{Q}(f),i\mathbf{Q}(g)]}=\big\{\xi^{i\mathbf{Q}(f)},\xi^{i\mathbf{Q}(g)}\big\}\nonumber\\
		&=& \big\{\xi^{-2i\mathbf{Q}(-f/2)},\xi^{-2i\mathbf{Q}(-g/2)}\big\}=
		\{-\widehat{f/2},-\widehat{g/2}\}=
		\dfrac{1}{4}\{\widehat{f},\widehat{g}\}\,.
	\end{eqnarray}
	The lemma follows. 
\end{proof}
	Proposition \ref{ppp dans un café au lait} follows from the last three lemmas. \\

	Since $\textup{Vect}_{\mathbb{R}}\{x,y,z\}\cong \mathfrak{su}(2)\,,$ the restriction of the 
	map $-1/2i\,\mathbf{Q}$ to $\textup{Vect}_{\mathbb{R}}\{x,y,z\}$ yields an unitary representation 
	$\mathfrak{su}(2)\rightarrow \mathfrak{u}(n+1)$ which is actually irreducible. 
	Hence, by considering the problem of extending K\"{a}hler functions on $S^{2}$ to $\mathbb{P}(\mathbb{C}^{n+1})\,,$ 
	we have been let to compute the irreducible unitary representations of the Lie algebra $\mathfrak{su}(2)\,,$ which is exactly 
	what physicists use to describe the spin of a particle. This means that we can recover all the probabilities 
	in the Stern-Gerlach experiment and that, \textit{in fine}, all information on the spin is encoded in 
	the binomial distribution $\mathcal{B}(n,q)\,.$

\section{Gaussians and the quantum harmonic oscillator}\label{section other examples} 
%	In this section, we discuss very briefly and without proof, an example of K\"{a}hlerification which is 
%	related to the quantum harmonic oscillator. \\
%	
%	Recall from Example \ref{eee normal with fixed...} that 
	Let $\mathcal{N}(\mu,1)$ be the set 
	of all probability density functions defined over $\Omega:=\mathbb{R}$ by 
	\begin{eqnarray}
		p(\xi;\mu):=\dfrac{1}{\sqrt{2\pi}}\,\textup{exp}\bigg\{-\dfrac{(\mu-\xi)^{2}}{2}\bigg\}\,,
	\end{eqnarray}
	where $\xi\in \Omega$ and $\mu\in \mathbb{R}\,.$ 

	Since $p(\xi;\mu)=\textup{exp}\big\{C(\xi)+\theta\cdot F(\xi)-\psi(\theta)\big\}$ with $C(\xi)=-1/2\cdot \xi^{2}\,,$ $\theta=\mu\,,$ 
	$F(\xi)=\xi$ 
	and $\psi(\theta):=1/2\cdot \theta^{2}+\ln(\sqrt{2\pi})\,,$ $\mathcal{N}(\mu,1)$ is an exponential family 
	whose natural parameter is $\theta=\mu\,.$ 

	The Fisher metric is easily seen to be constant $h_{F}(\theta)\equiv 1\,,$ which implies that 
	$T\mathcal{N}(\mu,1)\cong \mathbb{C}$ (isomorphism of K\"{a}hler manifolds). Moreover, it is not difficult 
	to see that the space $\mathscr{K}(\mathbb{C})$ of K\"{a}hler functions on $\mathbb{C}$ is 
	\begin{eqnarray}
	\mathscr{K}(\mathbb{C})=\textup{Vect}_{\mathbb{R}}\Big\{1,\,\,x,\,\,y,\,\,\dfrac{x^{2}+y^{2}}{2}\Big\}
	\end{eqnarray}
	 (here $x$ and $y$ are respectively the real and imaginary parts of $z\in \mathbb{C}$), with the following 
	commutators 
	\begin{eqnarray} 
		\{1,\,.\,\}=0\,,\,\,\,\{x,y\}=1\,,\,\,\,\Big\{x,\dfrac{x^{2}+y^{2}}{2}\Big\}=y\,,\,\,\,
		\Big\{y,\dfrac{x^{2}+y^{2}}{2}\Big\}=-x\,.
	\end{eqnarray}
	Clearly, $\Gamma(\mathcal{N}(\mu,1))$ is trivial. Hence, 
	$\mathcal{N}(\mu,1)^{\mathbb{C}}=T\mathcal{N}(\mu,1)/\{e\}\cong \mathbb{C}\,,$ i.e., 
	\begin{eqnarray}
		\mathcal{N}(\mu,1)^{\mathbb{C}}\cong\mathbb{C}\,.
	\end{eqnarray}
 	The canonical projection 
	$\pi_{\mathcal{N}(\mu,1)}\,:\,\mathbb{C}\rightarrow \mathcal{N}(\mu,1)$ is easily seen to be 
	\begin{eqnarray}
		\pi_{\mathcal{N}(\mu,1)}\,:\,\mathbb{C}\rightarrow \mathcal{N}(\mu,1)\,,\,\,\,\,
		\pi_{\mathcal{N}(\mu,1)}(z)(\xi)=\dfrac{1}{\sqrt{2\pi}}\,\exp\bigg\{-\dfrac{(x-\xi)^{2}}{2}\bigg\}\,,
	\end{eqnarray}
	where $z=x+iy\in \mathbb{C}\,.$ 

	For the spectral theory, observe that $\mathcal{A}_{\mathcal{N}(\mu,1)}=\textup{Vect}_{\mathbb{R}}\{1,\xi\}\,,$ where 
	$\xi\,:\,\mathbb{R}\rightarrow \mathbb{R}$ is the identify, and that the group of 
	holomorphic isometries of $\mathbb{C}$ is the group $E(2)=\mathbb{R}^{2}\rtimes O(2)$ of Euclidean isometries. 
	As a simple calculation shows, the only K\"{a}hler functions $f$ on $\mathbb{C}$ that can be written 
	as $f(z)=\int_{\mathbb{R}}\,X(\xi)\big[(\pi\circ \phi)(z)\big]\,d\xi\,,$ where $X\in \mathcal{A}_{\mathcal{N}(\mu,1)}$ 
	and $\phi\in E(2)\,,$
	are functions of the form $f(z)=u_{0}+ux+vy\,,$ where $u_{0},u,v\in \mathbb{R}\,.$ For such function, 
	the subset $\textup{spec}(f):=\textup{Im}(X)$ is well defined, i.e. independent of the decomposition of $f\,,$ and so is 
	its associated probability $P_{f,z}\,.$

	For $f(z)=u_{0}+ux+vy\,,$ calculations yield $\textup{spec}(f)=\mathbb{R}$ if $\textup{grad}(f) \not= 0\,,$ 
	$\textup{spec}(f)=\{u_{0}\}$ if $\textup{grad}(f) = 0$ and 
	\begin{eqnarray}
%		\textup{spec}(f)=
%		\left\lbrace
%		\begin{array}{ccc}
%			\mathbb{R} &\textup{if}& \textup{grad}(f) \not= 0\\
%			\{u_{0}\} &\textup{if}& \textup{grad}(f) = 0
%		\end{array}
%		\right.
%		\,;\,\,\,\,\,
		P_{f,z}=
		\left\lbrace
		\begin{array}{ccc}
			\dfrac{1}{\sqrt{2\pi}}\dfrac{1}{\|\textup{grad}(f)\|}\cdot 
			\textup{exp}\bigg\{-\dfrac{(\xi-f(z))^{2}}{2\|\textup{grad}(f)\|^{2}}\bigg\}
			&\textup{if}& \textup{grad}(f) \not=0\\
			\delta_{u_{0}} &\textup{if}& \textup{grad}(f)=0
		\end{array}
		\right.\,,
	\end{eqnarray}
	where $\textup{grad}(f)=(u,v)$ denotes the Riemannian gradient of $f$ and $\|\textup{grad}(f)\|=\sqrt{u^{2}+v^{2}}$ its Euclidean norm. \\

	Let us now relate the above formulas with the quantum harmonic oscillator. 
	Let $\hbar$ be a nonnegative real constant and let 
	$\Psi\,:\, \mathbb{C}\rightarrow C^{\infty}(\mathbb{R},\mathbb{C})$ be the function defined for $\xi\in \mathbb{R}$ 
	and $z=x+iy\in \mathbb{C}$ by
	\begin{eqnarray}
		\Psi(z)(\xi):=\dfrac{1}{(2\pi)^{1/4}}\,\textup{exp}\bigg\{-\dfrac{(\xi-x)^{2}}{4}\bigg\}\,
		\exp\bigg\{-\dfrac{i}{\hbar}\,y\,\xi\bigg\}\,.
	\end{eqnarray}
	Let us also define a linear map $\mathbf{Q}$ from the space $\mathscr{K}(\mathbb{C})$ to the space of unbounded operators 
	acting on $L^{2}(\mathbb{R},\mathbb{C})$ by 
	\begin{eqnarray}
	1\mapsto Id,\,\,\,\,\,\,\,\,\,x\mapsto x\,,\,\,\,\,\,\,\,\,\,y\mapsto i\hbar\frac{\partial}{\partial x}\,,
	\,\,\,\,\,\,\,\,\frac{x^{2}+y^{2}}{2}\mapsto 
	-\frac{\hbar^{2}}{2}\frac{\partial^{2}}{\partial x^{2}}
	+\frac{1}{2}x^{2}-\Big(\frac{\hbar^{2}}{8}+\frac{1}{2}\Big)\,.
	\end{eqnarray}
	Observe that $\mathbf{Q}$ is ``essentially" the operator which quantizes the classical harmonic oscillator.
\begin{proposition}
	For all $f\in \mathscr{K}(\mathbb{C})$ and for all $z\in \mathbb{C}\,,$ we have :
	\begin{eqnarray}\label{eee bientot la fin j'espere}
		f(z)=\big\langle \Psi(z),\,\mathbf{Q}(f)\cdot \Psi(z)\big\rangle\,,
	\end{eqnarray}
	where $\langle\,,\,\rangle$ is the usual $L^{2}$-scalar product on $L^{2}(\mathbb{R},\mathbb{C})\,.$
\end{proposition}
\begin{proof}
	By direct calculations.
\end{proof}
	Equation \eqref{eee bientot la fin j'espere} is the exact analogue of \eqref{eee je dois voir tilmann, il a pas l'air content...}, 
	but in an infinite dimensional context. Indeed, in \cite{Molitor-hydro} (work in progress), we regard the 
	space $\mathcal{D}:=\{\rho\,:\,M\rightarrow \mathbb{R}\,\vert\,\rho\,\,\textup{smooth}\,,\,\,
	\rho>0\,\,\,\textup{and}\,\,\,\int_{M}\,\rho(x)\cdot d\textup{vol}_{g}=1\}$ of smooth density probability 
	functions on a (compact) oriented Riemannian manifold $(M,g)$ as an infinite dimensional (Fr\'echet) manifold, 
	and we exhibit the analogues of the Fisher metric and 
	exponential connection on $\mathcal{D}\,,$ obtaining via Dombrowski's construction, an almost Hermitian structure on 
	$T\mathcal{D}$ which allows for an embedding $\mathbb{C}\hookrightarrow T\mathcal{D}\subseteq C^{\infty}(M,\mathbb{C})$ 
	similar to that of $S^{2}\hookrightarrow \mathbb{P}(\mathbb{C}^{n+1})$ in \S\ref{sss binomial}. 
	In this approach, the map $\Psi$ is obtained 
	by solving a simple partial differential equation related to a particular description of the tangent bundle of $\mathcal{D}$ 
	(see \cite{Molitor-hydro} for details).

\textbf{}\\\\

	\textbf{Acknowledgements.} It is a pleasure to thank Peter Heinzner, Shoshichi Kobayashi, Daniel Sternheimer, Hsiung Tze 
	and Tilmann Wurzbacher for many helpful discussions. I would like in particular to express my gratitude
	to Yoshiaki Maeda for having been such a great ``host professor" during my postdoctoral stay in Keio University ; I owe him a lot.

	This work was done with the financial support of the Japan Society for the Promotion of Science. 
%\begin{footnotesize}\bibliography{bibala}\end{footnotesize}

\end{document}